\documentclass{article}

\usepackage[margin=1 in]{geometry}
\usepackage{amsmath,amsthm,amsfonts}
\usepackage{amssymb}
\usepackage{bbm,breqn}
\usepackage{hyperref}
\usepackage{graphicx}
\usepackage{subcaption}
\usepackage{authblk}

\newtheorem{thm}{Theorem}[section]

\newtheorem{lemma}[thm]{Lemma}
\newtheorem{proposition}{Proposition}[section]
\newtheorem{remark}{Remark}[section]

\usepackage{color}

\newcommand{\df}{\textrm{d}}
\newcommand{\E}[2][n]{\mathbb{E}_{#1} \left[ {#2} \right]}
\newcommand{\N}{\mathcal{N}}
\newcommand{\ep}{\varepsilon}
\newcommand{\I}[1]{\mathbbm{1}_{\left \{ #1 \right \}} }
\newcommand{\bx}{\boldsymbol{x}}
\newcommand{\bw}{\boldsymbol{\df W}}

\title{Limiting Behaviors of High Dimensional Stochastic Spin Ensembles}

\date{}

\author[1]{Y. Gao} 
\author[2]{K. Kirkpatrick} 
\author[1]{J. Marzuola} 
\author[3]{J. Mattingly} 
\author[1]{K. Newhall} 

\affil[1]{Department of Mathematics, University of North Carolina at Chapel Hill}
\affil[2]{Department of Mathematics, University of Illinois at Urbana-Champaign}
\affil[3]{Department of Mathematics, Duke University}

\begin{document}

\maketitle

\begin{abstract}
               Lattice spin models in statistical physics are used to understand magnetism. Their Hamiltonians are a discrete form of a version of a Dirichlet energy, signifying a relationship to the Harmonic map heat flow equation. The Gibbs distribution, defined with this Hamiltonian, is used in the  Metropolis-Hastings (M-H) algorithm to generate dynamics tending towards an equilibrium state. In the limiting situation when the inverse temperature is large, we establish the relationship between the discrete M-H dynamics and the continuous Harmonic map heat flow associated with the Hamiltonian. We show the convergence of the M-H dynamics to the Harmonic map heat flow equation in two steps: First, with fixed lattice size and proper choice of proposal size in one M-H step, the M-H dynamics acts as gradient descent and will be shown to converge to a system of Langevin stochastic differential equations (SDE). Second, with proper scaling of the inverse temperature in the Gibbs distribution and taking the lattice size to infinity, it will be shown that this  SDE system  converges to the deterministic Harmonic map heat flow equation. Our results are not unexpected, but show remarkable connections between the M-H steps and the SDE Stratonovich formulation, as well as reveal trajectory-wise out of equilibrium dynamics to be related to a canonical PDE system with geometric constraints. 
\end{abstract}

\tableofcontents

\section{Introduction} \label{intro}

The Metropolis-Hastings (M-H) algorithm \cite{hastings1970monte} is widely used in particle statistics for model estimations \cite{newman1999monte,binder1993monte,landau2014guide,batrouni2004metastable,maccari2016numerical}.  It constructs a discrete-time Markov chain to sample a desired probability distribution by accepting or rejecting proposed states. For applications in statistical physics, it is often the Gibbs or canonical distribution that is to be sampled. In this case, the algorithm accepts all the proposed new states with lower energy and often rejects the proposals with higher energy.  Similar sampling can be achieved simulating a Langevin Stochastic differential equation (SDE) that performs gradient descent with noise; it too has the Gibbs distribution as its steady-state distribution.  This suggests that the Langevin SDE might be the optimal M-H algorithm in which all proposals are accepted.

For certain forms of probability distributions, the diffusion limit and therefore optimal scaling, of the random walk M-H algorithm has been obtained~\cite{roberts1997weak,breyer2000metropolis,mattingly2012diffusion}.  Specifically, for product measures in \cite{roberts1997weak} and the Gibbs distribution of a lattice model in \cite{breyer2000metropolis}, the weak convergence to Langevin diffusions has been shown by comparing generator functions. For non-product form measures the weak convergence to a stochastic partial differential equation was shown in \cite{mattingly2012diffusion}. 
These works consider the weak convergence only in equilibrium. Subsequent works \cite{jourdain2014optimal,jourdain2015optimal} consider scaling limits of out of equilibrium systems approaching equilibrium, but for product measures.   In this work, we fill a missing gap in the above mentioned works of trajectory-wise convergence, without assuming the system is in equilibrium, for non-product measures.

To address the question of trajectory-wise convergence, we study the XY and the classical Heisenberg lattice spin models \cite{stanley1968dependence} that play an important role in statistical physics to understand phase transitions and other phenomena including superconductivity  \cite{kosterlitz1973ordering,eley2012approaching}.  It is important to understand the limiting behavior of these models, including optimal scalings for simulations, and their critical properties.  For example, 
asymptotic results on the total spin of the mean-field XY and classical Heisenberg models have been studied by large deviation theory and Stein's method in \cite{kirkpatrick2013asymptotics,kirkpatrick2016asymptotics}. Numerically, Monte Carlo methods are used to verify analytical results about the XY model in \cite{maccari2016numerical,batrouni2004metastable} and the classical Heisenberg model in \cite{peczak1991high,chen1993static}.

The XY and classical Heisenberg models are defined on a periodic $d$-dimensional lattice $\mathbb{T}^d$ with $\delta x = \frac 1N$ the distance between adjacent vertices. Each spin sits at a lattice point and is described by a unit vector $\sigma_i: \mathbb{T}^d \to \mathbb{S}^m$, for $i=1\dots N$ where $m=1$ for the XY model and $m=2$ for the classical Heisenberg model.  We will focus primarily on the case of $\mathbb{T}^1 \to \mathbb{S}^2$, but continue the discussion for general $\mathbb{T}^d \to \mathbb{S}^m$ as the trajectory-wise convergence should follow similarly. The calculation from M-H to SDE should follow for higher dimension $T^d \to S^m$; SDE to PDE depends on the smooth solution of harmonic map heat flow equation, which we only have guaranteed for all time in the $T^2 \to S^2$ case using for instance the work \cite{guo1993landau}.  Here we assume that we are looking on a time scale for which the PDE has a smooth solution and focus on comparing to the microscopic dynamics.  It is an interesting topic for future work to study the nature of singularity formation in the microscopic and mesoscopic models.

The Hamiltonian of the system,  
\begin{equation} \label{eq:H}
  H = J \sum_{<i,j>} \| \sigma_i - \sigma_j \|^2,
\end{equation}
gives energy to misaligned neighboring spins where $<i,j>$ represents nearest neighbors and $J=N^{2-d}$ is a scaling factor. Denote $\sigma$ as the total spin configuration of $\sigma_i,i \in \mathbb{T}^d$, the M-H algorithm accepts/rejects based on the Gibbs distribution defined as
\begin{equation} \label{eq:P}
  \rho (\sigma) =  Z^{-1}\exp(-\beta H(\sigma)),
\end{equation}
where $\beta = (k_B T)^{-1}$ is the inverse temperature and $Z$ is the normalizing factor (aka partition function).

We will show that the M-H algorithm applied to the above lattice system, in the limit of small perturbations in the proposal, produces equivalent trajectories to the overdamped Langevin equation,
\begin{equation} \label{eq:Langevin}
 \df \sigma_i = \textrm{P}_{\sigma_i}^\perp (\Delta_N \sigma_i) \df t + \textrm{P}_{\sigma_i}^\perp \left( \sqrt{\frac N \beta} \df W_i \right),
\end{equation}
(interpreted in the Stratonovich form) where the $W_i$ are $(m+1)$-dimensional independent Brownian motions,
$
\Delta_N \sigma_i = - N^2 ( 2 \sigma_i - \sigma_{i+1} - \sigma_{i-1}), 
$  
is the discrete Laplacian and $\textrm{P}_x^\perp (y)$ for $\|x\|=1$ is the projection of $y$ onto the tangent plane of $x$. We find that the exact form of the projection does not matter, for example one could take either  $\textrm{P}_x^\perp (y) = y-(x \cdot y)x$ or $\textrm{P}_x^\perp (y) = x \times y$ when $m=2$.
The Stratonovich understanding of \eqref{eq:Langevin} is essential to keep the $\sigma_i$ as unit vectors, and for more on this equation see \cite{banas2014stochastic}. Our proof in section \ref{sec:SDE} naturally leads to the It\^o form of equation \eqref{eq:Langevin}, which includes an additional It\^o correction term of $-N\beta^{-1}\sigma_i \df t$.  This (overdamped) Langevin system \eqref{eq:Langevin} performs gradient descent on the energy defined by \eqref{eq:H} with the added constraint that $\sigma_i$ is confined to $\mathbb{S}^m$, $m=1,2$. In the case of $\mathbb{S}^2$ for the classical Heisenberg model, it is an SDE representation of the overdamped Landau-Lifshitz-Gilbert equation that has the Gibbs distribution as its invariant measure \cite{banas2014stochastic,kohn2005magnetic}.

Taking the number of lattice points, $N$, to infinity or equivalently the lattice spacing $\delta x = \frac 1N$ to zero, the limit of the deterministic part of \eqref{eq:Langevin} is the partial differential equation (PDE) called the harmonic map heat flow equation  
\begin{equation} \label{eq:harmonic}
 \partial_t \sigma = \textrm{P}_{\sigma}^\perp ( \Delta \sigma).
\end{equation}
In the $\mathbb{S}^2$ case, \eqref{eq:harmonic} is in the form of the overdamped Landau-Lifshitz equation \cite{guo2008landau}
\begin{equation}\label{eq:LLG}
 \partial_t \sigma = - \sigma \times (\sigma \times \Delta \sigma).
\end{equation}
In \cite{guo1993landau} this Landau-Lifshitz equation was shown to be equivalent to the Harmonic map heat flow from $\mathbb{T}^d \to \mathbb{S}^2$.  With the scaling $J=N^{2-d}$, the Hamiltonian in \eqref{eq:H} is the discrete form of the Dirichlet energy, $\int_\Omega | \nabla \sigma |^2 \df \Omega$, for this harmonic map heat flow.  This suggests that by decreasing the temperature, the out of equilibrium dynamics of the M-H algorithm converge to the deterministic flow of \eqref{eq:LLG} with large $N$ for the classical Heisenberg model.  We will show this equivalence by showing the convergence of the system of SDEs \eqref{eq:Langevin} to the PDE \eqref{eq:harmonic} in the limit of large $N$ with an appropriate scaling of the temperature to zero with $N$.  We point out that in order to obtain the finite temperature Stochastic PDE limit of the M-H dynamics in arbitrary dimension required a regularization of the noise.  We intend to pursue deriving a stochastic PDE limit of the M-H algorithm using colored noise in the proposal for future work.

While this current work does not focus on the dynamics of equation \eqref{eq:harmonic}, we point out that much work has been done on harmonic maps, the evolution of deterministic and stochastic harmonic map heat flows, as well as on describing the potential for singularity formation.  We cite for instance the now classical works of Eells and Sampson on Harmonic Maps \cite{eells1964harmonic} and of Struwe \cite{struwe1985evolution}, and the subsequent works of Chen-Struwe and Chen \cite{chen1989existence,chen1989weak}.  Rigidity and singularity formation was further understood in the works of Topping on the evolution of harmonic map heat flows and singularity formation \cite{topping1997rigidity,topping2002reverse,topping2004repulsion}, and the works of Lin-Wang \cite{lin1998energy,lin1999harmonic,lin2002harmonic}. We also note the book \cite{lin2008analysis} for a useful background on the subject.  The literature on the Harmonic Map Heat Flow is quite extensive and we do not suggest that the list here is complete.  

Though we do not establish the connection between the M-H dynamics and the Stochastic Harmonic Map Heat Flow in this work, these equations have also been studied recently.  We cite again the book \cite{banas2014stochastic} by Banas-Brzezniak-Neklyudov-Prohl that has many useful results in it about Stochastic ODE approximations and analysis of many aspects of the full Landau-Lifshitz-Gilbert stochastic PDE version of the full Landau-Lifshitz equation, which is a PDE similar to the harmonic map heat flow but including both dissipative and dispersive components of the flow.  We also cite more recent works on dynamics of stochastic Harmonic Map Heat Flow equations, including the works of Guo-Philipowski-Thalmaier \cite{guo2014stochastic}, Hocquet \cite{hocquet2018struwe,hocquet2019finite}, Chugreeva-Melcher \cite{chugreeva2018strong}.  For works on numerical discretization in a semi-discrete fashion of the stochastic Landau-Lifshitz equation, see the work of Alouges-De Bouard-Hocquet \cite{alouges2014semi}.

One method to obtain the deterministic limit of a stochastic system is to consider the hydrodynamic limit with relative entropy bounds \cite{guo1988nonlinear,yau1991relative,funaki1997motion}. Due to the geometric constraint in the XY and classical Heisenberg models, it is difficult to calculate the averages with respect to the Gibbs states as in \cite{guo1988nonlinear,yau1991relative,funaki1997motion} if the spin is expressed in Cartesian coordinates. One might try to use polar coordinates to do window averaging but the potential is not convex as in \cite{funaki1997motion}. Since the hydrodynamic limit for the XY and the classical Heisenberg models are not fully understood, we choose an alternative approach of taking inverse temperature $\beta$ to infinity along with particle number $N \to \infty$.    

One challenge in showing convergence of the spin models to diffusions is that the distribution \eqref{eq:P} is unaware of the confining geometry that the spins must remain in $\mathbb{S}^m$. Rather, it is included in the proposal step of the M-H algorithm.  Therefore, simply considering the equilibrium distribution is not enough to show equivalence, the proposal step must be taken into account for a trajectory-wise comparison between processes.  While always accepting the proposal step leads to each spin behaving independently like Brownian motion on the surface of $\mathbb{S}^m$, sampling a product measure, to consider the true M-H algorithm sampling \eqref{eq:P} we must therefore take into account the interdependence of the spins.  

Working directly with the proposal step, which includes a normalizing step, also includes challenges.  To linearize the nonlinear dynamics, we take the 
Taylor expansion of the M-H step and approximate it as a linear step.  The challenge here is that the coefficients of this expansion are random variables that can be arbitrarily large, therefore bounding the error is not trivial.    Also, the truncation of the proposal leads to a spin vector that does not stay on the sphere.   However, this displacement from the sphere is small, converging to zero as the size of the proposal tends towards zero.

Note that in the weak convergence result of M-H dynamics to diffusion processes \cite{roberts1997weak,breyer2000metropolis,mattingly2012diffusion}, the assumption of equilibrium is essential to bound the error terms. The result in this paper only assumes that the M-H dynamics (and thus the SDE system) start from a deterministic initial condition satisfying a certain regularity condition.  While the initial condition is assumed smooth, both the M-H and SDE dynamics immediately produce fluctuations, and the resulting trajectories are only close to the smooth deterministic PDE solution and not smooth themselves for all time.  Therefore, standard energy bounding techniques cannot be used.  To bound the error terms, we utilize scalings that are worse than those in the previously mentioned papers and are likely not optimal. We  use numerical simulations to explore how tight these bounds appear to be.

The remainder of the paper is as follows.  In Section \ref{sec:results} we present the main results in two parts.  First, we state the convergence of M-H dynamics to the SDE system \eqref{eq:Langevin} as the proposal size of M-H step goes to zero, then we state the convergence of the SDE system \eqref{eq:Langevin} to the deterministic PDE \eqref{eq:harmonic} as the lattice size goes to infinity and temperature to zero.  The key steps of the proofs are given in Sections \ref{sec:SDE} and \ref{sec:PDE} for the more complicated classical Heisenberg model from $\mathbb{T}^1 \to \mathbb{S}^2$ with details appearing in the Appendix. The proof for the XY model follows similarly.  For the M-H to SDE \eqref{eq:Langevin} proof in Section \ref{sec:SDE}, we apply a similar approach as in \cite{mattingly2012diffusion}, by first Taylor expanding the M-H step, keeping only the first three terms, then computing the required conditional expectations with respect to the Gaussian random variables to obtain the drift and diffusion terms of an Euler step for the diffusion process.  Then, the difference between the M-H and SDE dynamics in $L^2$ norm is bounded by a Gr\"onwall inequality.  For the SDE \eqref{eq:Langevin} to PDE \eqref{eq:harmonic} proof in Section \ref{sec:PDE}, we compare the SDE system with the finite difference approximation of the harmonic map heat flow equation \eqref{eq:harmonic}. The difference between the SDE and ODE system is governed by another diffusion process. We will rescale this process and show the rescaled error is bounded for a long time using stopping time.  These convergence results are compared to the convergence measured from the results of numerical simulations of the system in Section \ref{sec:numerical}.  Conclusions are presented in Section \ref{sec:conclusions}.

\begin{remark}
 We only show the case $\mathbb{T}^1 \to \mathbb{S}^2$. The calculation could be generalized for other cases of $\mathbb{T}^d \to \mathbb{S}^2$ quite similarly.
\end{remark}

%
\section{Main Results\label{sec:results}}
%


In this section we will explain how we apply the M-H algorithm to the XY and classical Heisenberg models, and state our main results.  Our first result is that the M-H dynamics is close to a stochastic Euler scheme for the SDE \eqref{eq:Langevin} in It\^o form. The bound on the error between the M-H dynamics and the SDE \eqref{eq:Langevin} is accomplished using arguments similar to the convergence of the stochastic Euler method. Our second result bounds the error between the SDE system and the finite difference approximation of the harmonic map heat flow equation~\eqref{eq:harmonic}.

Throughout the paper, we adopt the following notation.  We use the symbol $\sim$ when describing random variables.  For example, $z\sim \N(0,1)$ denotes that random variable $z$ is distributed with a normal distribution with mean 0 and variance 1.  For approximating deterministic functions, we use the notation $f =  O(\epsilon)$ as $\epsilon\to 0$ to indicate that $f$ is of lower order than $\epsilon$, specifically that $\lim_{\epsilon\to 0} | f/\epsilon | < \infty$.  When we approximate random variables, we use the notation $\approx$ as the difference between the variable and its approximation could be extremely large in a single realization, and therefore the notion of asymptotic approximation does not hold.  Therefore, when using the notation that random variable $z\approx O(\epsilon)$ we mean that the random variable has a deterministic prefactor that this $O(\epsilon)$. 

\subsection{Metropolis-Hastings step. \label{sec:MHsetup}}
Here, we explicitly state the M-H dynamics for the XY and classical Heisenberg models  with Hamiltonian given by \eqref{eq:H} for the case $d=1$.

Consider a set of spins evolving in time, $\sigma_i^n$ for particles $i=1 \dots N$ and time step $n\ge 0$, with time step size $\delta t$. To create the proposal, take the normal random vector
\[ w_i^n = \begin{pmatrix} z_1 \\ z_2  \end{pmatrix}, \quad \textrm{with }z_1,z_2 \sim \N(0,1) \] for the XY model and three-dimensional normal random vector
\[
w_i^n = \begin{pmatrix} z_1 \\ z_2 \\z_3 \end{pmatrix}, \quad \textrm{with }z_1,z_2,z_3 \sim \N(0,1),
\]
for the classical Heisenberg model.  Then project this noise vector onto the tangent plane of $\sigma_i^n$, forming the random vector 
\begin{equation}\label{eq:noise_projection}
\nu_i^n = P_{\sigma_i^n}^\perp(w_i^n) = w_i^n - (w_i^n,\sigma_i^n) \sigma_i^n.
\end{equation}
Since we are proving a trajectory-wise convergence result, we must imbed the M-H algorithm and the SDE dynamics in the same probability space. To this end, we define
\[
w_i^n \equiv \frac{W_i((n+1) \delta t) - W_i(n \delta t)}{\sqrt{\delta t}} ,
\]
where $W_i, 1 \le i \le N$ are the independent Brownian motions in \eqref{eq:Langevin}.

In the M-H algorithm, the intuitive idea for a proposal, $\tilde{\sigma}_i^n$, on a manifold at time-step $n$ is the exponential map
\begin{equation} \label{eq:expMap}
{\tilde{\sigma}_i^n} = \exp_{\sigma_i^n} (\varepsilon \nu_i^n) = \gamma_{\varepsilon \nu_i^n} (1),
\end{equation}
where $\gamma_{\varepsilon \nu_i^n}$ is the geodesic satisfying the nonlinear ODE $\nabla_{\dot \gamma} \dot \gamma = 0$ with $\nabla$ the affine connection on the manifold and $\varepsilon$ is the proposal size.  In practice, using the proposal 
\[
 \tilde{\sigma}_i^n =  \frac{\sigma_i^n + \varepsilon \nu_i^n}{\| \sigma_i^n + \varepsilon \nu_i^n \|}
\]
is computationally simpler and, as we will show, leads to the same convergence result.

The values $\sigma^n$ and $\tilde{\sigma}^n$ are used to denote the total spin configuration $\sigma_i^n, 1 \le i \le N$ at time step $n$ and the total proposal spin configuration $\tilde{\sigma}_i^n, 1 \le i \le N$. The proposal $\tilde{\sigma}^n$ is accepted with probability
\begin{equation} \label{eq:accept}
\alpha = 1 \wedge e^{-\beta\delta H},
\end{equation}
and rejected otherwise, where 
\begin{equation} \label{eq:dH}
\begin{aligned}
 \delta H = H(\tilde{\sigma}^n) - H(\sigma^n) = \sum_{j=1}^N \frac{\partial H}{\partial \sigma_j^n} \cdot (\tilde{\sigma}_j^n - \sigma_j^n) + 2J \sum_{j=1}^N (\tilde{\sigma}_j^n - \sigma_j^n) \cdot (\tilde{\sigma}_j^n - \sigma_j^n) \\
 - J \sum_{j=1}^N (\tilde{\sigma}_j^n - \sigma_j^n) \cdot (\tilde{\sigma}_{j+1}^n - \sigma_{j+1}^n +\tilde{\sigma}_{j-1}^n - \sigma_{j-1}^n )
 \end{aligned}
\end{equation}
is the difference between the Hamiltonian \eqref{eq:H} of the proposal $\tilde{\sigma}^n$ and of the current spin configuration $\sigma^n$. Then 
\[
 \sigma^{n+1} = \kappa_n \tilde{\sigma}^n + (1-\kappa_n) \sigma^n, \quad \kappa_n \sim \textrm{Bernoulli}(\alpha(\tilde{\sigma}^n,\sigma^n)).
\]

Repeating the proposal and accept/reject steps, we create a discrete Markov process at time steps $n+1,n+2,\ldots$ and we will show the convergence of the Markov chain to the solution to the Langevin SDE system \eqref{eq:Langevin}.

\begin{remark}
In fact, either choice of the following projection in equation \eqref{eq:noise_projection} gives us the same result for the classical Heisenberg model
\[
	 \textrm{P}_{\sigma_i^n}^\perp ( w_i^n) = \begin{cases}
	& \sigma_i^n \times w_i^n	\\
	& - \sigma_i^n \times ( \sigma_i^n \times w_i^n ) = w_i^n - \sigma_i^n (\sigma_i^n)^T w_i^n  \end{cases}
\]
as both lead to random walks on the sphere (see Appendix \ref{Appendix:diffusionsph}).
\end{remark}

\subsection{Convergence of Metropolis dynamics to SDE system.}
First we will show the convergence from the M-H dynamics to the Langevin SDE dynamics with a fixed number of particles $N$ as the proposal size $\varepsilon \to 0$. Intuitively, using the Taylor series truncation of the proposal, the approximation of one M-H step leads to an expression that looks like one Euler step for simulating the SDE \eqref{eq:Langevin} in It\^o form.

Let $\mathcal{F}_t$ denote the filtration generated by the set of Brownian motions $W_i(t)$, i=1\dots N, in \eqref{eq:Langevin} and Bernoulli random variables $\kappa_n$, $n=1\dots t/\delta t$.  We denote the conditional expectation $\E[]{\cdot | \mathcal{F}_{ t}}$ by $\E{\cdot}$.

The drift over one step of the Metropolis-Hastings algorithm for the $i$-th particle for small $\varepsilon$ is approximated by
\begin{equation}
\E{\sigma^{n+1}_i-\sigma^n_i} \approx -\frac 12 \beta \varepsilon^2  \textrm{P}_{\sigma^n_i}^{\perp}
\left( \frac{\partial H}{\partial \sigma^n_i} \right) - \varepsilon^2 \sigma^n_i,
\end{equation}
where $\textrm{P}_{\sigma^n_i}^{\perp} = I -  \sigma_i^n (\sigma_i^n)^T$ is the projection onto the tangent plane of $\sigma_i^n$.

Denoting the noise contribution over one step as
\begin{equation} \label{eq:gamma}
\Gamma_i^n \equiv \sigma_i^{n+1} - \sigma_i^n - \E{\sigma_i^{n+1} - \sigma_i^n},
\end{equation}
it is approximated by
\begin{equation}
\Gamma_i^n \approx \varepsilon \nu_i^n = \varepsilon \textrm{P}_{\sigma^n_i}^{\perp} (w_i^n).
\end{equation}
Thus, one step of the Metropolis-Hastings algorithm is approximately given by
\begin{equation} \label{eq:Euler step}  
\sigma_i^{n+1} - \sigma_i^n \approx -\frac 12 \beta \varepsilon^2 \textrm{P}_{\sigma^n_i}^{\perp}
\left( \frac{\partial H}{\partial \sigma^n_i} \right) - \varepsilon^2 \sigma^n_i +   \textrm{P}_{\sigma^n_i}^{\perp} (\ep w_i^n).
\end{equation}
Defining $\beta \varepsilon^2 = N \delta t$ where $\delta t$ is the time step size, the above equation changes to
\begin{equation}\label{eq:Euler step dt} 
 \sigma_i^{n+1} \approx \sigma_i^n  -\frac 12 N \textrm{P}_{\sigma^n_i}^{\perp}
\left( \frac{\partial H}{\partial \sigma^n_i} \right) \delta t - \frac N \beta \sigma^n_i \delta t +   \textrm{P}_{\sigma^n_i}^{\perp} \left( \sqrt{\frac N\beta}  w_i^n \sqrt{\delta t} \right).
\end{equation}
Since $\frac{\partial H}{\partial \sigma^n_i} = 2J (2\sigma_i^n - \sigma_{i+1}^n - \sigma_{i-1}^n)$ and $J=N$ when $d=1$, the above is the Euler step for the Langevin SDE \eqref{eq:Langevin} in It\^o interpretation
\begin{equation} \label{eq:SDE}
\df \sigma_i = \textrm{P}_{\sigma_i}^{\perp} \left( \Delta_N \sigma_i \right) \df t - \frac N \beta \sigma_i \df t + \textrm{P}_{\sigma_i}^\perp \left( \sqrt{\frac N \beta} \df W_i \right).
\end{equation}
This intuitive idea leads to the first result:
\begin{thm} \label{thm:1}
  Define the piecewise constant interpolation of M-H dynamics as $\bar{\sigma}_i(t)$,
  \begin{equation}
  \bar{\sigma}_i (t) = \sigma_i^n \quad n \delta t \le t < (n+1) \delta t,
  \end{equation}
  and $\sigma_i(t)$ as the solution for the Langevin SDE system \eqref{eq:SDE} with initial condition $\| \sigma_i(0) \| =1$, $1 \le i \le N$.
  If the noise used in the proposal for each M-H step is related to the Wiener processes driving the SDE as 
   $\varepsilon w_i^n = \sqrt{N\beta^{-1}} \left[ W_i((n+1)\delta t) -W_i(n \delta t) \right]$, then we have the following strong convergence result:
  \begin{equation}
  \E[]{ \sup_{0 \le s \le t} \| \sigma_i(s) - \bar{\sigma}_i(s)\|^2 } \le C_1 \sqrt{\delta t} \exp(C_2 T), \quad t \in [0,T],1 \le i \le N,
  \end{equation}
  for any $T \in (0,\infty)$, where $C_1,C_2$ are functions of $N,\beta,J,T$ and independent of the choice of $i$ and $\delta t$.
\end{thm}

\begin{remark}
The equation \eqref{eq:SDE} is equivalent to the SDE in Stratonovich form \eqref{eq:Langevin}
which gives $\df \|\sigma_i \|^2 = 2 \sigma_i \cdot \df \sigma_i = 0$ to make $\sigma_i$ stay on the unit sphere.
\end{remark}

\begin{remark}
Theorem \ref{thm:1} is a trajectory-wise convergence result. 
\end{remark}

\subsection{Convergence of SDE system to the Harmonic map heat flow equation.}
Notice in the SDE \eqref{eq:SDE}, if $\beta$ is chosen to be $\beta = N^\gamma, \gamma >1$, formally the noise part disappears with
$N \to \infty$. This gives the idea of the second result:

\begin{thm}\label{thm:2}
  For the harmonic map heat flow equation \eqref{eq:harmonic} with periodic boundary conditions and initial condition satisfying
  \begin{equation}\label{eq:smoothness}
  \|\sigma(\cdot,0) \| =1, \quad \| \nabla \sigma(\cdot,0) \| \le \lambda,
  \end{equation}
  for some $\lambda$ as in \cite{guo1993landau}, the solution exists and is smooth. Denote the finite difference approximation of \eqref{eq:harmonic} as 
  \begin{equation} \label{eq:fde}
  \df \tilde{\sigma}_i = \textrm{P}_{\tilde{\sigma}_i}^\perp ( \Delta_N \tilde{\sigma}_i), \quad \| \tilde{\sigma}_i \| =1.
  \end{equation}
By \cite[Theorem~1]{weinan2001numerical}, the difference between the solution to this finite difference approximation and the PDE \eqref{eq:harmonic}, $ \| \tilde{\sigma}_i(t) - \sigma(i \delta x,t)\|$, goes to zero on any fixed time interval where the solution remains well defined, as the space discretization $\delta x = \frac 1N$ goes to zero.
  
  For any $0<p<\frac 12$, there exist a constant $\gamma >1, \beta =N^\gamma$ and constants $C_1,C_2$ independent of $N$, such that if \[ \left( \left( \frac{N}{\beta} \right)^{1-p}T + C_1 \frac 1N \left( \frac{N}{\beta} \right)^{1-2p}  \right) e^{C_2 T} \le 1, \] then the difference between the SDE \eqref{eq:SDE} and the finite difference approximation \eqref{eq:fde} has the following bound
  \begin{equation}
  \E[]{ \sup_{0 \le s \le T} \frac 1N \sum_{i=1}^N \| \sigma_i(s) - \tilde{\sigma}_i(s) \|^2} \le \left(\frac{N}{\beta}\right)^{p/2} .
  \end{equation}
\end{thm}

Since $(N/\beta)^{p/2}$ is equivalent to $N^{(1-\gamma)p/2}$ with the defined scaling of $\beta=N^\gamma$, and $\gamma>1$, the quantity 
$(N/\beta)^{p/2}$ is small when $N$ is large, going to zero as $N\to\infty$.  Since the solution to the PDE is smooth, satisfying the conditions in \eqref{eq:smoothness}, the finite difference approximation is close to the PDE solution as shown in \cite{weinan2001numerical}, going to zero as $N=1/\delta x \to \infty$.   Therefore, at a fixed time $T$, the difference between the SDE system \eqref{eq:SDE} and the PDE \eqref{eq:harmonic} goes to zero as $N\to\infty$.

\begin{remark} \label{remark:beta}
 The choice of $\gamma$ depends on $p$ through the following relation
 \[
  \left( \frac N \beta \right)^{p/2} N^3 \le 1
 \]
 which is equivalent to 
 \[
 \gamma \ge 1 + \frac{6}{p}
 \]
 since $\beta = N^\gamma$.
 For a uniform bound in an interval $0 \le t \le T$, we need $p<\frac 12$ which requires $\gamma>13$. For a bound at a fixed time $t \in [0,T]$, we only need $p<1$ and thus require $\gamma>7$.  We do not believe this bound is sharp for the convergence result, which will be addressed in Section \ref{sec:numerical} when we perform numerical simulations of these models.  We find that $\gamma = \frac 32$ is enough to see convergence in our numerical simulations.  Using an approximation of being near equilibrium, it is possible to use sharper bounds on summation in $N$ from the regularity properties of the invariant measure to improve $\gamma$, but we do not pursue such techniques here.
\end{remark}

%
\section{Metropolis-Hastings dynamics to SDE system\label{sec:SDE}}
%

In this section the convergence of the M-H algorithm to the SDE \eqref{eq:SDE} for the classical Heisenberg model will be shown by calculating the drift and diffusion of one M-H step, which is approximately a stochastic Euler step for \eqref{eq:SDE}. Then the error estimation of stochastic Euler's method is used to give a bound on the difference between M-H and SDE dynamics with proposal size $\varepsilon \to 0$. Here the basic steps are outlined, the detail of error estimation is given in Appendix \ref{Appendix:dirft and diffusion}.

\begin{remark}
 The proof for the XY model will be similar, one only needs to change the random vector $\nu_i^n$ on the tangent plane to a two-dimensional vector. 
\end{remark}

\subsection{Set-up.} \label{subsec:setup}
In the calculation to follow, we have the following assumptions and notations.
The number of the particles $N$ on unit length is fixed and the limiting case $\varepsilon \to 0$ is considered. We have $\beta,J$ as functions of $N$ so they are also regarded as constant. 
On the unit sphere, the proposal is given by the exponential map \eqref{eq:expMap} and could be approximated by normalizing $\sigma_i^n + \varepsilon \nu_i^n$
\[
 \tilde{\sigma}_i^n = \exp_{\sigma_i^n} (\varepsilon \nu_i^n) \approx \frac{\sigma_i^n + \varepsilon \nu_i^n}{\| \sigma_i^n + \varepsilon \nu_i^n \|}.
\]
By Taylor expanding $\frac{\sigma_i^n + \varepsilon \nu_i^n}{\| \sigma_i^n + \varepsilon \nu_i^n \|}$, the proposal $\tilde{\sigma}_i^n$ can be approximated by order $\varepsilon$ and $\varepsilon^2$ expansion
\begin{equation} \label{eq:approx}
 \begin{aligned}
 \tilde{\sigma}_i^n & \approx \sigma_i^n + \varepsilon \nu_i^n,  \\
 \tilde{\sigma}_i^n & \approx \sigma_i^n + \varepsilon \nu_i^n -\frac 12 \varepsilon^2 (\nu_i^n \cdot \nu_i^n) \sigma_i^n.  
 \end{aligned}
\end{equation}

The proof of the following Lemma is shown in Appendix \ref{Appendix:dirft and diffusion}.
\begin{lemma} 
 Denote 
 \begin{align*}
  a_i^n & \equiv \tilde{\sigma}_i^n - \frac{\sigma_i^n + \varepsilon \nu_i^n}{\| \sigma_i^n + \varepsilon \nu_i^n \|},   \\
  c_i^n & \equiv \tilde{\sigma}_i^n - (\sigma_i^n + \varepsilon \nu_i^n), \\
  d_i^n & \equiv  \tilde{\sigma}_i^n - \left ( \sigma_i^n + \varepsilon \nu_i^n - \frac 12 \varepsilon^2 (\nu_i^n \cdot \nu_i^n) \sigma_i^n \right).
 \end{align*}
 Then $\E[]{\|a_i^n\|^k} \le A_k \varepsilon^{3k},\E[]{\|c_i^n\|^k} \le C_k \varepsilon^{2k}$ and $\E[]{\|d_i^n\|^k} \le D_k \varepsilon^{3k}$.
\end{lemma}

Using the approximation \eqref{eq:approx}, $\delta H$ in \eqref{eq:dH} can be written as
\begin{equation} \label{eq:H_approx}
  \begin{aligned}
  & \delta H = \varepsilon \frac{\partial H}{\partial \sigma_i^n} \cdot \nu_i^n + R_i^n + h_i^n \approx O(\varepsilon), \\
  & R_i^n \equiv \varepsilon \sum_{j \ne i} \frac{\partial H}{\partial \sigma_j^n} \cdot \nu_j^n  \approx O(\varepsilon), \\
  & h_i^n \equiv \sum_j \frac{\partial H}{\partial \sigma_j^n} \cdot c_j^n + 2J \sum_j \delta \sigma_j^n \cdot \delta \sigma_j^n - J \sum_{j} \delta \sigma_j^n \cdot ( \delta \sigma_{j+1}^n + \delta \sigma_{j-1}^n) \approx O(\varepsilon^2),
  \end{aligned}
\end{equation}
and we only keep the $\varepsilon$ term in $\delta H$ in the following calculation so $\delta H$ is approximated by a normal random variable. We are going to show the calculation for one specific particle $i$ so we take $i$-th term $\frac{\partial H}{\partial \sigma_i^n} \cdot \nu_i^n$ and the summation of $j \ne i$ terms as a single term $R_i^n$.

\subsection{Drift.} \label{subsec:drift}

\begin{proposition} \label{prop:theta}
  Let $\{ \sigma^n\}$ be the Markov chain given by the Metropolis-Hastings algorithm, and $\{\sigma_i^n\}$ the spin for $i$-th particle at time step
  $n$. Then
  \begin{equation}
  \E{\sigma^{n+1}_i-\sigma^n_i} = -\frac 12 \beta \varepsilon^2 \textrm{P}_{\sigma^n_i}^{\perp}
  \left( \frac{\partial H}{\partial \sigma^n_i} \right) - \varepsilon^2 \sigma^n_i + \theta_i^n,
  \end{equation}
  where the error term
  \begin{equation} \label{eq:theta}
   \theta_i^n \equiv \E{\sigma^{n+1}_i-\sigma^n_i} - \left( -\frac 12 \beta \varepsilon^2 \textrm{P}_{\sigma^n_i}^{\perp} \left( \frac{\partial H}{\partial \sigma^n_i} \right) - \varepsilon^2 \sigma^n_i \right)
  \end{equation}
  satisfies $\E[]{\|\theta_i^n\|^2} \le C \varepsilon^6$.
\end{proposition}

In the calculation we keep the order $\varepsilon^2$ term. The remainder is order $\varepsilon^3$ and will be shown to be bounded in the error estimation for M-H and SDE dynamics. The basic steps are given in the following calcuation, for details of the error estimation see Appendix \ref{Appendix:dirft and diffusion}.

Since $\sigma_i^{n+1} = \exp_{\sigma_i^n} (\varepsilon \nu_i^n)$ with probability $1 \wedge e^{-\beta \delta H}$ and stay $\sigma_i^n$ otherwise,
\begin{equation}\label{eq:drift} 
\begin{aligned}
 & \E{\sigma^{n+1}_i - \sigma^n_i} \\ 
& =  \E{\left(\exp_{\sigma_i^n} (\varepsilon \nu_i^n) - \sigma_i^n \right) \left( 1 \wedge e^{-\beta \delta H} \right)} \\
&  
\approx  \E{ \left( \varepsilon \nu_i^n - \frac{\varepsilon^2}{2} ( \nu_i^n \cdot \nu_i^n) \sigma_i^n + d_i^n \right)  \left( 1 \wedge
    e^{-\beta \delta H} \right) } \\
& 
=\varepsilon \E{   \nu_i^n\left( 1 \wedge
    e^{-\beta \delta H} \right) } - \frac{\varepsilon^2}{2}\E{  ( \nu_i^n \cdot \nu_i^n) \sigma_i^n \left( 1 \wedge
    e^{-\beta \delta H} \right) } + \E{  d_i^n\left( 1 \wedge
    e^{-\beta \delta H} \right) }.
\end{aligned}\end{equation}
We drop the third term in the last line of \eqref{eq:drift} as it is an $\varepsilon^3$ term:
$$
\E[]{  \left\|d_i^n \left( 1 \wedge
    e^{-\beta \delta H} \right) \right\| } \le \E[]{ \|d_i^n\| } \le C \varepsilon^3,
$$
since $0 < | 1 \wedge e^{-\beta \delta H} | \le 1 $ for finite $N$. 
For the second term in the last line of \eqref{eq:drift}, since $1 \wedge e^{-\beta \delta H} \approx 1 + O(\ep)$ we have that
$$
\frac{\varepsilon^2}{2}\E{  ( \nu_i^n \cdot \nu_i^n) \sigma_i^n \left( 1 \wedge
    e^{-\beta \delta H} \right) }  = \E{\frac{\varepsilon^2}{2} ( \nu_i^n \cdot \nu_i^n) \sigma_i^n}
    + O(\varepsilon^3) = \varepsilon^2 \sigma_i^n + O(\varepsilon^3).
$$
This corresponds to the It\^o correction term for \eqref{eq:Langevin}. 

The first term in the last line of \eqref{eq:drift} is the most difficult one to approximate. Using the notation in \eqref{eq:H_approx}
$$
1 \wedge e^{-\beta \delta H} = 1 \wedge e^{-\beta \left( \varepsilon \frac{\partial H}{\partial \sigma_i^n} \cdot \nu_i^n + R_i^n + h_i^n\right)}
\approx 1 \wedge e^{-\beta \left( \varepsilon \frac{\partial H}{\partial \sigma_i^n} \cdot \nu_i^n +
R_i^n \right)} + O(\ep^2),
$$
since $h_i^n \approx O(\ep^2)$, to write it as
\[
\E{ \varepsilon \nu_i^n \left(1\wedge e^{-\beta \delta H} \right)  } =
\E{ \varepsilon \nu_i^n \left(1\wedge e^{-\beta \left( \varepsilon \frac{\partial H}{\partial \sigma_i^n} \cdot \nu_i^n + R_i^n \right) }\right) } + O(\varepsilon^3).
\]
For any orthonormal basis $\{b_1,b_2,b_3\}$ in $\mathbb{R}^3$, the normal random vector $w_i^n$ can be expressed as \[w_i^n = (w_i^n \cdot b_1) b_1 +  (w_i^n \cdot b_2) b_2 + (w_i^n \cdot b_3) b_3 \] and $(w_i^n \cdot b_1),(w_i^n \cdot b_2),(w_i^n \cdot b_3)$ are independent standard normal random variables. Denote $r_1 = (w_i^n \cdot b_1),r_2 = (w_i^n \cdot b_2),r_3 = (w_i^n \cdot b_3)$,$w_i^n =r_1 b_1 + r_2 b_2 + r_3 b_3$. Choose $b_1,b_2$ two orthonormal vectors on the tangent plane of $\sigma_i^n$ and $b_3 = \sigma_i^n$,
\[
 \nu_i^n = \textrm{P}_{\sigma_i^n} (w_i^n) = r_1 b_1 + r_2 b_2,
\]
where $r_1,r_2 \sim \N(0,1)$ are independent.
Then,
\begin{equation}\label{eq:drift1} 
\begin{aligned}
\E{ \varepsilon \nu_i^n \left(1\wedge e^{-\beta \left( \varepsilon \frac{\partial H}{\partial \sigma_i^n} \cdot \nu_i^n + R_i^n \right) }\right) }
&= \E{\varepsilon r_1 b_1  \left(1\wedge e^{-\beta \left( \varepsilon r_1 \frac{\partial H}{\partial \sigma_i^n} \cdot b_1
+ \varepsilon r_2 \frac{\partial H}{\partial \sigma_i^n} \cdot b_2 + R_i^n \right) }\right) } \\
& + \E{\varepsilon  r_2 b_2 \left(1\wedge e^{-\beta \left( \varepsilon r_1 \frac{\partial H}{\partial \sigma_i^n} \cdot b_1
+ \varepsilon r_2 \frac{\partial H}{\partial \sigma_i^n} \cdot b_2 + R_i^n \right) }\right) }.
\end{aligned}\end{equation}
The two terms on the right are similar in form so we only show the calculation for the first one and the second one follows similarly.

\begin{remark}
 For the XY model, the projection of the normal random vector onto the tangent plane of $\sigma_i^n$ is represented by the form $r_1 b_1$, where $r_1 \sim N(0,1)$. The other parts of the calculation basically stays the same.
\end{remark}

Using tower property of conditional expectation for the first term on the RHS of \eqref{eq:drift1}, we have
\begin{align*}
&\E{\varepsilon r_1 b_1  \left(1\wedge e^{-\beta \left( \varepsilon r_1 \frac{\partial H}{\partial \sigma_i^n} \cdot b_1
+ \varepsilon r_2 \frac{\partial H}{\partial \sigma_i^n} \cdot b_2 + R_i^n \right) }\right) } \\
& \hspace{1cm}  = \mathbb{E}_n \left\{\E{\varepsilon r_1 b_1 \left(1\wedge e^{-\beta \left( \varepsilon r_1 \frac{\partial H}{\partial \sigma_i^n}
\cdot b_1 + \varepsilon r_2 \frac{\partial H}{\partial \sigma_i^n} \cdot b_2 + R_i^n \right) }\right) \Big | r_2,R_i^n } \right\}.
\end{align*}

We recall the following Lemma 2.4 in \cite{mattingly2012diffusion}. (See also \cite{roberts1997weak}.)
\begin{lemma}
For $z \sim \mathcal{N}(0,1)$,
\begin{equation}
\E[]{z \left( 1 \wedge e^{az+b} \right) } = a e^{\frac{a^2}{2} + b} \Phi \left( - \frac{b}{|a|} - |a|  \right),
\end{equation}
for any real constants $a,b$, and $\Phi(\cdot)$ is the CDF for the standard normal random variable.
\end{lemma}

The proof of this Lemma is the direct result of the integration for the expectation. And the Lemma gives
\begin{equation}\begin{aligned}
& \E{\varepsilon r_1 b_1 \left(1\wedge e^{-\beta \left( \varepsilon r_1 \frac{\partial H}{\partial \sigma_i^n} \cdot b_1 + \varepsilon r_2 \frac{\partial H}{\partial \sigma_i^n} \cdot b_2 + R_i^n \right) }\right) \Big | r_2,R_i^n } =\\
& \hspace{0.3cm}  -\beta \varepsilon^2 \left( \frac{\partial H}{\partial \sigma_i^n} \cdot b_1 \right) b_1
e^{\frac{\left( \beta \varepsilon \frac{\partial H}{\partial \sigma_i^n} \cdot b_1 \right)^2}{2} -
\beta \varepsilon r_2 \frac{\partial H}{\partial \sigma_i^n} \cdot b_2 - \beta R_i^n }
\Phi \left( \frac{\varepsilon r_2 \frac{\partial H}{\partial \sigma_i^n} \cdot b_2 + R_i^n}
{\left| \varepsilon \frac{\partial H}{\partial \sigma_i^n} \cdot b_1 \right| }
- \left| \beta \varepsilon \frac{\partial H}{\partial \sigma_i^n} \cdot b_1 \right| \right).
\end{aligned}\end{equation}
Before taking the expectation over $r_2$, we further simplify this expression by noting that $e^{O(\ep)} = 1 + O(\ep)$ resulting in
\begin{equation}\label{eq:drift2} 
\begin{aligned}
& \E{\varepsilon r_1 b_1 \left(1\wedge e^{-\beta \left( \varepsilon r_1 \frac{\partial H}{\partial \sigma_i^n}
\cdot b_1 + \varepsilon r_2 \frac{\partial H}{\partial \sigma_i^n} \cdot b_2 + R_i^n \right) }\right) \Big | r_2,R_i^n }\\
& \hspace{1cm} \approx  \left( -\beta \varepsilon^2 \left( \frac{\partial H}{\partial \sigma_i^n} \cdot b_1 \right) b_1 + O(\ep^3) \right)
\left[ \Phi \left( \frac{\varepsilon r_2 \frac{\partial H}{\partial \sigma_i^n} \cdot b_2 + R_i^n}
{\left| \varepsilon \frac{\partial H}{\partial \sigma_i^n} \cdot b_1 \right| }
  \right) + O(\ep) \right].
\end{aligned}\end{equation}
For a mean zero Gaussian random variable $z$, we know
\[
\E[]{\Phi(z)} = \E[]{\Phi(z) - \frac 12 + \frac 12} = \int_{-\infty}^\infty \left( \Phi(z) - \frac 12 + \frac 12 \right) p(z) \df z = \frac 12,
\]
as $\Phi(z) - \frac 12$ is an odd function and the probability density function $p(z)$ is even.

Notice that $ R_i^n = \varepsilon \sum_{j \ne i} \frac{\partial H}{\partial \sigma_j^n} \cdot \nu_j^n$ is a sum of independent mean zero Gaussian random variables, so
$\varepsilon r_2 \frac{\partial H}{\partial \sigma_i^n} \cdot b_2 + R_i^n$ is a Gaussian random variable with mean $0$, therefore
\[
\E{ -\beta \varepsilon^2 \left( \frac{\partial H}{\partial \sigma_i^n} \cdot b_1 \right) b_1
\Phi \left(  \frac{\varepsilon r_2 \frac{\partial H}{\partial \sigma_i^n} \cdot b_2 + R_i^n}
{\left| \varepsilon \frac{\partial H}{\partial \sigma_i^n} \cdot b_1 \right| } \right) } =
-\frac 12 \beta \varepsilon^2 \left( \frac{\partial H}{\partial \sigma_i^n} \cdot b_1 \right) b_1.
\]

The second term on the RHS of \eqref{eq:drift1} follows similarly,
\[
\E{\varepsilon r_2 b_2  \left(1\wedge e^{-\beta \left( \varepsilon r_1 \frac{\partial H}{\partial \sigma_i^n} \cdot b_1
+ \varepsilon r_2 \frac{\partial H}{\partial \sigma_i^n} \cdot b_2 + R_i^n \right) }\right) } =
  -\frac 12 \beta \varepsilon^2 \left( \frac{\partial H}{\partial \sigma_i^n} \cdot b_2 \right) b_2 + O(\varepsilon^3).
\]

Combining the above
\begin{equation*}
  \E{\sigma^{n+1}_i-\sigma^n_i} \approx -\frac 12 \beta \varepsilon^2 \textrm{P}_{\sigma^n_i}^{\perp} \left( \frac{\partial H}{\partial \sigma^n_i} \right) - \varepsilon^2 \sigma^n_i,
\end{equation*}
where $ \frac{\partial H}{\partial \sigma^n_i} = \frac{J}{N^2} \Delta_N \sigma_i^n$ and $\Delta_N \sigma_i^n = N^2 (\sigma_{i+1}^n + \sigma_{i-1}^n -2\sigma_i^n)$ denotes the discrete Laplacian.

\subsection{Diffusion.}
Recall $\Gamma_i^n$ in \eqref{eq:gamma},
\[
\Gamma_i^n =  \begin{cases}
		\varepsilon \nu_i^n + c_i^n - \E{\sigma_i^{n+1} - \sigma_i^n } & \textrm{with probability } \alpha \\
		- \E{\sigma_i^{n+1} - \sigma_i^n } & \textrm{with probability } 1-\alpha
	      \end{cases}
\]
with accept rate $\alpha$ in \eqref{eq:accept}. Since $\E{\sigma_i^{n+1} - \sigma_i^n }$ is an order $\varepsilon^2$ term and $\alpha \approx 1$ with small $\varepsilon$, we are going to show
\[
\Gamma_i^n \approx \varepsilon \nu_i^n.
\]

\begin{proposition} \label{prop:phi}
The diffusion term
\begin{equation}
\Gamma_i^n = \sigma_i^{n+1} - \sigma_i^n - \E{\sigma_i^{n+1} - \sigma_i^n } = \varepsilon \nu_i^n + \phi_i^n,
\end{equation}
where
\begin{equation} \label{eq:phi}
\phi_i^n \equiv \Gamma_i^n - \varepsilon \nu_i^n = \sigma_i^{n+1} - \sigma_i^n - \E{\sigma_i^{n+1} - \sigma_i^n } - \varepsilon\nu_i^n
\end{equation}
is a random variable with mean $\E[]{\phi_i^n} = 0$, variance $\E[]{\| \phi_i^n \|^2} \le C \varepsilon^3$, and covariance $\E[]{\phi_i^n \cdot \phi_i^m} = 0$ for $n \ne m$.
\end{proposition}

\begin{proof}
For the mean
\[
\phi_i^n = \sigma_i^{n+1} - \sigma_i^n - \E{\sigma_i^{n+1} - \sigma_i^n } - \varepsilon \nu_i^n,
\]
then $\E[]{\phi_i^n} = \E[]{\sigma_i^{n+1} - \sigma_i^n - \E{\sigma_i^{n+1} - \sigma_i^n } - \varepsilon \nu_i^n } = 0$.

For the variance,
\begin{equation} \label{eq:phi^2}
\begin{aligned}
 \E[]{ \| \phi_i^n \|^2} 
=& \E[]{ \left\| \exp_{\sigma_i^n}(\varepsilon \nu_i^n) - \sigma_i^n - \E{\sigma_i^{n+1} - \sigma_i^n } -\varepsilon \nu_i^n \right\|^2 \left( 1 \wedge e^{-\beta \delta H}\right) }  \\
& \qquad + \E[]{ \left\|-\varepsilon \nu_i^n -\E{\sigma_i^{n+1} - \sigma_i^n} \right\|^2 \left(1- \left( 1 \wedge e^{- \beta \delta H}\right)\right) }  \\
=& \E[]{\left\| c_i^n - \E{\sigma_i^{n+1}- \sigma_i^n} \right\|^2 \left( 1 \wedge e^{-\beta \delta H}\right) } \\ &\qquad+ \E[]{ \left\|-\varepsilon \nu_i^n -\E{\sigma_i^{n+1} - \sigma_i^n}\right\|^2 \left(1- \left( 1 \wedge e^{-\beta \delta H}\right)\right) }.  
\end{aligned}
\end{equation}

The first term in the last line of \eqref{eq:phi^2}
\begin{align*}
  & \E[]{\left\| c_i^n - \E{\sigma_i^{n+1}- \sigma_i^n} \right\|^2 \left( 1 \wedge e^{- \beta \delta H}\right) } \\
\le & \E[]{ \left\| c_i^n - \E{\sigma_i^{n+1}- \sigma_i^n} \right\|^4 }^{\frac 12} \E[]{\left( 1 \wedge e^{-\beta \delta H}\right)^2}^{\frac 12} \\
 &  \hspace{0.6cm}  * C \left( \E[]{\| c_i^n \|^4} + \E[]{ \left \| \E{\sigma_i^{n+1}- \sigma_i^n} \right \|^4} \right)^{\frac 12} \\
\le &  C \varepsilon^4,
\end{align*}
(here, $*$ denotes multiplication) as $\E{\sigma_i^{n+1}- \sigma_i^n} = -\frac 12 \beta \varepsilon^2 \textrm{P}_{\sigma^n_i}^{\perp} \left( \frac{\partial H}{\partial \sigma^n_i} \right) - \varepsilon^2 \sigma^n_i + O(\varepsilon^3)$ and $\E[]{ \|c_i^n \|^4 } \le C \varepsilon^8$ shown in Appendix \ref{Appendix:dirft and diffusion}.

For the second term in the last line of \eqref{eq:phi^2}, since $ \left| 1 - \left( 1 \wedge e^{- \beta \delta H}\right) \right| = \left| e^0 - e^{0 \wedge (-\beta \delta H)} \right| \le | \beta \delta H|$, we observe that
\begin{align*}
\mathbb{E}&\left[ \left\|-\varepsilon \nu_i^n -\E{\sigma_i^{n+1} - \sigma_i^n}\right\|^2 \left(1- \left( 1 \wedge e^{-\beta \delta H}\right)\right) \right]\\
&\le  \E[]{\left\|-\varepsilon \nu_i^n -\E{\sigma_i^{n+1} - \sigma_i^n}\right\|^4 }^{\frac 12} \E[]{ | \beta \delta H|^2}^{\frac 12} \\
& \le  C \varepsilon^3,
\end{align*}
for some constant $C$.   We have used that $-\varepsilon \nu_i^n -\E{\sigma_i^{n+1} - \sigma_i^n} = -\varepsilon \nu_i^n + O(\varepsilon^2)$ and $\delta H = \varepsilon \sum_{j} \frac{\partial H}{\partial \sigma_j^n} \cdot \nu_j^n + O(\varepsilon^2)$ are both order $\varepsilon$ term.

Combining the above, the variance in \eqref{eq:phi^2} is bounded by $\E[]{ \| \phi_i^n \|^2} \le C\varepsilon^3$.

For the covariance of $\phi_i^n, \phi_i^m$ at different time steps $n>m$, and $\zeta=x,y,z$ denotes the coordinates of the vector,
\[
 \E[]{\phi_{i,\zeta}^n \phi_{i,\zeta}^m} 
= \E[]{\E{\phi_{i,\zeta}^n \phi_{i,\zeta}^m}} 
= \E[]{ \phi_{i,\zeta}^m \E{\phi_{i,\zeta}^n}} 
= \E[]{ \phi_{i,\zeta}^m 0 } 
= 0.
\]

\end{proof}

\begin{remark}
In fact, the error term is $\E[]{\| \phi_i^n \|^2} = O(\varepsilon^3)$ and this determines the order of the convergence in Theorem \ref{thm:1}.
The details of the calculation are given in Appendix \ref{Appendix:diffusion}.
\end{remark}

\subsection{Proof of Theorem \ref{thm:1}}
For the error estimation, we apply similar techniques as in the proof of stochastic Euler's method.

Take $\sigma_i,\bar{\sigma}_i$ as in Theorem \ref{thm:1}. For simplicity we denote 
\begin{equation}\label{eq:mu_psi}
\mu_i(\sigma) = \textrm{P}_{\sigma_i}^\perp (\Delta_N \sigma_i) - \frac N \beta \sigma_i, \qquad \psi_i(\sigma) = \sqrt{N}{\beta} (I - \sigma_i \sigma_i^T)
\end{equation}
as the drift and diffusion coefficients in \eqref{eq:SDE}, respectively, where $\sigma$ is the collection of all the $\sigma_i$. When $N,J,\beta$ are fixed and $\|\sigma_i\|=1$, the coefficient $\mu(\sigma)$ and $\psi(\sigma)$ are Lipschitz continuous in each coordinates of $\sigma$. From Theorem 5.2.1 in \cite{oksendal2003stochastic}, the SDE system has a unique solution.

Now we have the following estimate on the error.
\begin{proposition} \label{proposition:mh_sde1}
  Define the error $e(t)$ between M-H interpolation $\bar{\sigma}_i$ and SDE \eqref{eq:SDE} solution $\sigma_i$ as
  \begin{equation} \label{eq:e1}
    e(t) \equiv \sup_{1 \le i \le N, 0 \le s \le t} \E[]{ \left\|\sigma_{i}(s) - \bar{\sigma}_{i}(s)\right\|^2 }. 
  \end{equation}
  For any fixed $T>0$, $e(t)$ is bounded by
  \begin{equation}
   e(t) \le C(N,J,\beta,T) \sqrt{\delta t} \quad t \in [0,T].
  \end{equation}
\end{proposition}

\begin{proof}
  For the proof we are going to show $e(t)$ satisfies the Gr\"onwall inequality ($C_i$ denotes some constant bound):
  \begin{equation} \label{eq:e3}
    e(t) \le \left( C_1 T + C_2 \right) \int_0^t e(s) \df s + \left( C_3 \sqrt{\delta t} + C_4 \delta t + C_5 \delta_t^2 \right),
  \end{equation}
  so $e(t) \le \left( C_3 \sqrt{\delta t} + C_4 \delta t + C_5 \delta t^2 \right) \exp \left(C_1 T (T + C_2) \right)$.

  Since $\bar{\sigma}_i(t) = \sigma_i^{  \left \lfloor \frac{t}{\delta t}  \right \rfloor} = \sigma_i^0 + \sum_{j=0}^{\left \lfloor \frac{s}{\delta t} \right \rfloor-1} \left( \sigma_i^{j+1} - \sigma_i^j \right)$ and both $\sigma_i,\bar{\sigma}_i$ start from the same initial condition, from the definition of $e(t)$ and $\theta_i^n$ in \eqref{eq:theta}, $\Gamma_i^n$ in \eqref{eq:gamma}, $\phi_i^n$ in \eqref{eq:phi} we have that 
\begin{align}
    e(t) =& \sup_{1 \le i \le N, 0 \le s \le t} \E[]{ \left\|\sigma_{i}(s) - \bar{\sigma}_{i}(s)\right\|^2 } \nonumber \\
    =& \sup_{1 \le i \le N, 0 \le s \le t} \mathbb{E} \left[ \left\| \int_0^s \mu_i \big(\sigma(u) \big) \df u + \int_0^s \psi_i \big(\sigma(u) \big) \df W_i(u) \right. \right. \nonumber \\
    & \hspace{3cm}-\sum_{j=0}^{\left \lfloor \frac{s}{\delta t} \right \rfloor-1} \left(\E{\sigma_i^{j+1} - \sigma_i^j} + \Gamma_i^j \right)\bigg\|^2 \bigg] \nonumber \\
    =& \sup_{1 \le i \le N, 0 \le s \le t} \mathbb{E} \left[  \left\| \int_0^s \mu_i \big( \sigma(u) \big) \df u + \int_0^s \psi_i \big( \sigma(u) \big) \df W_i(u) \right. \right. \nonumber \\ 
    & \hspace{3cm} - \sum_{j=0}^{\left \lfloor \frac{s}{\delta t} \right \rfloor-1} \left(\mu_i(\sigma^j) \delta t + \varepsilon \nu_i^j + \theta_i^j + \phi_i^j \right)\bigg\|^2  \bigg] 
\end{align}
where $ \mu_i(\sigma^j) \delta t = \int_{j \delta t}^{ (j+1) \delta t}  \mu_i \big( \bar{\sigma}(u) \big) \df u$ and $\varepsilon \nu_i^j =  \int_{j \delta t}^{ (j+1) \delta t}  \psi_i \big( \bar{\sigma}(u) \big) \df W_i(u) $. Applying H\"older's inequality and $\E[]{|X+Y|^2} \le 2 \E[]{X^2 + Y^2}$ produces
\begin{align} \label{eq:e2}
    e(t)   \le & C \sup_{1 \le i \le N, 0 \le s \le t} \mathbb{E} \left[ \left\| \int_0^{\left \lfloor \frac{s}{\delta t} \right \rfloor \delta t} \left[\mu_i \big( \sigma(u) \big) - \mu_i \big( \bar{\sigma}(u) \big)\right] \df u \right\|^2 \right. \nonumber \\
    & \hspace{1cm}  \left. + \left\| \int_0^{\left \lfloor \frac{s}{\delta t} \right \rfloor \delta t} \Big( \psi_i \big(\sigma(u) \big) - \psi_i \big(\bar{\sigma}(u) \big) \Big) \df W_i(u) \right\|^2 + \left\| \int_{\left \lfloor \frac{s}{\delta t} \right \rfloor \delta t}^s \mu_i \big(\sigma(u) \big) \df u \right\|^2 \right. \nonumber  \\
    & \hspace{1cm} \left.  + \left\| \int_{\left \lfloor \frac{s}{\delta t} \right \rfloor \delta t}^s \psi_i \big( \sigma_i(u) \big) \df W_i(u) \right\|^2 + \left\| \sum_{j=0}^{\left \lfloor \frac{s}{\delta t} \right \rfloor-1} \theta_i^j  \right\|^2 + \left\| \sum_{j=0}^{\left \lfloor \frac{s}{\delta t} \right \rfloor-1} \phi_i^j  \right\|^2 \right].  
\end{align}

Using H\"older inequality for the first term in \eqref{eq:e2} with the coordinate $\zeta = x,y,z$ 
$$
  \left|\int_{0}^{ \left \lfloor \frac{s}{\delta t} \right \rfloor \delta t} \left[ \mu_{i,\zeta} \big( \sigma(u) \big) - \mu_{i,\zeta} \big( \bar{\sigma} (u) \big) \right]\df u \right|^2 
  \le \int_0^{\left \lfloor \frac{s}{\delta t} \right \rfloor \delta t}  \Big( \mu_{i,\zeta} \big(\sigma(u)\big) - \mu_{i,\zeta} \big(\bar{\sigma} (u) \big) \Big)^2 \df u \times \int_0^{\left \lfloor \frac{s}{\delta t} \right \rfloor \delta t} 1^2 \df u,
$$
  since $\mu_i$ is Lipschitz,
  \[
   \int_0^{\left \lfloor \frac{s}{\delta t} \right \rfloor \delta t}  \Big( \mu_{i,\zeta} \big(\sigma(u)\big) - \mu_{i,\zeta} \big(\bar{\sigma} (u) \big) \Big)^2 \df u \le C_1 \int_0^{\left \lfloor \frac{s}{\delta t} \right \rfloor \delta t} \big( \sigma_{i,\zeta} (u) - \bar{\sigma}_{i,\zeta}(u) \big)^2 \df u,
  \]
  combine $\zeta = x,y,z$ terms,
  \[
  \sup_{1 \le i \le N, 0 \le s \le t}  \E[]{ \left\| \int_0^{\left \lfloor \frac{s}{\delta t} \right \rfloor \delta t}\left[ \mu_i \big(\sigma(u) \big) - \mu_i \big(\bar{\sigma}(u) \big)\right] \df u \right\|^2 } \le C_1 t \int_0^t e(s) \df s.
  \]
  
  Applying It\^o isometry to the second term of \eqref{eq:e2} for the $x$ coordinate
  \begin{align*}
  & \E[]{ \left( \int_{0}^{\left \lfloor \frac{s}{\delta t} \right \rfloor \delta t}\left[ \psi_{i} \big(\sigma(u)\big) - \psi_{i} \big(\bar{\sigma} (u) \big)\right] \df W_{i} (u)\right)_x^2 } \\
  =&   \mathbb{E} \left[ \left( \int_0^{\left \lfloor \frac{s}{\delta t} \right \rfloor \delta t} \bar{\sigma}_i^x(u) ( \bar{\sigma}_i^x(u) \df W_i^x(u) + \bar{\sigma}_i^y(u) \df W_i^y(u) + \bar{\sigma}_i^z(u) \df W_i^z(u)) \right. \right. \\
  & \hspace{2cm}  - \sigma_i^x(u) (\sigma_i^x(u) \df W_i^x(u) + \sigma_i^y(u) \df W_i^y(u) + \sigma_i^z(u) \df W_i^z(u)) \bigg)^2 \Bigg] \\
    =&  \mathbb{E} \left[ \left( \int_0^{\left \lfloor \frac{s}{\delta t} \right \rfloor \delta t}\left[ (\bar{\sigma}_i^x(u))^2 - (\sigma_i^x(u))^2\right] \df W_i^x(u) \right)^2  \right. \\
    & \hspace{2cm} + \left( \int_0^{\left \lfloor \frac{s}{\delta t} \right \rfloor \delta t} \left[\bar{\sigma}_i^x(u) \bar{\sigma}_i^y(u) - \sigma_i^x(u) \sigma_i^y(u)\right] \df W_i^y(u) \right)^2 \\
    & \hspace{2cm} \left. + \left( \int_0^{\left \lfloor \frac{s}{\delta t} \right \rfloor \delta t}\left[ \bar{\sigma}_i^x(u) \bar{\sigma}_i^z(u) - \sigma_i^x(u) \sigma_i^z(u)\right] \df W_i^z(u) \right)^2 \right] \\
  =& \mathbb{E} \left[ \int_0^{\left \lfloor \frac{s}{\delta t} \right \rfloor \delta t} \left( (\bar{\sigma}_i^x(u))^2 - (\sigma_i^x(u))^2 \right)^2 \df u + \int_0^{\left \lfloor \frac{s}{\delta t} \right \rfloor \delta t} \left( \bar{\sigma}_i^x(u) \bar{\sigma}_i^y(u) - \sigma_i^x(u) \sigma_i^y(u) \right)^2 \df u \right. \\
  & \hspace{2cm} \left. + \int_0^{\left \lfloor \frac{s}{\delta t} \right \rfloor \delta t} \left( \bar{\sigma}_i^x(u) \bar{\sigma}_i^z(u) - \sigma_i^z(u) \sigma_i^z(u) \right)^2 \df u \right] . 
  \end{align*}
  Since 
  \begin{align*}
    & \left( (\bar{\sigma}_i^x)^2 - (\sigma_i^x)^2 \right)^2 \le (|\bar{\sigma}_i^x|+|\sigma_i^x|)^2 ( \bar{\sigma}_i^x - \sigma_i^x)^2 \le 4 (\bar{\sigma}_i^x - \sigma_i^x)^2 ,\\
    & \left( \bar{\sigma}_i^x \bar{\sigma}_i^y - \sigma_i^x \sigma_i^y \right)^2  = \big( (\bar{\sigma}_i^x - \sigma_i^x) \bar{\sigma}_i^y + \sigma_i^x ( \bar{\sigma}_i^y - \sigma_i^y)  \big)^2 \le 2 (\bar{\sigma}_i^x - \sigma_i^x)^2 (\bar{\sigma}_i^y)^2 + 2 (\sigma_i^x)^2 ( \bar{\sigma}_i^y - \sigma_i^y)^2, \\
    & \left( \bar{\sigma}_i^x \bar{\sigma}_i^z - \sigma_i^x \sigma_i^z \right)^2  = \big( (\bar{\sigma}_i^x - \sigma_i^x) \bar{\sigma}_i^z + \sigma_i^x ( \bar{\sigma}_i^z - \sigma_i^z) \big)^2 \le 2 (\bar{\sigma}_i^x - \sigma_i^x)^2 (\bar{\sigma}_i^z)^2 + 2 (\sigma_i^x)^2 ( \bar{\sigma}_i^z - \sigma_i^z)^2, 
  \end{align*}
  then
  \begin{dmath*}
   \E[]{ \left( \int_{0}^{\left \lfloor \frac{s}{\delta t} \right \rfloor \delta t} \left[ \psi_{i} \big(\sigma(u) \big) - \psi_{i} \big(\bar{\sigma}(u) \big) \right]\df W_{i} (u)\right)_x^2 } \le  C \E[]{ \int_0^{\left \lfloor \frac{s}{\delta t} \right \rfloor \delta t}\left[ (\bar{\sigma}_i^x(u) - \sigma_i^x(u))^2 + (\bar{\sigma}_i^y(u) - \sigma_i^y(u))^2 + (\bar{\sigma}_i^z(u) - \sigma_i^z(u))^2 \right] \df u}
  \end{dmath*}
  and $y,z$ coordinates of the second term in \eqref{eq:e2} are similar. Summing up $x,y,z$ coordinates, the second term in \eqref{eq:e2} is bounded by
  \[
  \sup_{1 \le i \le N, 0 \le s \le t}  \E[]{\left\| \int_0^{\left \lfloor \frac{s}{\delta t} \right \rfloor \delta t} \Big ( \psi_i \big(\sigma(u) \big) - \psi_i \big(\bar{\sigma}(u) \big) \Big ) \df W_i(u) \right\|^2} \le C_2 \int_0^t e(s) \df s.
  \]

  For the third term in \eqref{eq:e2}
  \[
  \left\| \int_{\left \lfloor \frac{s}{\delta t} \right \rfloor \delta t}^s \mu_i \big(\sigma(u) \big) \df u \right\|^2 \le C_3 \delta t^2,
  \]
  since $\| \sigma_i\|=1$ and $ s-\left \lfloor \frac{s}{\delta t} \right \rfloor \delta t \le \delta t$.

  Apply It\^o isometry again for the fourth term in \eqref{eq:e2},
  \[
  \E[]{\left\| \int_{\left \lfloor \frac{s}{\delta t} \right \rfloor \delta t}^s \psi_i \big(\sigma_i(u)\big) \df W_i(u) \right\|^2} \le C_4 \delta t.
  \]

  From Cauchy inequality and $\E[]{ \left\| \theta_{i}^j \right\|^2} \le C \varepsilon^6$ in Proposition \ref{prop:theta}, the fifth term in \eqref{eq:e2} is bounded by
  \[
  \E[]{ \left\| \sum_{j=0}^{\left \lfloor \frac{s}{\delta t} \right \rfloor-1} \theta_{i}^j \right\|^2 } \le \left \lfloor \frac{s}{\delta t} \right \rfloor \sum_{j=0}^{\left \lfloor \frac{s}{\delta t} \right \rfloor-1} \E[]{ \left\| \theta_{i}^j \right\|^2} \le C_5 \left(\left \lfloor \frac{t}{\delta t} \right \rfloor \right)^2 \varepsilon^6 = C_6 \delta t.
  \]
From Proposition \ref{prop:phi}, $\E[]{\phi_i^j \cdot \phi_i^k} \le C \delta_{jk} \varepsilon^3$ with $\delta_{jk}$ the Kronecker delta, the sixth term in \eqref{eq:e2}
  \[
  \E[]{\left\| \sum_{j=0}^{\left \lfloor \frac{s}{\delta t} \right \rfloor-1} \phi_{i}^j \right\|^2 }  = \sum_{j=0}^{\left \lfloor \frac{s}{\delta t} \right \rfloor-1} \E[]{ \left\| \phi_{i}^j \right\|^2} \le C_7 \left \lfloor \frac{t}{\delta t} \right \rfloor \varepsilon^3 \le C_8 \sqrt{\delta t}.
  \]
  Combining all above, we get the Gr\"onwall inequality \eqref{eq:e3}.
\end{proof}

\begin{remark}
  In the Gr\"onwall inequality, the $C_3 \sqrt{\delta t}$ term decides the order of convergence. It comes from $\E[]{\left\| \sum_{j=0}^{\left \lfloor \frac{s}{\delta t} \right \rfloor-1} \phi_{i}^j \right\|^2 }$, which we show is an $O(\varepsilon^3)$ term in \ref{Appendix:diffusion}.
\end{remark}

With Proposition \ref{proposition:mh_sde1}, we can get a uniform bound by using Doob's martingale inequality in \cite{oksendal2003stochastic} for a nonnegative submartingale $X_t$ and constant $p>1$:
\begin{equation} \label{eq:Doob}
  \E[]{\left| \sup_{0\le s\le t} X_s \right|^p }^{1/p} \le \frac{p}{p-1} \E[]{|X_t|^p}^{1/p}.
\end{equation}

\begin{proposition}
 Define the error between M-H and SDE dynamics for $i$-th spin as
 \begin{equation}
   e_i(t) = \E[]{\sup_{0\le s \le t} \| \sigma_i(s) - \bar{\sigma}_i(s) \|^2 } \quad t \in [0,T], 1 \le i \le N,
 \end{equation}
 for any $T \in \mathbb{R}^+$. There exists some constant $C$ as a function of $T,N,J,\beta$ and independent of $i,\delta t$,
 \begin{equation} \label{eq:e2_prop4}
   e_i(t) \le C \sqrt{\delta t}.
 \end{equation}
for any $t \in [0,T]$.
\end{proposition}

\begin{proof}
  Similar to \eqref{eq:e2} in Proposition \ref{proposition:mh_sde1}, we have that
  \begin{align} \label{eq:e_prop4}
   \mathbb{E}\bigg[ \sup_{0\le s \le t} &\left\| \sigma_i(s) - \bar{\sigma}_i(s) \right\|^2\bigg]
    \le  C  \mathbb{E} \left[ \sup_{0 \le s \le t} \left \| \int_0^{\left \lfloor \frac{s}{\delta t} \right \rfloor \delta t } \left[ \mu_i(\sigma(u)) - \mu_i(\bar{\sigma}(u)) \right]\df u \right \|^2\right. \nonumber\\
    & + \sup_{0 \le s \le t} \left \| \int_{\left \lfloor \frac{s}{\delta t} \right \rfloor \delta t}^s  \mu_i(\bar{\sigma}(u)) \df u \right\|^2  \nonumber\\
    & + \sup_{0 \le s \le t} \left \| \int_0^{\left \lfloor \frac{s}{\delta t} \right \rfloor \delta t } \psi_i(\sigma(u)) - \psi_i(\bar{\sigma}(u)) \df W_i(u) +
    \int_{\left \lfloor \frac{s}{\delta t} \right \rfloor \delta t}^s  \psi_i(\sigma(u)) \df W_i(u) \right \|^2  \nonumber\\
    & \left. +  \sup_{0 \le s \le t} \left( \sum_{j=0}^{\left \lfloor \frac{s}{\delta t} \right \rfloor-1} \theta_i^j \right)^2 +
    \sup_{0 \le s \le t} \left( \sum_{j=0}^{\left \lfloor \frac{s}{\delta t} \right \rfloor-1} \phi_i^j \right)^2 \right],
  \end{align}
%
  and each term on the RHS will be bounded by $C \sqrt{\delta t}$ for some constant $C$.
 Applying H\"older's inequality for the first term in \eqref{eq:e_prop4}:
  \begin{align*}
    & \E[]{\sup_{0 \le s \le t} \left \| \int_0^{\left \lfloor \frac{s}{\delta t} \right \rfloor \delta t } \left[ \mu_i(\sigma(u)) - \mu_i(\bar{\sigma}(u)) \right] \df u \right \|^2 } \\
    & \hspace{0.5cm} \le  \E[]{ \sup_{0 \le s \le t} t \int_0^{\left \lfloor \frac{s}{\delta t} \right \rfloor \delta t}  \| \mu_i(\sigma(u)) - \mu_i(\bar{\sigma}(u))\|^2 \df u  } \\
    & \hspace{0.5cm} \le  \E[]{ t \int_0^{t}  \| \mu_i(\sigma(u)) - \mu_i(\bar{\sigma}(u))\|^2 \df u  } 
    \le  C t^2 \sum_{j=i-1}^{i+1} \sup_{0 \le s \le t} \E[]{ \| \sigma_i(s) - \bar{\sigma}_i(s) \|^2 } \\ 
    & \hspace{0.5cm} \le C t^2 \sqrt{\delta t}.
  \end{align*}
  The last inequality above is from $\sup_{0 \le s \le t, 1\le i \le N} \E[]{ \| \sigma_i(s) - \bar{\sigma}_i(s) \|^2 } \le C \sqrt{\delta t}$ in Proposition \ref{proposition:mh_sde1}.

  For the second term in \eqref{eq:e_prop4}, the length of the interval of integration is smaller than $\delta t$ and $\| \mu_i(\bar{\sigma}(u))\| \le C$ as $\| \bar{\sigma}_i \| =1$, so
  \begin{equation}
    \sup_{0 \le s \le t} \left \| \int_{\left \lfloor \frac{s}{\delta t} \right \rfloor \delta t}^s  \mu_i(\bar{\sigma}(u)) \df u \right\|^2 \le C \delta t^2.
  \end{equation}
  The integral in the third term of \eqref{eq:e_prop4} is a martingale. For $\zeta = x,y,z$, denote 
  \[ 
  M_{t,\zeta} \equiv \left( \int_0^{\left \lfloor \frac{s}{\delta t} \right \rfloor \delta t } \left[  \psi_i(\sigma(u)) - \psi_i(\bar{\sigma}(u)) \right]\df W_i(u) 
  + \int_{\left \lfloor \frac{s}{\delta t} \right \rfloor \delta t}^s  \psi_i(\sigma_i(u)) \df W_i(u) \right)_\zeta. 
  \]
  This is a Martingale and $X_t = \left| M_{t,\zeta} \right|$ is a nonnegative submartingale, hence by Doob's inequality \eqref{eq:Doob} 
  \[ \E[]{ \| \sup_{0 \le s \le t} X_s\|^2} \le 4 \E[]{\|X_t\|^2}. \]
  For a nonnegative submartingale $X_t$
  \[ \E[]{ \sup_{0 \le s \le t} X_s^2} = \E[]{ \left( \sup_{0 \le s \le t} X_s \right)^2} \le C \E[]{ X_t^2} = C \E[]{M_{t,\zeta}^2}. \]
  Applying It\^o isometry for the last term similar to the proof of Proposition \ref{proposition:mh_sde1} and summing for all coordinates $\zeta=x,y,z$: 
  \begin{align*}
    & \E[]{\sup_{0 \le s \le t} \left \| \int_0^{\left \lfloor \frac{s}{\delta t} \right \rfloor \delta t }\left[ \psi_i(\sigma(u)) - \psi_i(\bar{\sigma}(u)) \right]\df W_i(u)
     + \int_{\left \lfloor \frac{s}{\delta t} \right \rfloor \delta t}^s   \psi_i(\sigma_i(u)) \df W_i(u) \right \|^2} \\
    & \hspace{0.5cm} \le  C \E[]{ \left \| \int_0^{\left \lfloor \frac{t}{\delta t} \right \rfloor \delta t }\left[ \psi_i(\sigma(u)) - \psi_i(\bar{\sigma}(u))\right] \df W_i(u)
     + \int_{\left \lfloor \frac{t}{\delta t} \right \rfloor \delta t}^s  \psi_i(\sigma_i(u)) \df W_i(u) \right \|^2} \\
    & \hspace{0.5cm} \le  C_1 \E[]{ \int_0^{\left \lfloor \frac{t}{\delta t} \right \rfloor \delta t } \left\| \sigma_i(u) - \bar{\sigma}_i(u) \right\|^2 \df u } +
    C_2 \E[]{ \int_{\left \lfloor \frac{t}{\delta t} \right \rfloor \delta t}^s  \| \sigma_i(u) \|^2 \df u} \\
    & \hspace{0.5cm} \le C_1 \sqrt{\delta t} + C_2 \delta t.
  \end{align*}

  From Cauchy inequality, the fourth term in \eqref{eq:e_prop4} 
  \begin{align*}
    \sup_{0 \le s \le t} \left( \sum_{j=0}^{\left \lfloor \frac{s}{\delta t}-1 \right \rfloor} \theta_i^j \right)^2 \le \sup_{0 \le s \le t} \left \lfloor \frac{s}{\delta t} \right \rfloor \sum_{j=0}^{\left \lfloor \frac{s}{\delta t} \right \rfloor-1} \| \theta_i^j \|^2 \le \left \lfloor \frac{t}{\delta t} \right \rfloor \sum_{j=0}^{\left \lfloor \frac{t}{\delta t} \right \rfloor-1} \| \theta_i^j \|^2,
  \end{align*}
  from Proposition \ref{prop:theta} $\E[]{\| \theta_i^j \|^2} \le C \varepsilon^6$ so the last expectation is bounded by $C \delta t$.

  In the fifth term of \eqref{eq:e_prop4}, $\sum_{j=0}^{\left \lfloor \frac{s}{\delta t} \right \rfloor -1} \phi_i^j$ is a discrete martingale. Again using martingale inequality for each coordinate and then summing up,
  \begin{equation}
    \E[] {\sup_{0 \le s \le t} \left\| \sum_{j=1}^{\left \lfloor \frac{s}{\delta t} \right \rfloor} \phi_i^j \right\|^2 } \le \E[]{\left\| \sum_{j=1}^{\left \lfloor \frac{t}{\delta t} \right \rfloor} \phi_i^j \right\|^2 },
  \end{equation}
  from \ref{prop:phi} $\E[]{\phi_i^j \cdot \phi_i^k} \le \delta_{jk} C \varepsilon^3$, it is bounded by
  \begin{equation}
    \E[]{\sum_{j=1}^{\left \lfloor \frac{t}{\delta t} \right \rfloor} \|\phi_i^j\|^2} \le C \left \lfloor \frac{t}{\delta t} \right \rfloor \varepsilon^3 \le C \sqrt{\delta t}.
  \end{equation}
  
  Combining the above, \eqref{eq:e2_prop4} is obtained.
\end{proof}

%
\section{From SDE system to deterministic PDE\label{sec:PDE}}
%

In this section, we explain the convergence from the SDE system \eqref{eq:SDE} to the deterministic harmonic map heat flow equation without the dispersion term \eqref{eq:LLG} with proper choice of $\beta = N^\gamma$ and number of particles $N \to \infty$.

\begin{remark}
 For the XY model from $\mathbb{T}^1 \to \mathbb{S}^1$, the convergence from the Langevin equation \eqref{eq:Langevin} to the harmonic map heat flow equation \eqref{eq:harmonic} can be shown similarly by taking $\textrm{P}_{x}^\perp(y) = y - (x,y)x$ and the rest of the proof stays the same.
\end{remark}

From \cite{guo1993landau}, for sufficiently regular initial data, there exists a global smooth solution to the Landau-Lifshitz equation
\begin{equation} \label{eq:LL}
 \partial_t u = - \alpha_1 u \times (u \times \Delta u) + \alpha_2 u \times \Delta u
\end{equation}
with periodic boundary conditions, and $\alpha_1, \alpha_2$ some constants. From \cite{weinan2001numerical}, it is known that the finite difference approximation converges to the Landau-Lifshitz equation \eqref{eq:LL}. The harmonic map heat flow equation \eqref{eq:harmonic} only has the dissipation term of the Landau-Lifshitz equation \eqref{eq:LL} and the results  from \cite{weinan2001numerical} should still hold. We will assume a global smooth solution exists for the harmonic map heat flow equation \eqref{eq:harmonic} in all contexts below. The finite difference approximation \eqref{eq:fde} will converge to the harmonic map heat flow equation \eqref{eq:harmonic} as $N \to \infty$ and we only need to compare the SDE system \eqref{eq:SDE} and the ODE system \eqref{eq:fde}.

In the following the error between \eqref{eq:SDE} and \eqref{eq:fde} is calculated. Since $\beta = N^\gamma, \gamma >1$ and $N \to \infty$, we denote 
\begin{equation} \label{eq:epsilon} 
\epsilon \equiv \sqrt{\frac N \beta}
\end{equation}
as a small parameter going to zero with $N \to \infty$. The SDE is then written as
\begin{equation} \label{eq:SDE_ep}
 \df \sigma_i = \textrm{P}_{\sigma_i}^\perp \left( \Delta_N \sigma_i \right) \df t - \epsilon^2 \sigma_i \df t - \epsilon \textrm{P}_{\sigma_i}^\perp \left( \df W_i \right).
\end{equation}

\begin{lemma} \label{lemma:4.1}
  Define the error between the SDE \eqref{eq:SDE_ep} and the ODE \eqref{eq:fde} for the i-th spin as $\tilde{e}_i \equiv \sigma_i - \tilde{\sigma}_i$ and define $e_i \equiv \epsilon^{-p} \tilde{e}_i$ for $0 <p<1$. Define
  \begin{equation} \label{eq:elm41}
   e = \frac 1N \sum_i \|e_i\|^2,
  \end{equation}
  then the following inequality holds  for each realization of the noise
  \begin{equation} \label{eq:lemma4.1}
    e(t) \le \int_0^t \left( C_1 e^{3/2} + C_2 e \right) \df s + C_3 \epsilon^{2-2p} t + 2 \int_0^t \epsilon^{1-p} \frac 1N \sum_i \left( \textrm{P}_{\sigma_i}^{\perp} \left( dW_i(s),e_i(s) \right) \right),
  \end{equation}
  with $\epsilon$ small enough so that $\epsilon^p N^{5/2} \le 1$, or equivalently $\gamma \ge 1 + \frac5p$.
\end{lemma}

\begin{proof}
  Taking the projection given by
  \[
    \textrm{P}_{\sigma_i}^\perp (\Delta \sigma_i) = - \sigma_i \times ( \sigma_i \times \Delta \sigma_i) = \Delta \sigma_i - (\Delta \sigma_i,\sigma_i)\sigma_i
  \]
  together with $\|\sigma_i\|=1$, we have that
  \[
    (\Delta_N \sigma_i, \sigma_i) = - \frac 12 ( \| \nabla_N^+ \sigma_i \|^2 + \| \nabla_N^- \sigma_i \|^2),
  \]
  where $\nabla_N^+ \sigma_i = N(\sigma_{i+1} -\sigma_{i}),\nabla_N^- \sigma_i = N(\sigma_{i} -\sigma_{i-1}) $. The SDE system \eqref{eq:SDE_ep} can then be written as
  \[
  \df \sigma_i = \left( \Delta_N \sigma_i + \frac 12 ( \| \nabla_N^+ \sigma_i \|^2 + \| \nabla_N^- \sigma_i \|^2) \sigma_i \right) \df t - \epsilon^2 \sigma_i \df t - \epsilon \textrm{P}_{\sigma_i}^\perp \left( \df W_i(t) \right)
  \]
  and the ODE system \eqref{eq:fde} can similarly be written as
  \[
  \df \tilde{\sigma}_i = \left( \Delta_N \tilde{\sigma}_i + \frac 12 ( \| \nabla_N^+ \tilde{\sigma}_i \|^2 + \| \nabla_N^- \tilde{\sigma}_i \|^2) \tilde{\sigma}_i \right) \df t.
  \]
  By definition $\tilde{e}_i = \sigma_i - \tilde{\sigma}_i$ satisfies the following equation
  \begin{align*}
    \df \tilde{e}_i  =&  \Delta_N \tilde{e}_i \df t + \frac 12 \left[ 2 ( \nabla_N^+ \tilde{\sigma}_i, \nabla_N^+ \tilde{e}_i) \sigma_i + \| \nabla_N^+ \tilde{e}_i \|^2 \sigma_i + 2 ( \nabla_N^- \tilde{\sigma}_i, \nabla_N^- \tilde{e}_i) \sigma_i + \| \nabla_N^- \tilde{e}_i \|^2 \sigma_i \right]  \df t\\
    & + \frac 12 \left( \| \nabla_N^+  \tilde{\sigma}_i \|^2 + \| \nabla_N^- \tilde{\sigma}_i \|^2 \right) \tilde{e}_i \df t- \epsilon^2 \sigma_i \df t+ \epsilon \textrm{P}_{\sigma_i}^{\perp} \left( dW_i(t) \right).
  \end{align*}
  Since $\tilde{e}_i = \epsilon^p e_i$ with $0<p<1$,
  \begin{dmath} \label{eq:de_i}
    \df e_i  =  \Delta_N e_i \df t + \frac 12 \left[ 2 ( \nabla_N^+ \tilde{\sigma}_i, \nabla_N^+ e_i) \sigma_i + \epsilon^p \| \nabla_N^+ e_i \|^2 \sigma_i + 2 ( \nabla_N^- \tilde{\sigma}_i, \nabla_N^- e_i) \sigma_i \\
    + \epsilon^p \| \nabla_N^- e_i \|^2 \sigma_i \right]  \df t \\
     + \frac 12 \left( \| \nabla_N^+  \tilde{\sigma}_i \|^2 + \| \nabla_N^- \tilde{\sigma}_i \|^2 \right) e_i \df t- \epsilon^{2-p} \sigma_i \df t+ \epsilon^{1-p} \textrm{P}_{\sigma_i}^{\perp} \left( dW_i(t) \right).
  \end{dmath}
Applying It\^o's formula to $\frac 12 \df \| e_i \|^2$, we have that 
  \begin{dmath*}
    \frac 12 \df \| e_i \|^2 
    =  (e_i , \df e_i) + \mathcal{I}_i 
    =  (\Delta_N e_i,e_i) \df t+ \frac 12 \left[ 2 ( \nabla_N^+ \tilde{\sigma}_i, \nabla_N^+ e_i) + \epsilon^p \| \nabla_N^+ e_i \|^2  +  2 ( \nabla_N^- \tilde{\sigma}_i, \nabla_N^- e_i) + \epsilon^p \| \nabla_N^- e_i \|^2 \right] (\sigma_i,e_i)  \df t \\
    + \frac 12 \left( \| \nabla_N^+  \tilde{\sigma}_i \|^2 + \| \nabla_N^- \tilde{\sigma}_i \|^2 \right) \|e_i\|^2 \df t- \epsilon^{2-p} (\sigma_i,e_i) \df t + \epsilon^{1-p} \left( \textrm{P}_{\sigma_i}^{\perp} \left( \df W_i(t),e_i \right) \right) + \mathcal{I}_i,   
  \end{dmath*}
  where $\mathcal{I}_i$ represents the It\^o correction term of order $O(\epsilon^{2-2p})$ as shown in the following computation. 
  
  To calculate the It\^o correction $\mathcal{I}_i$ we consider the SDE systems for both $e_i$ in \eqref{eq:de_i} and $\sigma_i$ in \eqref{eq:SDE}. The It\^o correction $\mathcal{I}_i$ for $\df \|e_i\|^2$ combines three parts corresponding to $\frac{\partial^2 e_i^2}{\partial e_i^2},\frac{\partial^2 e_i^2}{\partial \sigma_i \partial e_i}$ and $\frac{\partial^2 e_i^2}{\partial \sigma_i^2}$. Since $\| \sigma_i\|=1$ we take $\left\| \textrm{P}_{\sigma_i}^{\perp} \left( \df W_i(t) \right) \right\|^2$ as bounded by $C \df t$. The first term, $\frac{\partial^2 e_i^2}{\partial e_i^2}$, is a constant and $\left\| \epsilon^{1-p} \textrm{P}_{\sigma_i}^{\perp} \left( \df W_i(t) \right) \right\|^2 \le C \epsilon^{2-2p} \df t$. The second term, $\frac{\partial^2 e_i^2}{\partial \sigma_i \partial e_i} = \frac{\partial^2 e_i \epsilon^{-p} (\sigma_i - \tilde{\sigma}_i)}{\partial \sigma_i \partial e_i}$, is order $\epsilon^{-p}$ but $\epsilon^{1-p} \textrm{P}_{\sigma_i}^{\perp} \left( \df W_i(t) \right) \cdot \epsilon\textrm{P}_{\sigma_i}^{\perp} \left( \df W_i(t) \right)$ is order $\epsilon^{2-p} \df t$ so the It\^o correction for the second term is also $O(\epsilon^{2-2p} \df t)$. For the third term $\frac{\partial^2 e_i^2}{\partial \sigma_i^2} = \epsilon^{-2p} \frac{\partial^2 (\sigma_i - \tilde{\sigma}_i)^2}{\partial \sigma_i^2}$ but $\left\| \epsilon \textrm{P}_{\sigma_i}^{\perp} \left( \df W_i(t) \right) \right\|^2 \le C \epsilon^{2} \df t$ so the It\^o correction for the third term is also $O(\epsilon^{2-2p}\df t)$.
  
 Returning to $ \frac 12 \df \| e_i \|^2$,   from the periodic boundary conditions, we know that 
  \[
    \sum_{i=1}^N (\Delta_N e_i, e_i) = - \sum_i \| \nabla_N^+ e_i \|^2 = - \sum_i \| \nabla_N^- e_i \|^2,
  \]
  hence summing up $\df \|e_i\|^2$ we have that  
  \begin{align} \label{eq:lm41_2} 
    & \df \left( \sum_i \|e_i\|^2 \right) \nonumber \\
    =&  - \sum_i \left( \| \nabla_N^+ e_i \|^2 + \| \nabla_N^- e_i \|^2 \right) \df t \nonumber \\
    &  +  \sum_i \left[ 2 ( \nabla_N^+ \tilde{\sigma}_i, \nabla_N^+ e_i) + \epsilon^p \| \nabla_N^+ e_i \|^2  +  2 ( \nabla_N^- \tilde{\sigma}_i, \nabla_N^- e_i) + \epsilon^p \| \nabla_N^- e_i \|^2 \right] (\sigma_i,e_i)  \df t \nonumber \\
    &  + \sum_i \left( \| \nabla_N^+  \tilde{\sigma}_i \|^2 + \| \nabla_N^- \tilde{\sigma}_i \|^2 \right) \|e_i\|^2 \df t- 2N \epsilon^{2-p} (\sigma_i,e_i) \df t +2 \sum_i \epsilon^{1-p} \left( \textrm{P}_{\sigma_i}^{\perp} \left( \df W_i,e_i \right) \right) \nonumber \\
    & + 2\sum_i \mathcal{I}_i \nonumber \\
    \le &  - \sum_i \left( \| \nabla_N^+ e_i \|^2 + \| \nabla_N^- e_i \|^2 \right) \df t \nonumber \\
    & +  \sum_i \left[ 2 ( \nabla_N^+ \tilde{\sigma}_i, \nabla_N^+ e_i) + \epsilon^p \| \nabla_N^+ e_i \|^2  +  2 ( \nabla_N^- \tilde{\sigma}_i, \nabla_N^- e_i) + \epsilon^p \| \nabla_N^- e_i \|^2 \right] (\sigma_i,e_i)  \df t \nonumber \\
    &  + \sum_i \left( \| \nabla_N^+  \tilde{\sigma}_i \|^2 + \| \nabla_N^- \tilde{\sigma}_i \|^2 \right) \|e_i\|^2 \df t- 2N \epsilon^{2-p} (\sigma_i,e_i) \df t \nonumber \\
    & +2 \sum_i \epsilon^{1-p} \left( \textrm{P}_{\sigma_i}^{\perp} \left( \df W_i,e_i \right) \right) + c N \epsilon^{2-2p} \df t .
  \end{align}
   For the second term of \eqref{eq:lm41_2}, from the Cauchy-Schwarz inequality, we observe that
   \begin{align*}
    \left| ( \nabla_N^+ \tilde{\sigma}_i, \nabla_N^+ e_i) (\sigma_i,e_i) \right| & \le \frac 12 \left(  \| \nabla_N^+ e_i \|^2 + \|\nabla_N^+ \tilde{\sigma}_i \|^2 \|e_i\|^2 \right) ,  \\
    \left| ( \nabla_N^- \tilde{\sigma}_i, \nabla_N^- e_i) (\sigma_i,e_i) \right| & \le \frac 12 \left(  \| \nabla_N^- e_i \|^2 + \|\nabla_N^- \tilde{\sigma}_i \|^2 \|e_i\|^2 \right).
   \end{align*}
   Since the solution for the harmonic map heat flow equation is smooth, in the third term of \eqref{eq:lm41_2},  $\| \nabla_N^+ \tilde{\sigma}_i \|^2 + \| \nabla_N^- \tilde{\sigma}_i\|^2$ can be bounded by some constant $C$. For the fourth term in \eqref{eq:lm41_2} $ | \epsilon^{2-p} (\sigma_i,e_i) | = | \epsilon^{2-2p} (\sigma_i, \epsilon^p e_i) | \le 2 \epsilon^{2-2p}$ since $\epsilon^p e_i = \sigma_i -\tilde{\sigma}_i$. Hence from \eqref{eq:lm41_2},
   $$\begin{aligned}
    \sum_i \|e_i(t)\|^2 
   & \le \int_0^t  \epsilon^p \sum_i \left( \|\nabla_N^+ e_i(s) \|^2 + \|\nabla_N^- e_i(s) \|^2 \right) \|e_i(s)\|  \df s  \\
   & + C_1 \int_0^t  \sum_i \|e_i(s)\|^2 \df s + C_2 N \epsilon^{2-2p}t 
    + 2 \int_0^t \epsilon^{1-p}  \sum_i \left( \textrm{P}_{\sigma_i}^{\perp} \left( dW_i(s),e_i(s) \right) \right).
   \end{aligned}$$
  Using the assumption that $\epsilon^p N^{5/2} \le 1$, (or equivalently $\gamma \ge 1 + 5/p$,) we bound the $ \sum_i \left( \epsilon^p \| \nabla_N^+ e_i \|^2 + \epsilon^p \| \nabla_N^- e_i \|^2 \right) \|e_i\|$ term by $ C \sum_i \frac{1}{\sqrt{N}} \|e_i\|^3$. Hence,
  \begin{equation}
   \frac 1N \sum_i \left( \epsilon^p \| \nabla_N^+ e_i \|^2 + \epsilon^p \| \nabla_N^- e_i \|^2 \right) \|e_i\| \le C \frac{1}{\sqrt{N^3}} \sum_i \|e_i\|^3 \le  C \left( \frac 1N \sum_i \|e_i\|^2 \right)^{\frac 32}
  \end{equation}
  as the p-norm is decreasing.  Finally, since $e = \frac 1N \sum_i \|e_i\|^2 $, we arrive at \eqref{eq:lemma4.1}.
  
  \end{proof}

\begin{remark}
  If we choose the parameters such that the small It\^o correction term, $C N \epsilon^{2-2p} t$,  from the martingale  is of lower order in $N$ and $\epsilon^p N^3 = O(1)$, (or equivalently $\gamma > 1 + \frac3p$,) we could show a similar result for $e= \sum_i \| e_i \|^2$ instead of $e=\frac 1N \sum_i \| e_i \|^2$. We could then bound $\E{\sum_i \|e_i\|^2 }$.
\end{remark}

Intuitively the inequality from Lemma \ref{lemma:4.1} is of Gr\"onwall type and the martingale part,
\[ 2 \int_0^t \epsilon^{1-p} \frac 1N \sum_i \left( \textrm{P}_{\sigma_i}^{\perp} \big( dW_i(s),e_i(s) \big) \right) \] has a small coefficient $\epsilon^{1-p}$. This implies that with large probability
\[
e(t) \le \int_0^t \left( C_1 e^{3/2}(s) + C_2 e(s) \right) \df s + C_3 \epsilon^{2-2p} t +\textrm{small term from martingale}
\]
and $e(t)$ is bounded for a long time interval. We use this idea to show the following proposition.
\begin{proposition} \label{prop:4.1}
  \begin{equation}
   \mathbb{P} \left( e(t) \le \epsilon^{2-2p} \exp(Ct) \right) \ge 1-e^{-\frac N2}, \quad t \in [0,T]
  \end{equation}
  for some constant $C$ and $T$ satisfying $\epsilon^{1-p} e^{\frac 12 CT} \le 1$.
\end{proposition}

To prove Proposition \ref{prop:4.1}, we will use the exponential martingale inequality for continuous $L^2$ martingale $M_t$ as in \cite[p.~25]{hairer2008spectral} or \cite{bass1998diffusions} :
\begin{equation} \label{eq:Mineq}
  \mathbb{P} ( \sup_t  M_t - \tfrac{a}{2} \langle M \rangle_t > b) \le e^{- ab }
\end{equation}
where $\langle M \rangle_t$ is the quadratic variation for $M_t$.

For the martingale in  \eqref{eq:lemma4.1}, we calculate its quadratic variation in the following lemma.
\begin{lemma}
  The quadratic variation for the martingale \[ M_t = 2 \int_0^t \epsilon^{1-p} \frac 1N \sum_i \left( \textrm{P}_{\sigma_i}^{\perp} \left( dW_i(s),e_i(s) \right) \right) \] is
  \begin{equation}
   \langle M \rangle_t = 4 \int_0^t \frac{\epsilon^{2-2p}}{N^2} \sum_i \left[ \|e_i(s)\|^2 - (\sigma_i(s),e_i(s))^2 \right]\df s.
  \end{equation}
\end{lemma}

\begin{proof}
  The quadratic variation for $M_t$ is captured by a direct summation of the square of the coefficients of the white noise, namely
  $$\begin{aligned}
   \bigg\langle 2 \int_0^t \epsilon^{1-p} &\frac 1N \sum_i \left( \textrm{P}_{\sigma_i}^{\perp} \left( dW_i(s),e_i(s) \right) \right) \bigg\rangle_t 
    =  \frac{4 \epsilon^{2-2p}}{N^2} \left \langle \int_0^t \sum_i \left( \textrm{P}_{\sigma_i}^\perp (\df W_i(s)),e_i(s) \right) \right \rangle_t  \\
    &=   \frac{4 \epsilon^{2-2p}}{N^2} \left \langle \int_0^t \sum_i \df W_i(s) \cdot e_i(s) - (\df W_i(s), \sigma_i(s)) \sigma_i(s) \cdot e_i(s) \right \rangle_t \\
    &=   \frac{4\epsilon^{2-2p}}{N^2} \left \langle \int_0^t \sum_i \sum_{\zeta=x,y,z} \left( e_i^\zeta(s)-(\sigma_i(s),e_i(s))\sigma_i^\zeta(s) \right) \df W_i^\zeta(s) \right \rangle_t \\
    &=  \frac{4\epsilon^{2-2p}}{N^2} \int_0^t \sum_i \sum_{\zeta=x,y,z} \left( e_i^\zeta(s)-(\sigma_i(s),e_i(s))\sigma_i^\zeta(s) \right)^2 \df s \\
    &=  \frac{4\epsilon^{2-2p}}{N^2} \int_0^t \sum_i \sum_{\zeta=x,y,z} \bigg[  (e_i^\zeta(s))^2+(\sigma_i(s),e_i(s))^2 (\sigma_i^\zeta(s))^2 \\
    & \qquad - 2 e_i^\zeta(s) \sigma_i^\zeta(s) (\sigma_i(s),e_i(s))\bigg] \df s \\
   & =  \frac{4\epsilon^{2-2p}}{N^2} \int_0^t \sum_i \bigg[ \|e_i(s)\|^2 + (\sigma_i(s),e_i(s))^2 \|\sigma_i(s)\|^2 \\
   & \qquad-2 (\sigma_i(s),e_i(s))(\sigma_i(s),e_i(s)) \bigg]\df s\\
    &=  \frac{4\epsilon^{2-2p}}{N^2} \int_0^t \sum_i \bigg[ \|e_i(s)\|^2 - (\sigma_i(s),e_i(s))^2 \bigg]\df s.
 \end{aligned}$$

\end{proof}

Now taking $a = \frac{N}{2 \epsilon^{2-2p}}$ and $b = \epsilon^{2-2p}$ in the inequality \eqref{eq:Mineq}, we have that 
\begin{dmath} 
  \mathbb{P} \left( \sup_{0 \le s \le t} 2 \int_0^s \epsilon^{1-p} \frac 1N \sum_i  \textrm{P}_{\sigma_i}^{\perp} \left( dW_i(u),e_i(u)  \right)  -  \int_0^s \frac 1N \sum_i \left(\|e_i(u)\|^2 - (\sigma_i(u),e_i(u))^2 \right) \df u > \epsilon^{2-2p} \right) \le e^{- N/2 }.
\end{dmath}
Thus, with probability $ \ge 1 - e^{-N/2}$, the following inequality holds
\begin{align*}
& 2 \int_0^t \epsilon^{1-p} \frac 1N \sum_i \left( \textrm{P}_{\sigma_i}^{\perp} \left( dW_i(s),e_i(s) \right) \right) - \int_0^t \frac 1N \sum_i \left(\|e_i(s)\|^2 - (\sigma_i(s),e_i(s))^2 \right) \df s \\
& \hspace{0.5cm}  \le  \sup_{0 \le s \le t} 2 \int_0^s  \frac{\epsilon^{1-p}}{N} \sum_i \left( \textrm{P}_{\sigma_i}^{\perp} \left( dW_i(u),e_i(u) \right) \right) \df u  \\
& \hspace{0.75cm} -  \int_0^s \frac 1N \sum_i \left(\|e_i(u)\|^2 - (\sigma_i(u),e_i(u))^2 \right) \df u \\
& \hspace{0.5cm}  \le  \epsilon^{2-2p}.
\end{align*}
Combining this with \eqref{eq:lemma4.1}, we observe that
\begin{align} \label{eq:lemma4.2}
  e(t) 
  &\le  \int_0^t \left( C_1 e^{3/2}(s) + C_2 e(s) \right) \df s + C_3 \epsilon^{2-2p} t + 2 \int_0^t \epsilon^{1-p} \frac 1N \sum_i \left( \textrm{P}_{\sigma_i}^{\perp} \left( dW_i(s),e_i(s) \right) \right) \nonumber\\
  &  \le  \int_0^t \left( C_1 e^{3/2}(s) + C_2 e(s) \right) \df s + C_3 \epsilon^{2-2p} t \nonumber\\
  &\qquad + \int_0^t \frac 1N \sum_i \left(\|e_i(s)\|^2 -
  (\sigma_i(s),e_i(s))^2 \right) \df s + \epsilon^{2-2p} \nonumber\\
  & \le  \int_0^t \left( C_1 e^{3/2}(s) + (C_2+1) e(s) \right) \df s + C_3 \epsilon^{2-2p} t + \epsilon^{2-2p},
\end{align}
where we have used the definition of $e(t)$ and that $(\sigma_i(t),e_i(t))^2 \le \|e_i(t)\|^2$.

Since \eqref{eq:lemma4.2} is a Gr\"onwall type inequality, we build a special upper solution $u(t) = \epsilon^{2-2p} e^{ (C_1+C_2+C_3+1) t}$. Then
\[ \frac{\df u}{\df t} = (C_1+C_2+C_3+1) u(t) \ge C_1 u^{3/2}(t) + (C_2+1)u(t) + C_3 \epsilon^{2-2p}, \]
where $C_1 u(t) \ge C_1 u^{3/2}(t)$ when $\epsilon^{1-p} e^{ \frac 12 (C_1+C_2+C_3+2) t} \le 1$. We observe that 
\[ e^{(C_1+C_2+C_3+1) t } \ge 1 + (C_1+C_2+C_3+1) t, \]
so  $ C_3 u(t) \ge C_3 \epsilon^{2-2p}$. As $u(0) = \epsilon^{2-2p}$, we have that
\begin{equation}
  u(t) \ge \int_0^t \left( C_1 u^{3/2}(s) + (C_2+1) u(s) \right) \df s + C_3 \epsilon^{2-2p} t + \epsilon^{2-2p}
\end{equation}
and $u(t)$ is an upper bound for $e(t)$.

\subsection{Proof of Theorem \ref{thm:2}}

When $\epsilon$ is small enough we have that $\epsilon^{1-p} e^{ \frac 12 (C_1+C_2+C_3+2) T} \le 1$ and $u(t)$ is an upper bound for $e(t)$. We observe that with $u(t) =  \epsilon^{2-2p} e^{ (C_1+C_2+C_3+1) t}$,
\begin{equation}
  \mathbb{P} \left( e(t) \le u(t)  \right) \ge 1 - e^{-N/2} , \quad t \in [0,T]
\end{equation}
and the inequality in Proposition~\ref{prop:4.1} follows.

In fact, as $e_i(t) = \epsilon^{-p} ( \sigma_i(t) - \tilde{\sigma}_i(t))$, the bound $e(t) = \frac 1N \sum_i \|e_i(t)\|^2 \le C$ by some constant $C$ is sufficient. Appealing to a stopping time argument similar to the strong uniqueness proof of the SDE with locally Lipschitz continuous coefficients (see e.g.\ \cite[Chapter~5.2,Theorem~2.5]{karatzas2012brownian}), we have the following result.
\begin{proposition}
Given $\epsilon,N > 0$, there exists a constant $C$ independent of $\epsilon,N$ such that if $\epsilon^{2-2p} T e^{CT} \le 1$, then 
   \[
  \E[]{e(t)} \le 1, \quad t \in [0,T].
  \]
  
  \end{proposition}

\begin{proof}
  Define a deterministic time $T \equiv \min \{ t \in [0,\infty]: \E[]{e(t)} \ge 1 \}$ and a stopping time $\tau \equiv \min \{ t \in [0,T]: e(t) \ge 1\}$. Since $e(0)=0$, we have both $T>0$ and $\tau>0$. If $T$ is infinite then we are done, so assume $T$ is bounded.

  From Lemma \ref{lemma:4.1},
  \begin{equation}
    e(t \wedge \tau) \le \epsilon^{2-2p} (t \wedge \tau) + \int_0^{t \wedge \tau} C_1 e^{3/2}(s) + C_2 e(s) \df s+  \int_0^{t \wedge \tau}   \frac{ \epsilon^{1-p}}{N} \sum_i \textrm{P}_{\sigma_i}^\perp \left( \df W_i(s) \right).
  \end{equation}
  As $e\ge 0$, we have that
  \begin{equation}
    e(t\wedge \tau) \le \epsilon^{2-2p} (t \wedge \tau) + \int_0^t C_1 e^{3/2}(s \wedge \tau) + C_2 e(s \wedge \tau) \df s +  \int_0^{t \wedge \tau} \epsilon^{1-p} \frac 1N \sum_i \textrm{P}_{\sigma_i}^\perp \left( \df W_i(s)\right)
  \end{equation}
  and $ e^{3/2} (s \wedge \tau) \le e(s \wedge \tau)$ as $e(s \wedge \tau) \le 1$, so
  \begin{equation}
    e(t \wedge \tau) \le \epsilon^{2-2p} (t \wedge \tau) + \int_0^t (C_1+C_2) e(s \wedge \tau) \df s +  \int_0^{t \wedge \tau} \epsilon^{1-p} \frac 1N \sum_i \textrm{P}_{\sigma_i}^\perp \left( \df W_i(s) \right).
  \end{equation}
  Taking the expectation on both sides, we arrive at the fact that
  \begin{equation}
    \E[]{e(t \wedge \tau)} \le \epsilon^{2-2p} \E{(t \wedge \tau)} + \int_0^t (C_1+C_2) \E[]{e(s \wedge \tau)} \df s .
  \end{equation}
 The expectation $\E[]{\int_0^{t \wedge \tau} \textrm{P}_{\sigma_i}^\perp \left( \df W_i(s) \right)} = 0$ can be deduced from the optional stopping theorem for continuous time or take $\int_0^{t\wedge \tau} \textrm{P}_{\sigma_i}^\perp \left( \df W_i(s) \right) = \int_0^{T} \mathbbm{1}(s< t \wedge \tau) \textrm{P}_{\sigma_i}^\perp \left( \df W_i(s) \right)$.
  Notice $\E[]{(t \wedge \tau)} \le T$, so
  \[
  \E[]{e(t \wedge \tau)} \le \int_0^t (C_1+C_2) \E[]{e(s \wedge \tau)} \df s + \epsilon^{2-2p} T
  \]
  and by a Gr\"onwall's inequality,
  \begin{equation}
    \E[]{e(t\wedge \tau)} \le \epsilon^{2-2p} T e^{(C_1+C_2)t}.
  \end{equation}
  Choosing $C > C_1 + C_2$, the result follows.  
  
  \end{proof}

By similar arguments, a uniform bound on $e(t)$ can be obtained with a weaker condition on $T$.
\begin{proposition}
Given $\epsilon, N > 0$, there exist constants $\tilde{C}_1,\tilde{C}_2$ independent of $\epsilon,N$ such that if $(\epsilon^{2-2p} T + \tilde{C}_1 \frac{\epsilon^{2-4p}}{N} ) e^{\tilde{C}_2T} \le 1$, for $0<p<\tfrac12$, then   
  \[
  \E[]{\sup_{0\le s\le t} e(s)} \le 1, \quad t \in [0,T].
  \]

\end{proposition}

\begin{proof}
  Define $T \equiv \min \{t: \E[]{\sup_{0\le s\le t} e(s)} \ge 1 \}$ and $\tau \equiv \min \{ t \in [0,T]: e(t) \ge 1\}$. Since $e(0)=0$ we still have $T>0,\tau>0$.

  Again from Lemma \ref{lemma:4.1}, we observe that
  \begin{dmath}
    e(t\wedge \tau) \le \int_0^t \left[ C_1 e^{3/2}(s \wedge \tau) + C_2 e(s \wedge \tau) \right]\df s + \epsilon^{2-2p} (t \wedge \tau) 
    + \int_0^{t \wedge \tau}
    \epsilon^{1-p} \frac 1N \sum_i  \left( \textrm{P}_{\sigma_i}^\perp \left( \df W_i(s) \right),e_i(s) \right).
  \end{dmath}
  Taking the supremum on both sides, we arrive at the fact that
  \begin{equation}\begin{aligned}
    \sup_{0 \le s \le t} e(s \wedge \tau) 
   & \le  \sup_{0 \le s \le t} \int_0^s \left[ C_1 e^{3/2}(u \wedge \tau) + C_2 e(u \wedge \tau) \right]\df u +  \sup_{0 \le s \le t} \epsilon^{2-2p} (s \wedge \tau) \\
   &  + \sup_{0 \le s \le t}  \int_0^{s \wedge \tau} \epsilon^{1-p} \frac 1N \sum_i
    \left( \textrm{P}_{\sigma_i}^\perp \left( \df W_i(u) \right),e_i(u) \right).
 \end{aligned}\end{equation}
  The above, together with $0 \le e(t\wedge \tau) \le 1$ and $t\wedge \tau \le T$, gives
  \begin{dmath}
    \sup_{0 \le s \le t} e(s \wedge \tau) \le \int_0^t (C_1 + C_2) e(u \wedge \tau) \df u +
    \epsilon^{2-2p} T + \sup_{0 \le s \le t}  \int_0^{s \wedge \tau} \epsilon^{1-p} \frac 1N \sum_i
    \left( \textrm{P}_{\sigma_i}^\perp \left( \df W_i(u) \right),e_i(u) \right).
  \end{dmath}
  Taking expectations on both sides, we arrive at the fact that
  \begin{multline} \label{eq:prop4.3}
    \E[]{\sup_{0 \le s \le t} e(s \wedge \tau)} 
    \le \int_0^t (C_1 + C_2) \E[]{e(u \wedge \tau)} \df u 
\\+ \epsilon^{2-2p} T + \E[]{\sup_{0 \le s \le t}  \int_0^{s \wedge \tau} \epsilon^{1-p} \frac 1N \sum_i \left( \textrm{P}_{\sigma_i}^\perp \left( \df W_i(u) \right),e_i(u) \right)}.
  \end{multline}
  For the first integral on the right hand side, we note that $\int_0^t (C_1 + C_2) \E[]{e(u \wedge \tau)} \df u \le \int_0^t (C_1 + C_2) \E{\sup_{0 \le s \le u} e(s \wedge \tau)} \df u$. Doob's Martingale inequality \eqref{eq:Doob} is used to give a bound of $C\epsilon^{1-2p}$ for the last expecation in \eqref{eq:prop4.3}.

  Denote
  \begin{equation*}
    M_s = \int_0^{s \wedge \tau}  \frac{\epsilon^{1-p}}{N} \sum_i \left( \textrm{P}_{\sigma_i}^\perp \left( \df W_i(u) \right),e_i(u) \right)  = \int_0^s \mathbbm{1}(u \le \tau)\frac{ \epsilon^{1-p} }{N} \sum_i \left( \textrm{P}_{\sigma_i}^\perp \left( \df W_i(u) \right),e_i(u) \right).
  \end{equation*}
  Since $\mathbbm{1}(u \le \tau) \in \mathcal{F}_u$, the above is a martingale. Using Doob's martingale inequality \eqref{eq:Doob} we have that
  \begin{equation*}
    \E[]{ \left( \sup_{0\le s \le t} M_s \right)^2} \le 4 \E[]{M_t^2} = 4 \int_0^t \E[]{ \frac{\epsilon^{2-2p}}{N^2} \sum_i \|e_i(u)\|^2 - (\sigma_i(u),e_i(u))^2  } \df u\le C \frac{\epsilon^{2-4p}}{N}.
  \end{equation*}
  The second equality is from It\^o isometry and the third inequality is because $\|\sigma_i\|,\|\tilde{\sigma}_i\|=1$ so $|(\sigma_i,e_i)| \le \|\sigma_i\|$ and $\|\epsilon^p e_i \| = \| \sigma_i - \tilde{\sigma}_i \| \le 2$.

  Now a Gr\"onwall's inequality gives
  \begin{equation}
    \E[]{\sup_{0 \le s \le t} e(s \wedge \tau)} \le \int_0^t (C_1 + C_2) \E[]{\sup_{0 \le s \le u} e(s \wedge \tau)} \df u + \epsilon^{2-2p} T +  C_3 \frac{\epsilon^{2-4p}}{N}
  \end{equation}
  and $ \E[]{\sup_{0 \le s \le t} e(s \wedge \tau)} \le \left( \epsilon^{2-2p} T + C_3 \frac{\epsilon^{2-4p}}{N} \right) e^{(C_1+C_2)t}$. As in the proof of the previous Proposition, choosing $\tilde{C}_2 > C_1 + C_2$ and $\tilde{C}_1 = C_3$, and assuming $0<p<\tfrac12$, the result thus follows similarly.
  
  \end{proof}

%
\section{Numerical results\label{sec:numerical}}
%

In this section, we present results from numerical simulations of the systems, showing both the temporal dynamics and the order of convergence. 
  The numerical convergence tests indicate that the error decays at least as well as predicted in Theorems \ref{thm:1} and \ref{thm:2}.   We check both the cases of the XY model and the classical Heisenberg model.  The dynamics of the M-H algorithm are simulated as explained in Sec.~\ref{sec:MHsetup}. To simulate the SDE \eqref{eq:SDE}, written in the It\^o sense, we use the stochastic Euler's method combined with a normalizing step to project the spin back onto the sphere after each time step for both the XY model and the classical Heisenberg model.   The PDE \eqref{eq:LLG} is numerically integrated by discretizing in space and using the Euler's method presented in \cite{weinan2001numerical} which includes a normalization step.

The out-of-equilibrium to equilibrium dynamics of the M-H algorithm, SDE, and discretized PDE are shown in Figure \ref{fig:dynamics}.  Figure \ref{fig:dynamics1} shows the $\mathbb{T}^1 \to \mathbb{S}^1$ case of the XY model in terms of the polar coordinate $\theta$ of each spin.  Figure \ref{fig:dynamics2} shows the $\mathbb{T}^1 \to \mathbb{S}^2$ case of the classical Heisenberg model with each spin plotted on the same unit sphere; nearest neighbors are connected by a solid line.  In both cases, the M-H dynamics tend to lag behind the SDE and PDE which more closely follow each other.  This suggests the error between the M-H algorithm and the PDE is dominated by the error between the M-H algorithm and the SDE.  Thus the order of convergence between M-H algorithm and harmonic map heat flow equation should almost follow the order of convergence in Theorem \ref{thm:1}.

\begin{figure}[t]
	\centering
	\begin{subfigure}{\textwidth}
		\includegraphics[width=\textwidth]{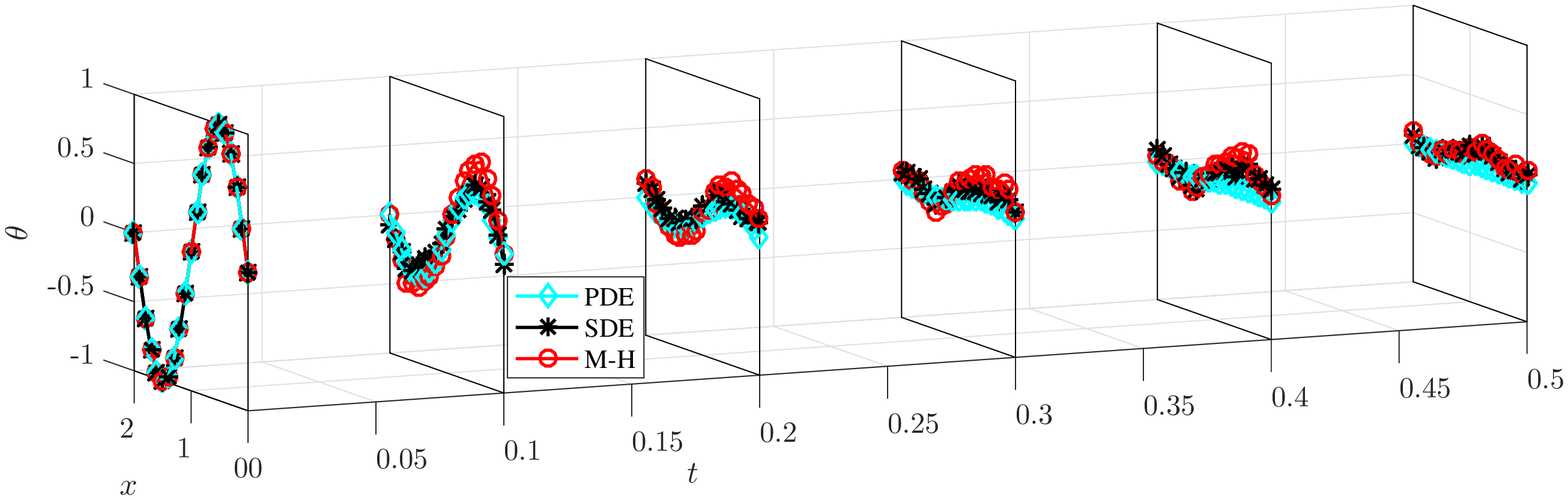}
		\caption{ $\mathbb{T}^1 \to \mathbb{S}^1$ }
		\label{fig:dynamics1}
	\end{subfigure}
	
	\begin{subfigure}{\textwidth}
		\includegraphics[width=\textwidth]{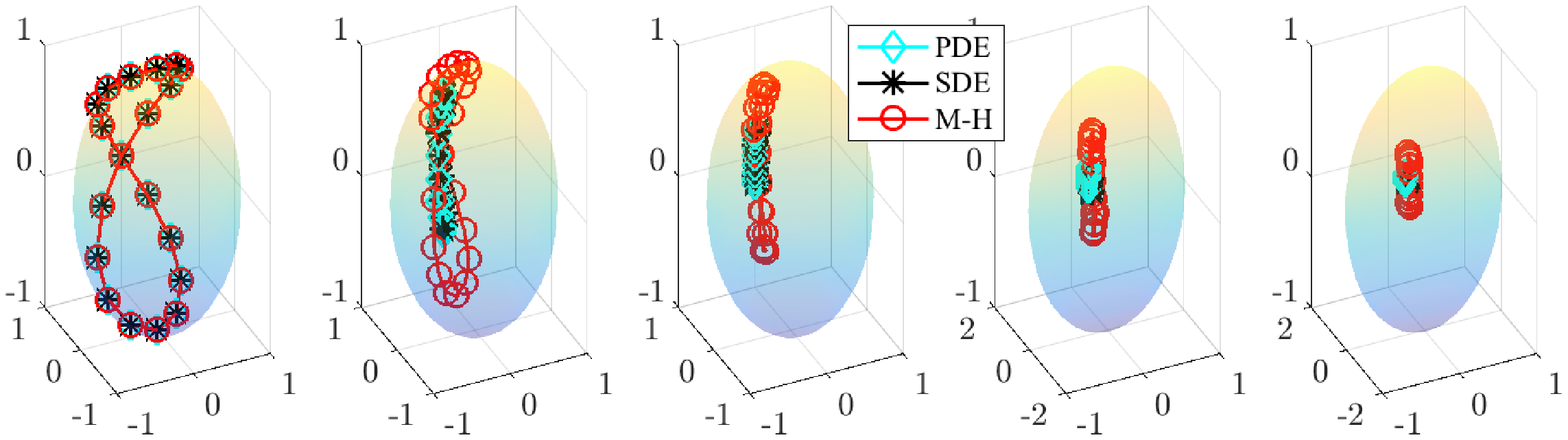}
		\caption{$\mathbb{T}^1 \to \mathbb{S}^2$}
		\label{fig:dynamics2}
	\end{subfigure}
	\caption{Dynamics of M-H algorithm (red circles), Langevin equation (black stars) and harmonic map heat flow equation (cyan diamonds) at various instances of time. They follow each other to converge to equilibrium. In both panels: lattice length $L=2$, space discretization $\delta x = \frac 1N = \frac{1}{10}$, time step size for M-H algorithm $\delta t=\frac{1}{N^3}= 0.001$, inverse temperature $\beta = N^{3/2}\approx 31.6$, proposal size $\varepsilon = \sqrt{\frac{N\delta t}{\beta}}\approx 0.0178$.}
	\label{fig:dynamics}
\end{figure}

Figures \ref{fig:order1} and \ref{fig:order2} show the order of convergence for the error between the M-H algorithm and the Langevin equation with respect to the time step size $\delta t$, for which the equivalent M-H proposal size is  $\varepsilon^2 = \delta t \frac{N}{J\beta} $.  The error is calculated at a fixed time $T$ as
\begin{equation}
\mathbb{E}\left[ \sqrt{ \frac{1}{N}\sum_{i=1}^N | \sigma_i^{\textrm{MH}}(T) - \sigma_i^{\textrm{SDE}}(T) |^2 } \;\right]
\end{equation} 
where the expectation is taken over multiple realizations.  All four frames support that the convergence is at least as good as $\delta t^{1/4}$, which is equivalent to the $\sqrt{\delta t}$ convergence given in Theorem \ref{thm:1}, since the 2-norm is used in the numerical experiments (thus the error is expected to be of order $(\sqrt{\delta t})^{\frac 12} = \delta t^{\frac 14}$).  The faster convergence of order $\delta t^{1/2}$ in panels (b) and (c) of Fig.~\ref{fig:order1} we suspect is due to the fact that these out-of-equilibrium dynamics are dominated by the deterministic part of the SDE, and this part has different error scaling from the noisy dynamics.  In equilibrium, the deterministic term, $P_{\sigma_i}^{\perp}(\delta_N\sigma_i)$, is small since it is zero at the minimum of the Hamiltonian (maximum of the Gibbs distribution), and the noisy part of the dynamics dominate.

Proposition \ref{prop:theta} states the error on the deterministic drift of one Metropolis step, $\theta_i^n$, is of size $\epsilon^3$.  Dividing by a time-step $\delta t$ that is proportional to $\epsilon^2$ so that the left-hand side approximates a derivative for the SDE, the resulting error is $O(\epsilon)$ or equivalently $O(\delta t^{1/2})$ as seen in the numerical simulations.  Similarly, from Proposition \ref{prop:phi} the error on the stochastic diffusion of one Metropolis step, $\phi_i^n$, is of size $\epsilon^{3/2}$ implying error of order $O(\delta t^{1/4})$, after dividing by the size of the first order term, $\epsilon$.  To further test if the difference in convergence order is from the deterministic terms dominating, we increase the size of the noise, $\sqrt{\frac N \beta}$,  in Eq.~\eqref{eq:SDE} by decreasing $\beta$ to make the noisy dynamics dominate.  The out-of-equilibrium error with small $\beta=1$ shown in Fig.~\ref{fig:order2} has $\delta t^{1/4}$ convergence, confirming the original statement of Theorem \ref{thm:1}.  We therefore conclude that the error bound of $\delta t^{1/4}$ is tight, and this error comes from the noisy part of the dynamics.

\begin{figure}[t]
	\centering
	\begin{subfigure}{0.475\textwidth}
		\includegraphics[width=\textwidth]{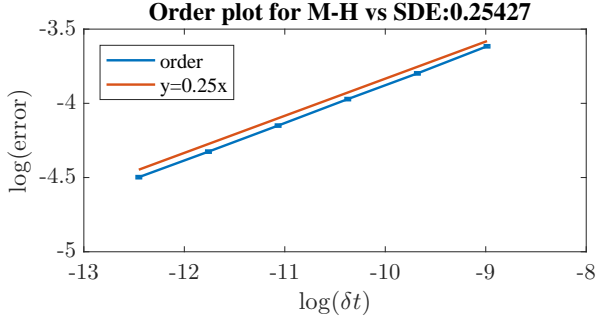}
		\caption{$\mathbb{T}^1 \to \mathbb{S}^1$: near equilibrium initial condition}
		\label{fig:order1a}
	\end{subfigure}
	\hfill
	\begin{subfigure}{0.475\textwidth}
		\includegraphics[width=\textwidth]{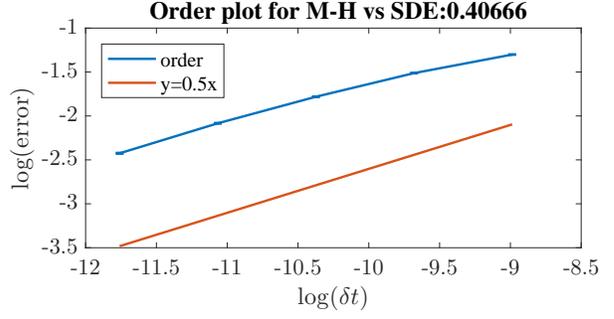}
		\caption{$\mathbb{T}^1 \to \mathbb{S}^1$: out of equilibrium initial condition}
		\label{fig:order1b}
	\end{subfigure}
	\vskip\baselineskip
	\begin{subfigure}{0.475\textwidth}
		\includegraphics[width=\textwidth]{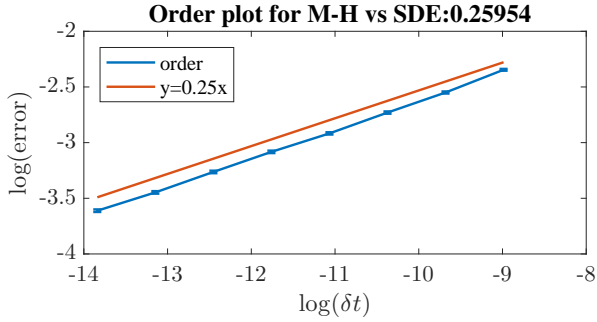}
		\caption{$\mathbb{T}^1 \to \mathbb{S}^2$: near equilibrium initial condition}
		\label{fig:order1c}
	\end{subfigure}
	\hfill
	\begin{subfigure}{0.475\textwidth}
		\includegraphics[width=\textwidth]{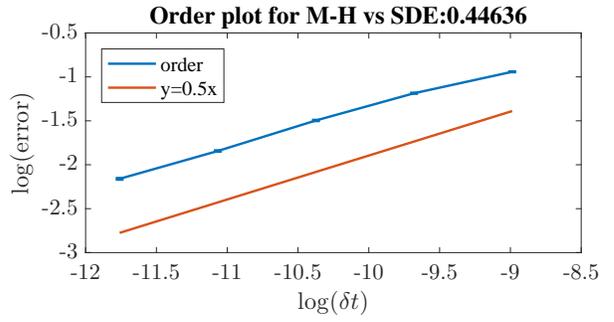}
		\caption{$\mathbb{T}^1 \to \mathbb{S}^2$: out of equilibrium initial condition}
		\label{fig:order1d}
	\end{subfigure}
	\caption{Order of convergence for the error between M-H algorithm and Langevin equation with respect to time step size $\delta t = \frac{J\beta}{N^2} \varepsilon^2$ for $\beta = N^{3/2}$. When the initial condition is near equilibrium, the order of convergence is approximately 0.25 as predicted in Theorem \ref{thm:1}. When the initial condition is out of equilibrium, the order is better than 0.25 and close to 0.5. In all four panels: lattice length $L=2$, space discretization $\delta x = \frac 1N = \frac{1}{10}$, time step size for M-H algorithm $\delta t=\frac{1}{N^3}$, inverse temperature $\beta = N^{3/2}$, proposal size $\varepsilon = \sqrt{\frac{N\delta t}{\beta}}$.}
	\label{fig:order1}
\end{figure}

\begin{figure}[h!]
	\centering
	\begin{subfigure}{0.475\textwidth}
		\includegraphics[width=\textwidth]{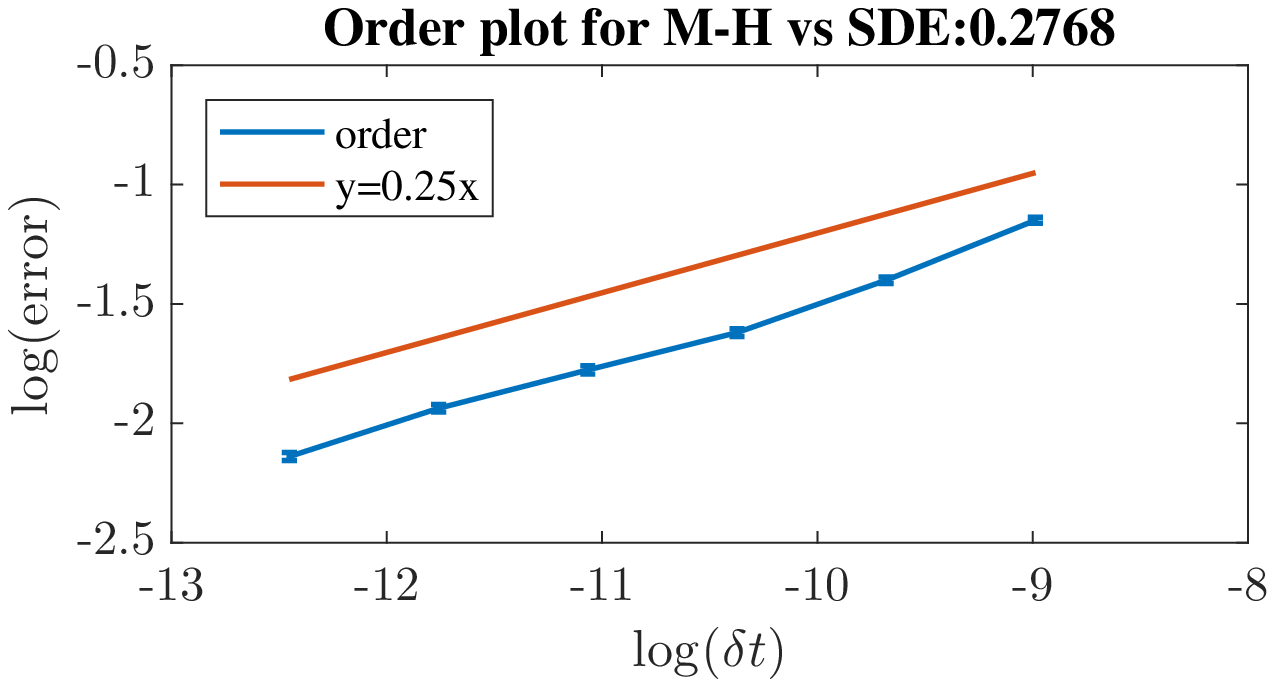}
		\caption{$\mathbb{T}^1 \to \mathbb{S}^1$}
		\label{fig:2dsde}
	\end{subfigure}
	\hfill
	\begin{subfigure}{0.475\textwidth}
		\includegraphics[width=\textwidth]{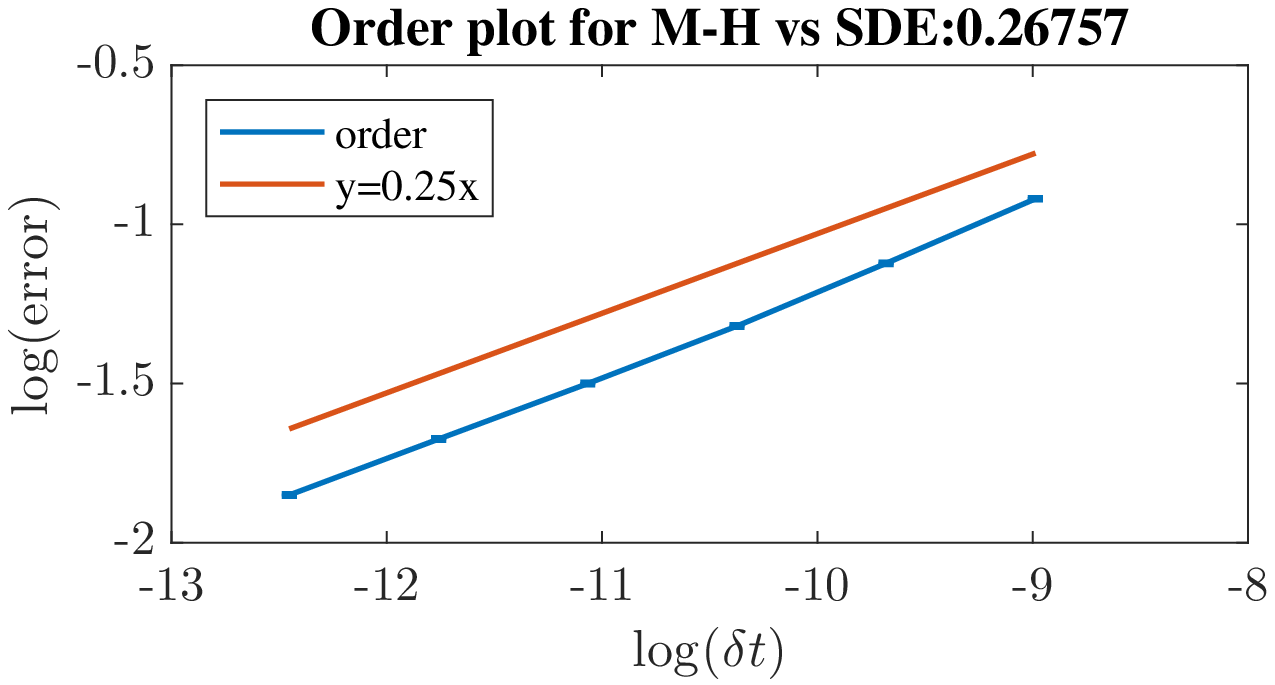}
		\caption{$\mathbb{T}^1 \to \mathbb{S}^2$}
		\label{fig:3dsde}
	\end{subfigure}
	
	\caption{Order of convergence for the error between M-H algorithm and Langevin equation with respect to time step size $\delta t = \frac{J\beta}{N^2} \varepsilon^2$ for $\beta = 1$. The order of convergence is approximately 0.25 as predicted in Theorem \ref{thm:1} with out of equilibrium initial condition. In both panels: lattice length $L=2$, space discretization $\delta x = \frac 1N = \frac{1}{10}$, time step size for M-H algorithm $\delta t=\frac{1}{N^3}$, inverse temperature $\beta = 1$, proposal size $\varepsilon = \sqrt{\frac{N\delta t}{\beta}}$.}
	\label{fig:order2}
\end{figure}

\begin{figure}[h]
	\centering
	\begin{subfigure}{0.475\textwidth}
		\includegraphics[width=\textwidth]{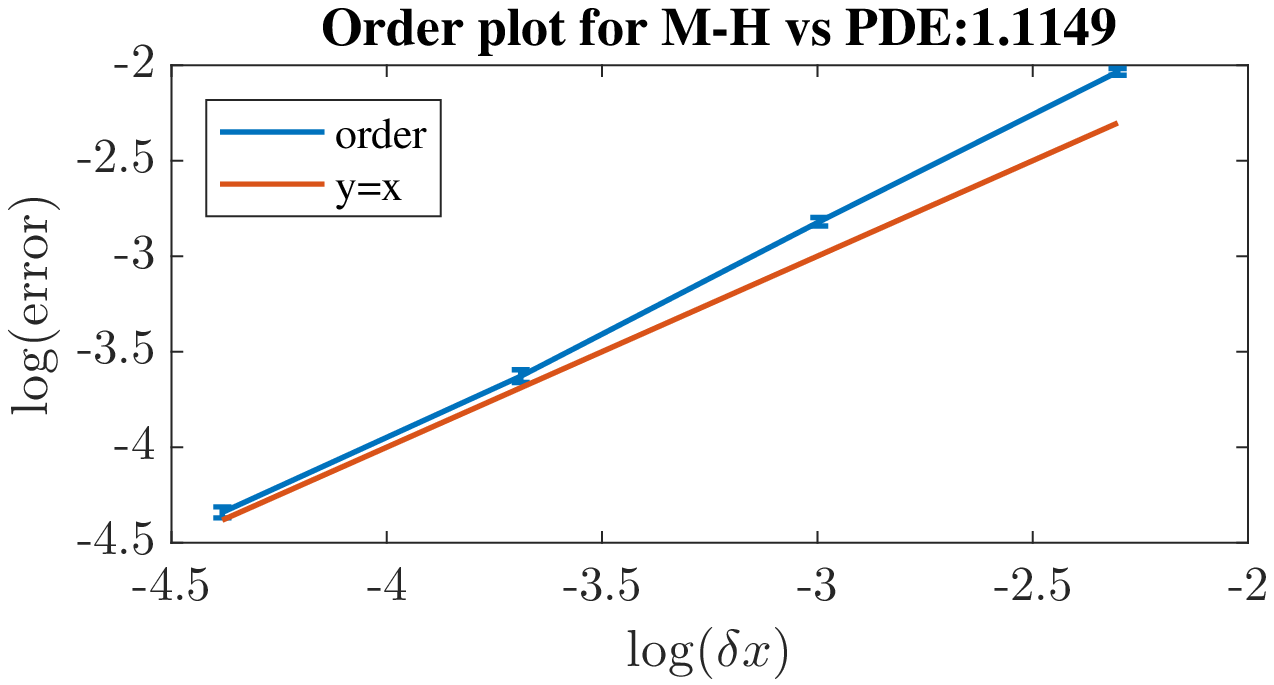}
		\caption{$\mathbb{T}^1 \to \mathbb{S}^1$}
		\label{fig:order3a}
	\end{subfigure}
	\hfill
	\begin{subfigure}{0.475\textwidth}
		\includegraphics[width=\textwidth]{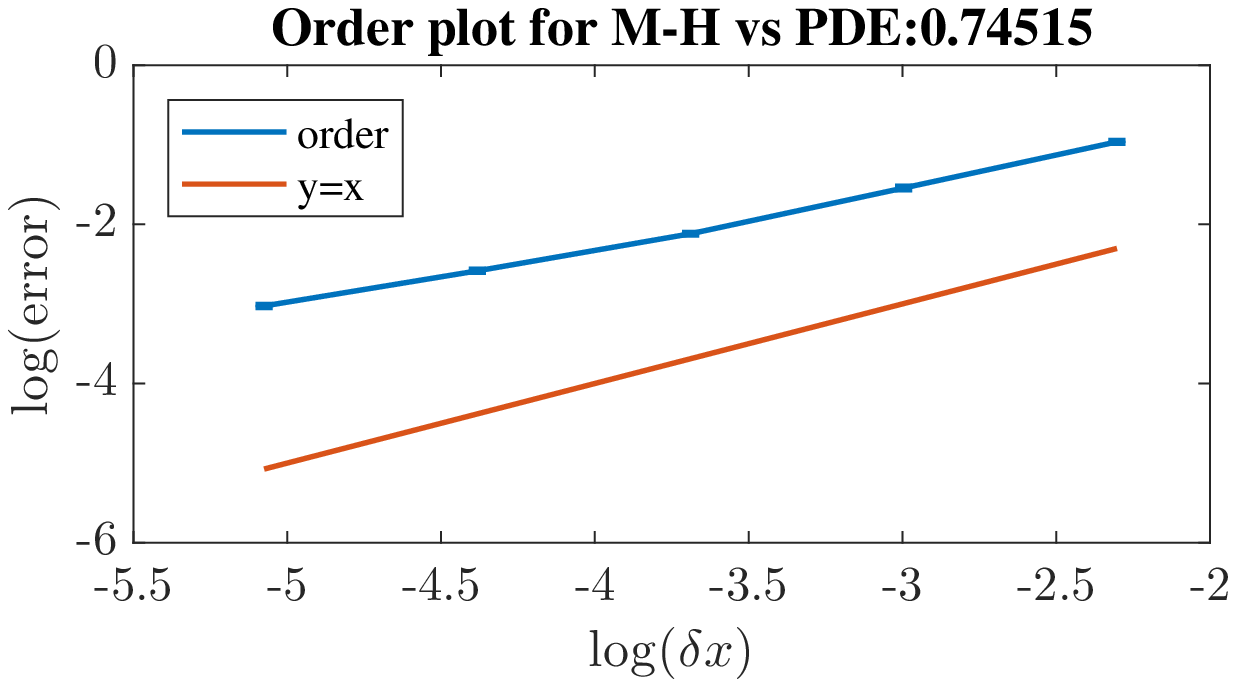}
		\caption{$\mathbb{T}^1 \to \mathbb{S}^2$}
		\label{fig:order3b}
	\end{subfigure}
	
	\caption{Order of convergence for the error between M-H algorithm and harmonic map heat flow equation with respect to lattice discretization size $\delta x = \frac{1}{N}$. From Figure \ref{fig:dynamics} the error between M-H algorithm and Langevin equation dominates over the error between Langevin equation and harmonic map heat flow equation. The order of convergence is expected to be $\delta t^{\frac 14} = \delta x$ since we choose $\delta t = \delta x^4$ and it is approximately the case in \ref{fig:order3a} and \ref{fig:order3b}. The analytical results in Theorems \ref{thm:1} and \ref{thm:2} does not give as good convergence rate and demands worse scaling of $\delta t,\beta$ as function of $N$. The following parameters tested in numerical experiments are enough: Lattice length $L=2$, initial space discretization $\delta x = \frac 1N = \frac{1}{10}$, time step size for the M-H algorithm $\delta t=\frac{1}{N^3}$, inverse temperature $\beta = 1$, proposal size $\varepsilon = \sqrt{\frac{N\delta t}{\beta}}$.}
	\label{fig:order3}
\end{figure}

Figure \ref{fig:order3} shows the convergence test for the error between the M-H and the PDE dynamics with respect to $\delta x = \frac 1N$.  The discrete version of the PDE is simulated with the time-step scaling of $\delta t = \frac{1}{N^4}$ and $\beta = N^{3/2}$, which are also used with $\varepsilon=\sqrt{\frac{N\delta t}{\beta}}$ in the M-H algorithm.  These scalings give 
the order of convergence to be approximately 1, better than our analytical result in Theorem \ref{thm:2}.  A possible explanation is that the error from Theorem \ref{thm:1} dominates.  As discussed above, Fig.~\ref{fig:dynamics} implies the error between the M-H algorithm and the Langevin SDE dominates over the error between the Langevin SDE and the harmonic map heat flow equation, thus we would expect error of $\delta t^{\frac 14}$ in Theorem \ref{thm:1} to dominate.  Since we choose the scaling of $\delta t = \frac{1}{N^4} = \delta x^4$ this order of convergence with respect to $\delta x = \frac 1N$ is expected to be $\delta t^{\frac 14} = \delta x$, or order one.  We also point out that the scalings of $\delta t = \frac{1}{N^4}$ and $\beta = N^{3/2}$ are better than the scalings one might guess from Theorem \ref{thm:1} ($\sqrt{\delta t}$ smaller than the order of $e^{-C_2}$ with $C_2$ an increasing function of $N$) and Remark \ref{remark:beta} ($\beta \gg N^7$). We suspect from the numerical experiments that the scalings of $\beta = N^{\frac 32}, \delta t= \frac{1}{N^4}$ are tight bounds resulting in order one convergence, but do not have a proof as of yet.

%
\section{Conclusion\label{sec:conclusions}}
%

We have shown that as the proposal size $\varepsilon\to 0$ in the Metropolis Hastings algorithm, the Metropolis dynamics converges to
the Langevin stochastic differential equation dynamics.  With proper scaling of $\beta = N^\gamma, \gamma>1$ and the number of particles $N \to \infty$, the SDE dynamics converges to the deterministic harmonic map heat flow dynamics.

Several future works are suggested by the results we have obtained.
First, the scaling confirmed by the numerical simulations suggest that even tighter bounds on the error can analytically be found.
 One thought to improve the scalings in the proofs is to try to divide the dynamics into two situations: near equilibrium and out of equilibrium. When the dynamics is out of equilibrium, the drift term $\textrm{P}_{\sigma_i}^\perp (\Delta_N \sigma_i)$ in the SDE \eqref{eq:SDE} dominates the behavior, driving down the energy.  Similarly, in the M-H dynamics, there is large probability of proposing a lower energy state, and therefore the proposal is often accepted.     In this sense, both dynamics are performing deterministic gradient descent.   On the other hand, when the dynamics are near equilibrium, the drift term $\textrm{P}_{\sigma_i}^\perp (\Delta_N \sigma_i)$ in SDE \eqref{eq:SDE}  is approximately zero and the system therefore fluctuates around equilibrium. The M-H dynamics with small proposal size would also stay in the neighborhood of the equilibrium state for a long time.  We hope by reconsidering the deterministic error out of equilibrium and a different approach exploiting in-equilibrium dynamics, a better scaling can be proven in future work.

Second, we would like to consider the Stochastic partial differential equation limit of the Metropolis-Hastings algorithm.  As in \cite{roberts1997weak,breyer2000metropolis,mattingly2012diffusion} this convergence result might imply the optimal scaling of proposal size $\varepsilon$ in the M-H algorithm.  We expect the correct scaling of the temperature is $\beta = N^\gamma$, $\gamma=0$, as formally, the noise term in the Euler step \eqref{eq:Euler step dt} of the SDE system \eqref{eq:Langevin} is of size $\sqrt{\delta t / \delta x}$ and converges to space-time white noise as $\delta t \to 0, N\to\infty$. 
  Note that the corresponding It\^o correction term in \eqref{eq:Euler step dt} would tend to infinity as $N\to\infty$.  This suggests seeking a regularization of the noise, such as using colored, spatially-correlated noise, particularly when considering more than one spatial dimension.  As in \cite{mattingly2012diffusion}, we suspect this addition of correlated noise to the M-H proposal will lead to a non-local drift term.  We plan to pursue these Stochastic limits in future work.

\appendix

\section{Drift and diffusion calculation} \label{Appendix:dirft and diffusion}

Here we state two simple inequalities that are used later. The first is
\[
\E[]{ |X+Y|^k } \le C_k ( \E[]{|X|^k} + \E[]{|Y|^k})
\]
for some constant $C_k$ as $ (X+Y)^k \le \mathbb{I}(|X| \le |Y|) 2^k |Y|^k + \mathbb{I}(|X| > |Y|) 2^k |X|^k$. Furthermore, this is also true for vectors
\[
\E[]{ \|X+Y \|^k } \le C_k ( \E[]{\|X\|^k} + \E[]{\|Y\|^k})
\]
since $ \E[]{ \|X+Y \|^k } = \E[]{ (\|X+Y\|^2)^{\frac k2}} \le \E[]{ (2\|X\|^2 + 2\|Y\|^2)^{\frac k2} } \le 2^{\frac k2} C_k ( \E[]{\|X\|^k} + \E{\|Y\|^k})$.

The second is H\"older inequality
\[
\E[]{XY} \le \E[]{|XY|} \le \E[]{|X|^p}^{\frac 1p} \E[]{|Y|^q}^{\frac 1q}
\]
with $ \frac 1p + \frac 1q =1$.

\subsection{Exponential map.}
In Section \ref{subsec:setup} we use the notation:
\[
\exp_{\sigma_i^n} (\varepsilon \nu_i^n) = \sigma_i^n + \varepsilon \nu_i^n + c_i^n = \sigma_i^n + \varepsilon \nu_i^n - \frac 12 \varepsilon^2 (\nu_i^n \cdot \nu_i^n) \sigma_i^n + d_i^n.
\]
Now we estimate $c_i^n,d_i^n$ as:
\[
\E[]{\|c_i^n\|^k} \le C_k \varepsilon^{2k}, \E[]{\|d_i^n\|^k} \le D_k \varepsilon^{3k}
\]
for any postive integer $k$ and some constants $C_k,D_k$ independent of $i,n$.

Notice
\begin{dmath*}
\exp_{\sigma_i^n} (\varepsilon \nu_i^n) = \frac{\sigma_i^n + \varepsilon \nu_i^n}{\| \sigma_i^n + \varepsilon \nu_i^n \|} +a_i^n = \frac{\sigma_i^n + \varepsilon \nu_i^n}{\sqrt{\| \sigma_i^n \|^2 + \varepsilon^2 \|\nu_i^n \|^2 - 2 \varepsilon (\sigma_i^n \cdot \nu_i^n) }} +a_i^n= \frac{\sigma_i^n + \varepsilon \nu_i^n}{ \sqrt{1 +\varepsilon^2 \|\nu_i^n \|^2  }} +a_i^n.
\end{dmath*}
Taylor expanding $\frac{1}{\sqrt{1 +\varepsilon^2 \|\nu_i^n \|^2  } }$, we have
\[
\exp_{\sigma_i^n} (\varepsilon \nu_i^n) = \left( \sigma_i^n + \varepsilon \nu_i^n \right) \left[  1 -
\frac{\varepsilon^2}{2} \| \nu_i^n \|^2 + \eta_i^n  \right] + a_i^n,
\]
where $\eta_i^n$ is the remainder of the Taylor expansion for $\frac{1}{\sqrt{1 +\varepsilon^2 \|\nu_i^n \|^2  } }$. Then,
\begin{gather}
d_i^n = -\frac{\varepsilon^3}{2} \|\nu_i^n\|^2 \nu_i^n + (\sigma_i^n + \varepsilon \nu_i^n) \eta_i^n  + a_i^n, \label{eq:d} \\
c_i^n = -\frac{\varepsilon^2}{2} \| \nu_i^n \|^2 \sigma_i^n + d_i^n + a_i^n.  \label{eq:c}
\end{gather}

Let us first deal with the term $a_i^n$, since the geodesic on the unit sphere is the great circle,
$\exp_{\sigma_i^n} (\varepsilon \nu_i^n)$, $\frac{\sigma_i^n + \varepsilon \nu_i^n}{\| \sigma_i^n + \varepsilon \nu_i^n \|} $ are on the same great circle. The arc length of geodesic is $\|\varepsilon \nu_i^n \|$, and the arc length of $\frac{\sigma_i^n + \varepsilon \nu_i^n}{\| \sigma_i^n + \varepsilon \nu_i^n \|}$ is $\arctan \left( \| \varepsilon \nu_i^n \|\right)$. The vector $a_i^n$ is the straight line connecting the two points of the difference between these two arcs, and is bounded by the difference of arc lengths:
\[
\| a_i^n \| \le \| \varepsilon \nu_i^n \| - \arctan \left( \| \varepsilon \nu_i^n \| \right).
\]

Taylor expanding for $\arctan x$,
\[
\arctan x = x - \frac{x^3}{3} + r, r = \frac{f^{(4)}(\xi)}{4!} x^4, f^{(4)}(x) = \frac{24x(1-x^2)}{(1+x^2)^4},
\]
and when $x = \| \varepsilon \nu_i^n \| \ge 0, |f^{(4)}(x)| \le | \frac{24x}{(1+x^2)} \frac{1-x^2}{1+x^2} \frac{1}{(1+x^2)^2}| \le 24$, $|r| \le x^4$. Hence, from $ x - \arctan x = \frac{x^3}{3} - r$,
\[
\| a_i^n \| \le \| \varepsilon \nu_i^n \| - \arctan \left( \| \varepsilon \nu_i^n \| \right)
\le \frac{\| \varepsilon \nu_i^n \|^3}{3} + \| \varepsilon \nu_i^n \|^4,
\]
so
\[
\E[]{ \| a_i^n \|^k } \le C \E{\| \varepsilon \nu_i^n \|^{3k} } \le C \varepsilon^{3k}
\]

In the Taylor expansion for $f(x) = \frac{1}{\sqrt{1+x}}, x \ge 0$, the remainder $ r = f(x) - (1 -\frac x2)$ is given by
\[
r = \frac{f''(\xi)}{2} x^2 = \frac 38  (1 + \xi)^{-\frac 52} x^2,  \xi \in [0,\infty)
\]
and $|r| \le \frac 38 x^2$. Applying the above estimates for $r=\eta_i^n$ with $x=\varepsilon^2 \|\nu_i^n \|^2 \ge 0$, we observe
\[
| \eta_i^n |  \le \frac 38 \varepsilon^4 \| \nu_i^n \|^4.
\]

Now we could get the bound for the terms $c_i^n,d_i^n$.
The first term for $d_i^n$ in \eqref{eq:d} is bounded by
\[
\E[]{\left\|\frac{\varepsilon^3}{2} \|\nu_i^n\|^2 \nu_i^n\right\|^k } \le C \E[]{\varepsilon^{3k} \|\nu_i^n\|^{3k} } \le C \varepsilon^{3k}.
\]
The second term in \eqref{eq:d} gives
\[
\E[]{\|(\sigma_i^n + \varepsilon \nu_i^n) \eta_i^n\|^k} \le
\E{ \| \sigma_i^n + \varepsilon \nu_i^n \|^{2k} }^{\frac 12} \E{|\eta_i^n|^{2k} }^{\frac 12} \le c \varepsilon^{4k}
\]
by H\"older's inequality. This is because the first term in the right hand side is bounded by
\[
\E[]{ \| \sigma_i^n + \varepsilon \nu_i^n \|^{2k} } \le C_1 \left( \E[]{\|\sigma_i^n\|^{2k}} + \varepsilon^{2k} \E[]{\| \nu_i^n \|^{2k} } \right) \le C_1 + C_2 \varepsilon^{2k}
\]
and the second term in the right hand side is bounded by
\[
\E[]{|\eta_i^n|^{2k} } \le C \E[]{\varepsilon^{8k} \|\nu_i^n\|^{8k} } \le C \varepsilon^{8k}.
\]

The first term for $c_i^n$ in \eqref{eq:c}
\[
\E[]{\left\|\frac{\varepsilon^2}{2} \| \nu_i^n \|^2 \sigma_i^n \right\|^k} \le C \varepsilon^{2k} \E[]{ \|\nu_i^n\|^{2k} \|\sigma_i^n\|^k  } \le C \varepsilon^{2k}.
\]

The bound for $\E[]{\|c_i^n\|^k},\E[]{\|d_i^n\|^k}$ are found using the inequalities
\[
\E[]{ \|X+Y \|^k } \le C_k ( \E[]{\|X\|^k} + \E[]{\|Y\|^k})
\]
with the above bounds for the terms in $c_i^n,d_i^n$.

\subsection{Drift.}

Now we give the error estimation for the drift calculation. Denote
\[
\theta_i^n \equiv \E{\sigma_i^{n+1} - \sigma_i^n } - \left(-\frac 12 \beta \varepsilon^2 \textrm{P}_{\sigma^n_i}^{\perp} \left( \frac{\partial H}{\partial \sigma^n_i} \right) - \varepsilon^2 \sigma^n_i \right).
\]
We will show
\[
\E[]{ \|\theta_i^n\|^2 } \le C \varepsilon^6.
\]

For simplicity we write $\theta_i^n = \sum_k \theta_k$, where $\theta_k$ denotes the error for each step of the drift calculation in Section \ref{subsec:drift}. And we will show each $\E[]{\|\theta_k \|^2} \le C \varepsilon^6$.

Notice $\theta_k$ are conditional expectations in the form of $\theta_k = \E{\bf{X}}$. Since $f(x)=x^2$ is a convex function, we have $(\E{\bf{X}})^2 \le \E{\|\bf{X}\|^2}$, hence
\[
 \E[]{\|\theta_k\|^2} = \E[]{(\E{\bf{X}})^2} \le \E[]{\E{\|\bf{X}\|^2}} = \E[]{\| \bf{X} \|^2}.
\]
We will use this to bound $\E[]{\|\theta_k\|^2}$.

In the drift calculation, we first take the approximation $\E{\sigma_i^{n+1} - \sigma_i^n } \approx \E{ \left( \varepsilon \nu_i^n - \frac{\varepsilon^2}{2} ( \nu_i^n \cdot \nu_i^n) \sigma_i^n  \right)  \left( 1 \wedge  e^{-\beta \delta H} \right) }$. Denote the difference of them as
\[
\theta_1 \equiv \E{\sigma_i^{n+1} - \sigma_i^n } - \E{ \left( \varepsilon \nu_i^n - \frac{\varepsilon^2}{2} ( \nu_i^n \cdot \nu_i^n) \sigma_i^n  \right)  \left( 1 \wedge  e^{-\beta \delta H} \right) } = \E{d_i^n \left( 1 \wedge  e^{-\beta \delta H} \right)},
\]
then by H\"older's inequality
\[
  \E[]{\|\theta_1 \|^2} \le \E[]{ \left\|d_i^n \left( 1 \wedge  e^{-\beta \delta H} \right) \right \|^2} \le C \E[]{\|d_i^n \|^4 }^{\frac 12}   \E[]{ \left( 1 \wedge  e^{-\beta \delta H} \right)^4}^{\frac 12} \le C \varepsilon^6.
\]

For the second term in drift $ \E{\left(- \frac{\varepsilon^2}{2} ( \nu_i^n \cdot \nu_i^n) \sigma_i^n  \right)  \left( 1 \wedge  e^{-\beta \delta H} \right) } \approx \E{ \left(- \frac{\varepsilon^2}{2}
( \nu_i^n \cdot \nu_i^n) \sigma_i^n  \right) }$, denote the error term as
\begin{dmath*}
\theta_2 \equiv \E{\left(- \frac{\varepsilon^2}{2}
( \nu_i^n \cdot \nu_i^n) \sigma_i^n  \right)  \left( 1 \wedge  e^{-\beta \delta H} \right) } - \E{\left(- \frac{\varepsilon^2}{2}
( \nu_i^n \cdot \nu_i^n) \sigma_i^n  \right)}
= \E{\left(- \frac{\varepsilon^2}{2}
( \nu_i^n \cdot \nu_i^n) \sigma_i^n  \right)  \left( 1 \wedge  e^{-\beta \delta H} - 1 \right) }.
\end{dmath*}
We have $ \left| \left( 1 \wedge e^{x} \right) - 1 \right| \le | x|$, and
$ | \delta H| \le C \sum_j \left( \| \varepsilon \nu_j^n \| + \| \varepsilon \nu_j^n \|^2 + \|c_j^n\| \right)$
so by H\"older inequality
\begin{dmath*}
\E[]{\| \theta_2 \|^2} \le \E[]{ \left\|- \frac{\varepsilon^2}{2} ( \nu_i^n \cdot \nu_i^n) \sigma_i^n  \right\|^4 }^{\frac 12} \E[]{\left( 1 \wedge  e^{-\beta \delta H} - 1 \right)^4 }^{\frac 12} \\ 
\le C \E[]{ \varepsilon^8 \|\nu_i^n\|^8 \| \sigma_i^n\|^4 }^{\frac 12} \E[]{ |\beta \delta H|^4 }^{\frac 12} \le C \beta^2 \varepsilon^6.
\end{dmath*}

Next we replace $\delta H$ by $\varepsilon \frac{\partial H}{\partial \sigma_i^n} \cdot \nu_i^n + R_i^n$ in the drift calculation. Define
\[
\theta_3 \equiv \E{\varepsilon \nu_i^n \left( 1 \wedge e^{-\beta \delta H} \right)} - \E{\varepsilon \nu_i^n \left( 1 \wedge e^{-\beta \left( \varepsilon \frac{\partial H}{\partial \sigma_i^n} \cdot \nu_i^n + R_i^n \right)} \right)}.
\]
Notice that $\delta H = \varepsilon \frac{\partial H}{\partial \sigma_i^n} \cdot \nu_i^n + R_i^n + h_i^n $ and $| (1\wedge e^{x+\delta x}) - (1 \wedge
x)| \le |\delta x|$, hence
\[
\E[]{\|\theta_3\|^2} \le \E[]{\|\varepsilon \nu_i^n\|^4 }^{\frac 12} \E[]{ (\beta h_i^n)^4 }^{\frac 12}.
\]
Since
$|h_i^n| \le C \sum_j \left( \| \varepsilon \nu_j^n \|^2 + \|c_j^n\| \right)$, we have $\E[]{(h_i^n)^4} \le C \varepsilon^8$ and
\[
\E[]{\|\theta_3\|^2} \le \E[]{\|\varepsilon \nu_i^n\|^4 }^{\frac 12} \E[]{ (h_i^n)^4 }^{\frac 12} \le C \beta^2 \varepsilon^6.
\]

Then we write the drift term as in \eqref{eq:drift1}, and have the first term given by
\begin{align*}
& \E[]{\varepsilon r_1 b_1 \left(1\wedge e^{-\beta \left( \varepsilon r_1 \frac{\partial H}{\partial \sigma_i^n} \cdot b_1 + \varepsilon r_2 \frac{\partial H}{\partial \sigma_i^n} \cdot b_2 + R_i^n \right) }\right) \Big | R_i^n, r_2 }\\
 &=   -\beta \varepsilon^2 \left( \frac{\partial H}{\partial \sigma_i^n} \cdot b_1 \right) b_1 e^{\frac{\left( \beta \varepsilon \frac{\partial H}{\partial \sigma_i^n} \cdot b_1 \right)^2}{2} + \beta \varepsilon r_2 \frac{\partial H}{\partial \sigma_i^n} \cdot b_2 + \beta R_i^n } \Phi \left( \frac{\varepsilon r_2 \frac{\partial H}{\partial \sigma_i^n} \cdot b_2 + R_i^n}  {\left| \varepsilon \frac{\partial H}{\partial \sigma_i^n} \cdot b_1 \right| } - \left| \beta \varepsilon \frac{\partial H}{\partial \sigma_i^n} \cdot b_1 \right| \right).
\end{align*}
We approximate $e^{\frac{\left( \beta \varepsilon \frac{\partial H}{\partial \sigma_i^n} \cdot b_1 \right)^2}{2} + \beta \varepsilon r_2 \frac{\partial H}{\partial \sigma_i^n} \cdot b_2 + \beta R_i^n }$ by $1$ in the above equation and bound the following term
\begin{align*}
\theta_4 
\equiv & \mathbb{E} \left[ -\beta \varepsilon^2 \left( \frac{\partial H}{\partial \sigma_i^n} \cdot b_1 \right) b_1
e^{\frac{\left( \beta \varepsilon \frac{\partial H}{\partial \sigma_i^n} \cdot b_1 \right)^2}{2} +\beta \varepsilon r_2 \frac{\partial H}{\partial \sigma_i^n} \cdot b_2 + \beta R_i^n }  \right. \\
& \qquad \left. \times \Phi \left( \frac{\varepsilon r_2 \frac{\partial H}{\partial \sigma_i^n} \cdot b_2 + R_i^n} {\left| \varepsilon \frac{\partial H}{\partial \sigma_i^n} \cdot b_1 \right| } - \left| \beta \varepsilon \frac{\partial H}{\partial \sigma_i^n} \cdot b_1 \right| \right) \right]  \\
& \quad - \E{-\beta \varepsilon^2 \left( \frac{\partial H}{\partial \sigma_i^n} \cdot b_1 \right) b_1 \Phi \left( \frac{\varepsilon r_2 \frac{\partial H}{\partial \sigma_i^n} \cdot b_2 + R_i^n} {\left| \varepsilon \frac{\partial H}{\partial \sigma_i^n} \cdot b_1 \right| } - \left| \beta \varepsilon \frac{\partial H}{\partial \sigma_i^n} \cdot b_1 \right| \right)}
\end{align*}
by 
\[ \E{\| \theta_4 \|^2} \le C \varepsilon^6. \]

For $ z \sim \N(\mu,\sigma^2)$, we use
\begin{equation} \label{eq:ezk}
\E[]{\left| e^z -1 \right|^k} = \E[]{\left| e^z -1 \right|^k \mathbbm{1}(z \le 2) }+
\E[]{\left| e^z -1 \right|^k \mathbbm{1}(z > 2)}.
\end{equation}
For the first term in \eqref{eq:ezk}, since $ |e^z -1| \le e^2 z$ for $z \le 2$, we have 
\[ \E[]{\left| e^z -1 \right|^k \mathbbm{1}(z \le 2)| } \le e^{2k} \E[]{ z^k}. \] When $k=4$, $\E[]{ z^k}= \mu^4 + 6 \mu^2 \sigma^2 + 3\sigma^4$.

For the second term in \eqref{eq:ezk}, $\left| e^z -1 \right|^k \le e^{kz}$ when $z >2$, suppose $2+\mu+k \sigma^2 \ge 1$ we have
\begin{align*}
&\E{e^{kz} \mathbbm{1}(z>2)} 
= \int_2^\infty \frac{1}{\sqrt{2\pi \sigma^2}} e^{kz} e^{ -\frac{(z-\mu)^2}{2\sigma^2} } \df z \\
&\hspace{.5cm} = \int_2^\infty \frac{1}{\sqrt{2\pi \sigma^2}} e^{k\mu + \frac{k^2 \sigma^2}{2}} e^{ -\frac{(z-\mu-k\sigma^2)^2}{2\sigma^2}} \df z 
= e^{k\mu + \frac{k^2 \sigma^2}{2}} \int_{2+\mu+k \sigma^2}^{\infty} \frac{1}{\sqrt{2\pi \sigma^2}}  e^{-\frac{x^2}{2\sigma^2}} \df x \\
& \hspace{.5cm}\le  e^{k\mu + \frac{k^2 \sigma^2}{2}} \int_{2+\mu+k \sigma^2}^{\infty} \frac{1}{\sqrt{2\pi \sigma^2}}  e^{-\frac{x}{2\sigma^2}} \df x 
=  e^{k\mu + \frac{k^2 \sigma^2}{2}} \sqrt{\frac{2}{\pi}} \sigma e^{-\frac{2 + \mu + k \sigma^2}{2\sigma^2}} \\
&  \hspace{.5cm}\le  e^{k\mu + \frac{k^2 \sigma^2}{2}} \sqrt{\frac{2}{\pi}} \sigma e^{-\frac{1}{2\sigma^2}}.
\end{align*}

Notice that $\frac{\left( \beta \varepsilon \frac{\partial H}{\partial \sigma_i^n} \cdot b_1 \right)^2}{2} +
\beta \varepsilon r_2 \frac{\partial H}{\partial \sigma_i^n} \cdot b_2 + \beta R_i^n \sim \N(c_1 \varepsilon^2,c_2 \varepsilon^2)$, when $\varepsilon$ is small $2+\mu+k \sigma^2 \ge 1$, hence we get
\[
\E[]{ \left( e^{ \frac{\left( \beta \varepsilon \frac{\partial H}{\partial \sigma_i^n} \cdot b_1 \right)^2}{2} +
\beta \varepsilon r_2 \frac{\partial H}{\partial \sigma_i^n} \cdot b_2 + \beta R_i^n} -1 \right)^4 } \le
e^8 c_2^2 \varepsilon^4 + \sqrt{\frac{2 c_2}{\pi}} \varepsilon e^{-\frac{1}{2c_2 \varepsilon^2}} \le C \varepsilon^4
\]
as the term $e^{-\frac{1}{2c_2 \varepsilon^2}}$ decays faster than any polynomial of $\varepsilon$ as $\varepsilon \to 0$.

So we have
\begin{align*}
  \E[]{ \| \theta_4 \|^2} 
  =&   \E[]{ \left\| -\beta \varepsilon^2 \left( \frac{\partial H}{\partial \sigma_i^n} \cdot b_1 \right) b_1 \Phi \left( \frac{\varepsilon r_2 \frac{\partial H}{\partial \sigma_i^n} \cdot b_2 + R_i^n} {\left| \varepsilon \frac{\partial H}{\partial \sigma_i^n} \cdot b_1 \right| } - \left| \beta \varepsilon \frac{\partial H}{\partial \sigma_i^n} \cdot b_1 \right| \right) \right\|^4}^{\frac 12} \\
   & \quad \times \E[]{ \left| e^{ \frac{\left( \beta \varepsilon \frac{\partial H}{\partial \sigma_i^n} \cdot b_1 \right)^2}{2} + \beta \varepsilon r_2 \frac{\partial H}{\partial \sigma_i^n} \cdot b_2 + \beta R_i^n} -1 \right|^4 }^{\frac 12} \\
   \le &  C \varepsilon^4 \varepsilon^2.
\end{align*}

Then we approximate $ \Phi \left( \frac{\varepsilon r_2 \frac{\partial H}{\partial \sigma_i^n} \cdot b_2 + R_i^n}
{\left| \varepsilon \frac{\partial H}{\partial \sigma_i^n} \cdot b_1 \right| }
- \left| \beta \varepsilon \frac{\partial H}{\partial \sigma_i^n} \cdot b_1 \right|\right)$ by
$\Phi \left( \frac{\varepsilon r_2 \frac{\partial H}{\partial \sigma_i^n} \cdot b_2 + R_i^n}
{\left| \varepsilon \frac{\partial H}{\partial \sigma_i^n} \cdot b_1 \right| }  \right)$ in the calculation. Denote
\begin{gather*}
\theta_5 \equiv \E[]{-\beta \varepsilon^2 \left( \frac{\partial H}{\partial \sigma_i^n} \cdot b_1 \right) b_1 \Phi \left( \frac{\varepsilon r_2 \frac{\partial H}{\partial \sigma_i^n} \cdot b_2 + R_i^n} {\left| \varepsilon \frac{\partial H}{\partial \sigma_i^n} \cdot b_1 \right| } - \left| \beta \varepsilon \frac{\partial H}{\partial \sigma_i^n} \cdot b_1 \right| \right)} \\
- \E{-\beta \varepsilon^2 \left( \frac{\partial H}{\partial \sigma_i^n} \cdot b_1 \right) b_1
\Phi \left( \frac{\varepsilon r_2 \frac{\partial H}{\partial \sigma_i^n} \cdot b_2 + R_i^n}
{\left| \varepsilon \frac{\partial H}{\partial \sigma_i^n} \cdot b_1 \right| }  \right)}.
\end{gather*}
From $ \left| \Phi(x+\delta x) - \Phi(x) \right| = \left| \Phi'(\xi) \delta x \right| = \left| e^{-\frac{(\xi-\mu)^2}{2\sigma^2}} \delta x \right| \le |\delta x| $, we have
\begin{align*}
& \E[]{ \theta_5^2} 
= \E[]{ \left\|-\beta \varepsilon^2 \left( \frac{\partial H}{\partial \sigma_i^n} \cdot b_1 \right) b_1  \right\|^4 }^{\frac 12} \\
  & \quad \times \E[]{ \left| \Phi \left( \frac{\varepsilon r_2 \frac{\partial H}{\partial \sigma_i^n} \cdot b_2 + R_i^n} {\left| \varepsilon \frac{\partial H}{\partial \sigma_i^n} \cdot b_1 \right| } - \left| \beta \varepsilon \frac{\partial H}{\partial \sigma_i^n} \cdot b_1 \right| \right)  - \Phi \left( \frac{\varepsilon r_2 \frac{\partial H}{\partial \sigma_i^n} \cdot b_2 + R_i^n} {\left| \varepsilon \frac{\partial H}{\partial \sigma_i^n} \cdot b_1 \right| }  \right) \right|}^{\frac 12} \\
& \hspace{1cm} \le  C \varepsilon^4 \varepsilon^2.
\end{align*}

Similarly we would get the bound for
\[
\E[]{\varepsilon  r_2 b_2 \left(1\wedge e^{-\beta \left( \varepsilon r_1 \frac{\partial H}{\partial \sigma_i^n} \cdot b_1 + \varepsilon r_2 \frac{\partial H}{\partial \sigma_i^n} \cdot b_2 + R_i^n \right) }\right) }
\]
and see every $\E[]{\| \theta_i \|^2}$ is bounded by $C \varepsilon^6$ so $ \E[]{\|\theta_i^n\|^2} \le C \varepsilon^6$.

\subsection{Diffusion.} \label{Appendix:diffusion}
Here we are show that $\E[]{\|\phi_i^n\|^2} = O(\varepsilon^3)$. In
\begin{dmath*}
  \E[]{ \| \phi_i^n \|^2}  = \E[]{\left\| c_i^n - \E{\sigma_i^{n+1}- \sigma_i^n} \right\|^2 \left( 1 \wedge e^{-\beta \delta H}\right) } +  \E[]{ \left\|-\varepsilon \nu_i^n -\E{\sigma_i^{n+1} - \sigma_i^n}\right\|^2 \left(1- \left( 1 \wedge e^{-\beta \delta H}\right)\right) }
\end{dmath*}
every term is order $\varepsilon^4$ except for $\E[]{ \varepsilon^2 \|\nu_i^n\|^2 \left(1- \left( 1 \wedge e^{-\beta \delta H}\right)\right) }$. We show
\[ \E[]{ \|\nu_i^n\|^2 \left( 1 \wedge e^{-\beta \delta H}\right)} = \E[]{ \|\nu_i^n\|^2 } + O(\varepsilon). \]

Indeed, since $ | (1 \wedge e^x) - (1 \wedge (1+x)) | \le x^2 $,
\begin{align*}
 & \E[]{ \| \nu_i^n \|^2  \left( 1 \wedge e^{-\beta \left(  \varepsilon \frac{\partial H}{\partial \sigma_i^n} \cdot \nu_i^n + R_i^n \right)} \right)} \\
& \hspace{.5cm} \approx  \E[]{ \| \nu_i^n \|^2  \left( 1 \wedge \left(1-\beta \left(  \varepsilon \frac{\partial H}{\partial \sigma_i^n} \cdot \nu_i^n + R_i^n \right) \right) \right)} + O(\varepsilon^2) \\
 &  \hspace{.5cm}=  \E[]{ \|\nu_i^n\|^2 \I{ \beta \left(  \varepsilon \frac{\partial H}{\partial \sigma_i^n} \cdot \nu_i^n + R_i^n \right) <0} } \\
 & \hspace{1.1cm} + \E[]{ \|\nu_i^n\|^2 \left( 1 - \beta \left(  \varepsilon \frac{\partial H}{\partial \sigma_i^n} \cdot \nu_i^n + R_i^n \right) \right)  \I{ \beta \left(  \varepsilon \frac{\partial H}{\partial \sigma_i^n} \cdot \nu_i^n + R_i^n \right) >0} } \\
& \hspace{0.5cm} =  \E[]{\|\nu_i^n\|^2} - \E[]{ \|\nu_i^n\|^2 \left( \beta \left(  \varepsilon \frac{\partial H}{\partial \sigma_i^n} \cdot \nu_i^n + R_i^n \right) \right) \I{ \beta \left(  \varepsilon \frac{\partial H}{\partial \sigma_i^n} \cdot \nu_i^n + R_i^n \right) >0} }.
\end{align*}

Then it remains to show $\E[]{ \|\nu_i^n\|^2 \left( \beta \left(  \varepsilon \frac{\partial H}{\partial \sigma_i^n} \cdot \nu_i^n + R_i^n \right) \right) \I{ \beta \left(  \varepsilon \frac{\partial H}{\partial \sigma_i^n} \cdot \nu_i^n + R_i^n \right) >0} }$ is approximately
an $O(\varepsilon)$ term. As before take $\nu_i^n = r_1 b_1 + r_2 b_2$ and we will only take care of $r_1$, since the calculation for $r_2$ is similar. Denote $R \equiv R_i^n + \varepsilon \frac{\partial H}{\partial \sigma_i^n} \cdot b_2 r_2 \sim N(0,\varepsilon^2 c_1^2), c_1 = O(1)$. The expectation of the $r_1$ part is
\[
 -\beta \E[]{ r_1^2 \left( R + \varepsilon \frac{\partial H}{\partial \sigma_i^n} \cdot b_1 r_1\right)  \I{ \beta \left(  \varepsilon \frac{\partial H}{\partial \sigma_i^n} \cdot b_1 r_1 + R \right) >0}}.
\]
To calculate, first condition on $R$ and compute it over $r_1$. This involves expectations in forms of 
\[
\E[]{ z^2 \I{(az+b)>0)}}, \text{and} \ \E[]{z^3 \I{(az+b)>0} }.
\]  
For $z \sim N(0,1)$, a direct calculation gives
\begin{equation} \label{eq:Ez2}
\begin{aligned}
 \E[]{z^2 \I{(az+b)>0} } &= \frac{-b}{\sqrt{2\pi} |a|} e^{-\frac{b^2}{2a^2}} + \Phi \left( \frac{b}{|a|} \right), \\
 \E[]{z^3 \I{(az+b)>0}} &= \frac{1}{\sqrt{2\pi}} \left( 2 + \frac{b^2}{a^2} \right) e^{-\frac{b^2}{2a^2}} \textrm{sign}(a).
\end{aligned}
\end{equation}
Using the tower property, we have
\begin{dmath} \label{eq:Er1}
 \E[]{ r_1^2 \left( R + \varepsilon \frac{\partial H}{\partial \sigma_i^n} \cdot b_1 r_1\right)  \I{ \beta \left(  \varepsilon \frac{\partial H}{\partial \sigma_i^n} \cdot b_1 r_1 + R \right) >0}} = \E[]{\E{ \left. r_1^2 R \I{ \beta \left(  \varepsilon \frac{\partial H}{\partial \sigma_i^n} \cdot b_1 r_1 + R \right) >0} \right| R}}  + \E[]{\E{ \left. \varepsilon \frac{\partial H}{\partial \sigma_i^n} \cdot b_1 (r_1)^3 \I{ \beta \left(  \varepsilon \frac{\partial H}{\partial \sigma_i^n} \cdot b_1 r_1 + R \right) >0} \right| R}}.
\end{dmath}

Denote $c_2 =\frac{\partial H}{\partial \sigma_i^n} \cdot b_1$, for the first term in \eqref{eq:Er1}, the first formula in \eqref{eq:Ez2} gives
\begin{equation} \label{eq:Er2}
\E{ \left. r_1^2 R \I{ \beta \left(  \varepsilon \frac{\partial H}{\partial \sigma_i^n} \cdot b_1 r_1 + R \right) >0} \right| R} = \E{-\frac{R^2}{\sqrt{2\pi} |\varepsilon c_2|} e^{-\frac{R^2}{2 \varepsilon^2 c_2^2}} + R \Phi \left( \frac{R}{|\varepsilon c_2|} \right) }.
\end{equation}
For the second term in \eqref{eq:Er1}, the second formula in \eqref{eq:Ez2} gives
\begin{dmath} \label{eq:Er3}
 \E{ \left. \varepsilon \frac{\partial H}{\partial \sigma_i^n} \cdot b_1 (r_1)^3 \I{ \beta \left(  \varepsilon \frac{\partial H}{\partial \sigma_i^n} \cdot b_1 r_1 + R \right) >0} \right| R} =
 \E{ \varepsilon c_2 \frac{1}{\sqrt{2\pi}} \left( 2 + \frac{R^2}{\varepsilon^2 c_2^2} \right) e^{-\frac{R^2}{2\varepsilon^2 c_2^2}} \textrm{sign}(c_2) }
 = \E{\frac{2 \varepsilon |c_2|}{\sqrt{2\pi}} e^{-\frac{R^2}{2\varepsilon^2 c_2^2}} + \frac{R^2}{ \sqrt{2\pi} \varepsilon |c_2|} e^{-\frac{R^2}{2\varepsilon^2 c_2^2}} }.
\end{dmath}

Combining \eqref{eq:Er2} and \eqref{eq:Er3}, we have
\begin{dmath*}
\E{ \left. r_1^2 \left( R + \varepsilon \frac{\partial H}{\partial \sigma_i^n} \cdot b_1 r_1\right) \I{ \beta \left(  \varepsilon \frac{\partial H}{\partial \sigma_i^n} \cdot b_1 r_1 + R \right) >0} \right| R} = \E{ R \Phi \left( \frac{R}{|\varepsilon c_2|} \right)} + \E{\frac{2 \varepsilon |c_2|}{\sqrt{2\pi}} e^{-\frac{R^2}{2\varepsilon^2 c_2^2}} }.
\end{dmath*}
Since $R \left( \Phi \left( \frac{R}{|\varepsilon c_2|} \right) - \frac 12 \right) \ge 0$, the first term $\E{ R \Phi \left( \frac{R}{|\varepsilon c_2|} \right)} = \E{ R \left( \Phi \left( \frac{R}{|\varepsilon c_2|} \right) - \frac 12 \right) + \frac 12 R} \ge 0$. In the second term $\frac{R}{\varepsilon c_2} \sim N(0,\frac{c_1^2}{c_2^2})$ and 
\[ \E{e^{-\frac{R^2}{2\varepsilon^2 c_2^2}} } = \frac{c_2}{\sqrt{c_1^2+c_2^2}} \sim O(1) \]
after a direct calculation. This shows 
\[ -\beta \E{ r_1^2 \left( R + \varepsilon \frac{\partial H}{\partial \sigma_i^n} \cdot b_1 r_1\right) \I{ \beta \left(  \varepsilon \frac{\partial H}{\partial \sigma_i^n} \cdot b_1 r_1 + R \right) >0}} \sim O(\varepsilon) \]
and the $r_2$ part follows similarly.

Thus we conclude $\E[]{ \|\nu_i^n\|^2 \left( \beta \left(  \varepsilon \frac{\partial H}{\partial \sigma_i^n} \cdot \nu_i^n + R_i^n \right) \right) \I{ \beta \left(  \varepsilon \frac{\partial H}{\partial \sigma_i^n} \cdot \nu_i^n + R_i^n \right) >0} }$ is an $\varepsilon$ term.

\section{Quadratic variation}

For an n dimensional process $X$
\[
\df X= \mu \df t + \sigma \df W.
\]
It\^o's chain rule for a function $f(X)$ is
\begin{equation} \label{6}
  \df f(X_1, X_2, \ldots, X_n) = \sum_i \frac{\partial f}{\partial X_i} \df X_i + \frac 12 \sum_i \sum_j \frac{\partial^2 f}{\partial X_i \partial X_j}  \df [X_i,X_j](t),
\end{equation}
where
\[
\df [X_i,X_j] (t) = a_{ij} \df t, \quad a = \sigma \sigma^T.
\]

For $H = J \sum_{<i,j>} \| \sigma_i - \sigma_j \|^2 $, denote $i,j$ as $i,j$-th spin and $\alpha,\beta=x,y,z$ for the coordinates of the spin
\[
\df H = \sum_{i,\alpha} \frac{\partial H}{\partial \sigma_i^\alpha} \df \sigma_i^\alpha + \frac 12 \sum_{i,j,\alpha,\beta} \frac{\partial^2 H}{\partial \sigma_i^\alpha \partial \sigma_j^\beta} \df [ \sigma_i^\alpha, \sigma_j^\beta ]
\]
where
\[
\frac{\partial H}{\partial \sigma_i^x} = - \frac{2J}{N^2} \triangle_N \sigma_i^x, \quad
\frac{\partial^2 H}{\partial (\sigma_i^x)^2} = \begin{cases}
  4J & j=i \\
  -2J & j = i \pm 1 \\
  0 & \textrm{otherwise}
\end{cases}
\]
and $ \frac{\partial^2 H}{\partial \sigma_i^x \partial \sigma_i^y} = 0$. The results for $y,z$ are similar.

The noise term in SDE is $\textrm{P}_{\sigma_i}^{\perp}(\df W_i) = \df W_i - (\df W_i,\sigma_i) \sigma_i$. For $i \ne j$, $W_i,W_j$ are independent and then $\df [\sigma_i^\alpha,\sigma_j^\beta]=0$, so only $\df [\sigma_i^x,\sigma_i^x]$ need to be calculated. Since $W_i^x,W_i^y,W_i^z$ are also independent, the quadratic variation is calculated by summing up the coefficients before each $\df W_i^\alpha$
\begin{equation*}
  \begin{aligned}
    \df [\sigma_i^x,\sigma_i^x] = \epsilon^2 \left[ \left(1 - (\sigma_i^x)^2 \right)^2 + (\sigma_i^x \sigma_i^y)^2 + (\sigma_i^x \sigma_i^z)^2 \right] \df t = \epsilon^2 \left[ 1 - (\sigma_i^x)^2 \right] \df t, \\
    \df [\sigma_i^y,\sigma_i^y] = \epsilon^2 \left[ \left(1 - (\sigma_i^y)^2 \right)^2 + (\sigma_i^x \sigma_i^y)^2  + (\sigma_i^y \sigma_i^z)^2 \right] \df t= \epsilon^2 \left[ 1 - (\sigma_i^y)^2 \right] \df t, \\
    \df [\sigma_i^z,\sigma_i^z] = \epsilon^2 \left[ \left(1 - (\sigma_i^z)^2 \right)^2 + (\sigma_i^x \sigma_i^z)^2 + (\sigma_i^y \sigma_i^z)^2 \right] \df t= \epsilon^2 \left[ 1 - (\sigma_i^z)^2 \right] \df t.
  \end{aligned}
\end{equation*}

As $(\sigma_i^x)^2 +(\sigma_i^y)^2+(\sigma_i^z)^2=1$, the It\^o correction is
\[ \frac 12 4J \epsilon^2 \sum_i [ 1 + 1 +1 - (\sigma_i^x)^2 -(\sigma_i^y)^2 -(\sigma_i^y)^2 ] = \frac 12 \sum 4J * 2 \epsilon^2 = 4JN\epsilon^2\]
and from \eqref{eq:dH}
\begin{dmath} 
  \df H = \sum -\frac{2J}{N^2} \triangle_N \sigma_i \cdot \left[ \left(
  \frac{JN}{N^2} \textrm{P}_{\sigma_i}^{\perp}
  \left(\Delta_N \sigma_i(t) \right) dt - \epsilon^2 \sigma_i(t) dt +   \textrm{P}_{\sigma_i}^{\perp}
  \left( \epsilon \df W_i(t) \right)\right) \right] \\
  + 4J N \epsilon^2.
\end{dmath}

Denote $M_t$ as the martingale given by
\begin{equation} \label{eq:dM} 
  \df M_t = \sum_i - \frac{2J}{N^2} \triangle_N \sigma_i \cdot  \textrm{P}_{\sigma_i}^{\perp}
  \left( \epsilon \df W_i(t) \right)
\end{equation}
and $\langle M \rangle_t$ the corresponding quadratic variation. The quantity $\df \langle M \rangle_t$ is calculated by summing the square of the coefficients of $\df W_i$ in $x,y,z$ components:
\begin{align*}
   \df \langle M \rangle_t 
 =&  \df \left \langle \sum_i - \frac{2J}{N^2} \triangle_N \sigma_i \cdot \textrm{P}_{\sigma_i}^{\perp} \left( \epsilon \df W_i(t) \right) \right \rangle_t \\  
  =& \df \left \langle \sum_i - \frac{2J \epsilon}{N^2} \triangle_N \sigma_i \cdot \left[  \df W_i -
  (\df W_i \cdot \sigma_i) \sigma_i \right] \right \rangle_t  \\
  =&  \df \left \langle  \sum_i \left[ \frac{-2J\epsilon}{N^2} \sum_{\alpha = x,y,z} \left(  \triangle_N \sigma_i \cdot \df W_i - (\triangle_N \sigma_i \cdot \sigma_i) \sigma_i \df W_i \right)_{\alpha} \right] \right \rangle_t  \\
  =& \sum_i \frac{4J^2 \epsilon^2}{N^4} \left[ \sum_{\alpha = x,y,z} \left( \triangle_N \sigma_i - (\triangle_N \sigma_i \cdot \sigma_i)  \sigma_i \right)_{\alpha}^2 \right] \df t \\
  =& \sum_i \frac{4J^2 \epsilon^2}{N^4} \left \| \textrm{P}_{\sigma_i}^{\perp} (\triangle_N \sigma_i)  \right \|^2 \df t,
\end{align*}
for the last step $\textrm{P}_{\sigma_i}^{\perp} (\triangle_N \sigma_i) = \triangle_N \sigma_i - (\triangle_N \sigma_i \cdot \sigma_i)
\sigma_i$.

Then the inequality for continuous $L^2$ martingale $M_t$
\begin{equation} \label{eq:Ineq} 
  \mathbb{P} ( \sup ( M_t - \alpha/2 \langle M \rangle_t) > \beta) \le e^{- \alpha \beta}
\end{equation}
is used to get a bound on $H$.

Notice in \eqref{eq:dH}, we have
\[
\triangle_N \sigma_i \cdot  \left( \textrm{P}_{\sigma_i}^{\perp} \left(\Delta_N \sigma_i \right) \right) = \left \|  \textrm{P}_{\sigma_i}^{\perp} \left(\Delta_N \sigma_i \right)   \right \|^2.
\]
We can write it in the form of $\df \langle M \rangle_t$ by observing
\[
\df H = \df M_t - \frac{N}{2 \epsilon^2} \df \langle M \rangle_t + \left[ \sum \epsilon^2 \frac{2J}{N^2} \triangle_N \sigma_i \cdot \sigma_i + 4JN \epsilon^2 \right] \df t
\]
and the last term in the bracket
\[
\sum \epsilon^2 \frac{2J}{N^2} \triangle_N \sigma_i \cdot \sigma_i + 4JN \epsilon^2 = \sum 2J \epsilon^2 ( \sigma_{i+1} + \sigma_{i-1} ) \cdot \sigma_i \le 4JN \epsilon^2.
\]
So
\begin{align*}
  H_N (t) &= H_N(0) + M_t - \frac{N}{2 \epsilon^2} \langle M \rangle_t + \int \left[ \sum \epsilon^2 \frac{2J}{N^2} \triangle_N \sigma_i  \cdot \sigma_i + 4JN \epsilon^2 \right] \df s \\
  &\le H_N(0) + M_t - \frac{N}{2\epsilon^2} \langle M \rangle_t + 4J\epsilon^2 N t.
\end{align*}

Take $\alpha = \frac{N}{\epsilon^2}$ in the inequality $\mathbb{P} ( \sup ( M_t - \alpha/2 \langle M \rangle_t) > \beta) \le e^{- \alpha \beta}$,
\begin{equation}
  \mathbb{P} \left( \sup_{t \le T} H_N \ge H_N(0) + 4J N \epsilon^2T +\beta \right) \le
  \mathbb{P} \left( \sup_{t \le T} M_t - \frac{N}{2\epsilon^2} \langle M \rangle_t \ge \beta \right) \le  e^{-\frac{\beta N}{\epsilon^2}}.
\end{equation}

\section{Diffusion on sphere} \label{Appendix:diffusionsph}

We will use Fokker-Planck equation to show the Stratonovich SDE
\begin{equation} \label{eq:rw}
  \df \bx = \textrm{P}_{\boldsymbol{x}}^\perp ( \boldsymbol{\df W})
\end{equation}
in $\mathbb{R}^2$ and $\mathbb{R}^3$ are describing Brownian motion on the unit circle and unit sphere. And in $\mathbb{R}^3$ it is regardless of the choice for $\textrm{P}_{\boldsymbol{x}}^\perp (\boldsymbol{y}) =  \boldsymbol{x} \times \boldsymbol{y}$ or $\textrm{P}_{\boldsymbol{x}}^\perp (\boldsymbol{y}) = - \boldsymbol{x} \times ( \boldsymbol{x} \times \boldsymbol{y} ) = I - \boldsymbol{x} \bx^T$.

\subsection{Circle \texorpdfstring{$\mathbb{S}^1$}{}.} 

For a Stratonovich SDE with the form
 \[
 \df X_i = b_i \df t + \sum_j \sigma_{ij} \df W_j,
 \]
 the corresponding It\^o drift coefficient is
 \begin{equation}
 \tilde{b}_{i} = b_i + \frac 12 \sum_j \sum_k \frac{\partial \sigma_{ij}}{\partial x_k} \sigma_{kj}.
 \end{equation}

On the circle $\textrm{P}_{\boldsymbol{x}}^\perp ( \bw) = (I - \bx \bx^T) \boldsymbol{\df W} = \boldsymbol{\df W} - (\bw, \bx) \bx $.  The corresponding It\^o form for \eqref{eq:rw} is
\begin{equation}
\df \bx = - \frac 12 \bx + (I - \bx \bx^T) \boldsymbol{\df W}.
\end{equation}

For It\^o SDE $\df \bx = \mu(\bx,t) \df t+ \sigma(\bx,t) \bw$, the  Fokker-Planck equation is
\begin{equation}
\frac{\partial \rho}{\partial t} = - \sum_i \frac{\partial}{\partial x_i} (\mu_i \rho) + \frac 12 \sum_i \sum_j \frac{\partial^2}{\partial x_i \partial x_j} (D_{ij} \rho),
\end{equation}
with diffusion tensor
\[
D_{ij} = \sum_{k} \sigma_{ik} \sigma{jk} = (\sigma \sigma^T)_{ij}.
\]

The  Fokker-Planck equation for $\df \bx = - \frac 12 \bx + (I - \bx \bx^T) \boldsymbol{\df W}$ is
\begin{equation}
\partial_t \rho =  \frac 12 \left[ y^2 \partial_x^2 \rho + x^2 \partial_y^2 \rho  - 2xy \partial_x \partial_y \rho - x \partial_x \rho - y \partial_y \rho \right].
\end{equation}

The Laplacian on the circle in polar coordinate is $\partial_t \rho = \partial_{\theta \theta} \rho$. Use transformation $x = r \cos \theta, y = r \sin \theta$,
\[
  \partial_{\theta \theta} = ( - y \partial_x + x \partial_y) ( - y \partial_x + x \partial_y) = y^2 \partial_x^2 + x^2 \partial_y^2 - 2xy \partial_x \partial_y - x \partial_x - y \partial_y
\]
corresponding to the Fokker-Planck equation above.

\subsection{Sphere \texorpdfstring{$\mathbb{S}^2$}{}.}

In the sphere case, the projection can take the following two forms
\[
\textrm{P}_{\boldsymbol{x}}^\perp (\boldsymbol{y}) =
	\begin{cases}
	   \boldsymbol{x} \times \boldsymbol{y} \\
	   - \boldsymbol{x} \times ( \boldsymbol{x} \times \boldsymbol{y} ) = (I - \boldsymbol{x} \bx^T) \boldsymbol{y}
	\end{cases}.
\]
In both cases the It\^o correction are the same as $- \bx$. The It\^o form for \eqref{eq:rw} is
\begin{equation}
  \df \bx = - \bx \df t + \textrm{P}_{\bx}^\perp ( \df W).
\end{equation}

In the Fokker-Planck equation calculation, for both projections the diffusion tensor are the same
\[
D = \begin{pmatrix}
     y^2+z^2 & -xy & -xz \\
     -xy & x^2+z^2 & -yz \\
     -xz & -yz & x^2+y^2
    \end{pmatrix}
\]
using the fact $x^2+y^2+z^2=1$. The Fokker-Planck equation is
\begin{align*}
  & \partial_t \rho  = \frac 12 \left[ (y^2+z^2) \rho_{xx} + (x^2+z^2) \rho_{yy} + (x^2+y^2) \rho_{zz} - \partial_x ( xy \rho_y + xz \rho_z) \right.  \\
  & \hspace{2cm}  \left. - \partial_y ( xy \rho_x + yz \rho_z) - \partial_z ( xz \rho_x + yz \rho_y) \right].
\end{align*}

The Laplacian on $\mathbb{S}^2$ in polar coordinate is
\[
\frac{1}{\sin \theta} \partial_\theta ( \sin \theta \partial_\theta) + \frac{1}{\sin^2 \theta} \partial^2_{\phi}.
\]
Using the change of coordinate
\[
\begin{cases}
 x = r \sin \theta \cos \phi \\
 y = r \sin \theta \sin \phi \\
 z = r \cos \theta
\end{cases},
\]
we have
\begin{align*}
  \partial_\theta &= \frac{xz}{\sqrt{1-z^2}} \partial_x + \frac{yz}{\sqrt{1-z^2}} \partial_y - \sqrt{1-z^2} \partial_z \\
  \partial^2_\theta &= \frac{x^2 z^2}{1-z^2} \partial^2_x  + \frac{y^2 z^2}{1-z^2} \partial^2_y + 2 \frac{xyz^2}{1-z^2} \partial^2_{xy} -2xz \partial^2_{xz} - 2yz \partial^2_{yz} - x \partial_x - y \partial_y \\
  & \qquad + (1-z^2) \partial^2_z - z \partial_z \\
  \partial^2_\phi &= ( - y \partial_x + x \partial_y) ( - y \partial_x + x \partial_y) = y^2 \partial_x^2 + x^2 \partial_y^2 - 2xy \partial_x \partial_y - x \partial_x - y \partial_y.
\end{align*}
As $x^2+y^2+z^2=1, \sin^2 \theta = x^2 +y^2, \frac{\cos \theta}{\sin \theta} = \frac{z}{\sqrt{x^2+y^2}}$, the equation
\[
\partial_t \rho = \frac{1}{\sin \theta} \partial_\theta  ( \sin \theta \partial_\theta \rho ) + \frac{1}{\sin^2 \theta} \partial^2_{\phi} \rho
\]
is corresponding to the Fokker-Planck equation above.


\begin{thebibliography}{10}

\bibitem{alouges2014semi}
Fran{\c{c}}ois Alouges, Anne De~Bouard, and Antoine Hocquet.
\newblock A semi-discrete scheme for the stochastic landau--lifshitz equation.
\newblock {\em Stochastic Partial Differential Equations: Analysis and
  Computations}, 2(3):281--315, 2014.

\bibitem{banas2014stochastic}
Lubomir Banas, Zdzislaw Brzezniak, Mikhail Neklyudov, and Andreas Prohl.
\newblock {\em Stochastic ferromagnetism: analysis and numerics}, volume~58.
\newblock Walter de Gruyter, 2014.

\bibitem{bass1998diffusions}
Richard~F Bass.
\newblock {\em Diffusions and elliptic operators}.
\newblock Springer Science \& Business Media, 1998.

\bibitem{batrouni2004metastable}
GG~Batrouni.
\newblock Metastable states in the planar two-dimensional xy model and
  dissipation in superfluid flow.
\newblock {\em Physical Review B}, 70(18):184517, 2004.

\bibitem{binder1993monte}
Kurt Binder, Dieter Heermann, Lyle Roelofs, A~John Mallinckrodt, and Susan
  McKay.
\newblock Monte carlo simulation in statistical physics.
\newblock {\em Computers in Physics}, 7(2):156--157, 1993.

\bibitem{breyer2000metropolis}
Laird~A Breyer and Gareth~O Roberts.
\newblock From {Metropolis} to diffusions: {Gibbs} states and optimal scaling.
\newblock {\em Stochastic Processes and their Applications}, 90(2):181--206,
  2000.

\bibitem{chen1993static}
Kun Chen, Alan~M Ferrenberg, and DP~Landau.
\newblock Static critical behavior of three-dimensional classical {Heisenberg}
  models: A high-resolution {Monte Carlo} study.
\newblock {\em Physical Review B}, 48(5):3249, 1993.

\bibitem{chen1989weak}
Yunmei Chen.
\newblock The weak solutions to the evolution problems of harmonic maps.
\newblock {\em Mathematische Zeitschrift}, 201(1):69--74, 1989.

\bibitem{chen1989existence}
Yunmei Chen and Michael Struwe.
\newblock Existence and partial regularity results for the heat flow for
  harmonic maps.
\newblock {\em Mathematische Zeitschrift}, 201(1):83--103, 1989.

\bibitem{chugreeva2018strong}
Olga Chugreeva and Christof Melcher.
\newblock Strong solvability of regularized stochastic
  landau--lifshitz--gilbert equation.
\newblock {\em IMA Journal of Applied Mathematics}, 83(2):261--282, 2018.

\bibitem{eells1964harmonic}
James Eells and Joseph~H Sampson.
\newblock Harmonic mappings of riemannian manifolds.
\newblock {\em American journal of mathematics}, 86(1):109--160, 1964.

\bibitem{eley2012approaching}
Serena Eley, Sarang Gopalakrishnan, Paul~M Goldbart, and Nadya Mason.
\newblock Approaching zero-temperature metallic states in mesoscopic
  superconductor--normal--superconductor arrays.
\newblock {\em Nature Physics}, 8(1):59, 2012.

\bibitem{funaki1997motion}
Tadahisa Funaki and Herbert Spohn.
\newblock Motion by mean curvature from the {Ginzburg-Landau} interface model.
\newblock {\em Communications in Mathematical Physics}, 185(1):1--36, 1997.

\bibitem{guo2008landau}
Boling Guo and Shijin Ding.
\newblock {\em {Landau-Lifshitz} Equations}, volume~1.
\newblock World Scientific, 2008.

\bibitem{guo1993landau}
Boling Guo and Min-Chun Hong.
\newblock The {Landau-Lifshitz} equation of the ferromagnetic spin chain and
  harmonic maps.
\newblock {\em Calculus of Variations and Partial Differential Equations},
  1(3):311--334, 1993.

\bibitem{guo2014stochastic}
Hongxin Guo, Robert Philipowski, and Anton Thalmaier.
\newblock A stochastic approach to the harmonic map heat flow on manifolds with
  time-dependent riemannian metric.
\newblock {\em Stochastic Processes and their Applications},
  124(11):3535--3552, 2014.

\bibitem{guo1988nonlinear}
MZ~Guo, GC~Papanicolaou, and SRS Varadhan.
\newblock Nonlinear diffusion limit for a system with nearest neighbor
  interactions.
\newblock {\em Communications in Mathematical Physics}, 118(1):31--59, 1988.

\bibitem{hairer2008spectral}
Martin Hairer and Jonathan~C Mattingly.
\newblock Spectral gaps in wasserstein distances and the 2d stochastic
  navier-stokes equations.
\newblock {\em The Annals of Probability}, pages 2050--2091, 2008.

\bibitem{hastings1970monte}
W~Keith Hastings.
\newblock {Monte Carlo} sampling methods using {Markov} chains and their
  applications.
\newblock {\em Biometrika}, 57(1):97--109, 1970.

\bibitem{hocquet2018struwe}
Antoine Hocquet.
\newblock Struwe-like solutions for the stochastic harmonic map flow.
\newblock {\em Journal of Evolution Equations}, 18(3):1189--1228, 2018.

\bibitem{hocquet2019finite}
Antoine Hocquet et~al.
\newblock Finite-time singularity of the stochastic harmonic map flow.
\newblock In {\em Annales de l'Institut Henri Poincar{\'e}, Probabilit{\'e}s et
  Statistiques}, volume~55, pages 1011--1041. Institut Henri Poincar{\'e},
  2019.

\bibitem{jourdain2014optimal}
Benjamin Jourdain, Tony Leli{\`e}vre, B{\l}a{\.z}ej Miasojedow, et~al.
\newblock Optimal scaling for the transient phase of {M}etropolis {H}astings
  algorithms: the longtime behavior.
\newblock {\em Bernoulli}, 20(4):1930--1978, 2014.

\bibitem{jourdain2015optimal}
Benjamin Jourdain, Tony Leli{\`e}vre, B{\l}a{\.z}ej Miasojedow, et~al.
\newblock Optimal scaling for the transient phase of the random walk
  {M}etropolis algorithm: the mean-field limit.
\newblock {\em The Annals of Applied Probability}, 25(4):2263--2300, 2015.

\bibitem{karatzas2012brownian}
Ioannis Karatzas and Steven Shreve.
\newblock {\em Brownian motion and stochastic calculus}, volume 113.
\newblock Springer Science \& Business Media, 2012.

\bibitem{kirkpatrick2013asymptotics}
Kay Kirkpatrick and Elizabeth Meckes.
\newblock Asymptotics of the mean-field heisenberg model.
\newblock {\em Journal of Statistical Physics}, 152(1):54--92, 2013.

\bibitem{kirkpatrick2016asymptotics}
Kay Kirkpatrick and Tayyab Nawaz.
\newblock Asymptotics of mean-field o (n) models.
\newblock {\em Journal of Statistical Physics}, 165(6):1114--1140, 2016.

\bibitem{kohn2005magnetic}
Robert~V Kohn, Maria~G Reznikoff, and Eric Vanden-Eijnden.
\newblock Magnetic elements at finite temperature and large deviation theory.
\newblock {\em Journal of nonlinear science}, 15(4):223--253, 2005.

\bibitem{kosterlitz1973ordering}
John~Michael Kosterlitz and David~James Thouless.
\newblock Ordering, metastability and phase transitions in two-dimensional
  systems.
\newblock {\em Journal of Physics C: Solid State Physics}, 6(7):1181, 1973.

\bibitem{landau2014guide}
David~P Landau and Kurt Binder.
\newblock {\em A guide to Monte Carlo simulations in statistical physics}.
\newblock Cambridge university press, 2014.

\bibitem{lin1999harmonic}
Fang-Hua Lin and Chang-You Wang.
\newblock Harmonic and quasi-harmonic spheres.
\newblock {\em Communications in Analysis and Geometry}, 7(2):397--429, 1999.

\bibitem{lin2002harmonic}
Fang-Hua Lin and Chang-You Wang.
\newblock Harmonic and quasi-harmonic spheres, part ii.
\newblock {\em Communications in Analysis and Geometry}, 10(2):341--375, 2002.

\bibitem{lin1998energy}
Fanghua Lin and Changyou Wang.
\newblock Energy identity of harmonic map flows from surfaces at finite
  singular time.
\newblock {\em Calculus of Variations and Partial Differential Equations},
  6(4):369--380, 1998.

\bibitem{lin2008analysis}
Fanghua Lin and Changyou Wang.
\newblock {\em The analysis of harmonic maps and their heat flows}.
\newblock World Scientific, 2008.

\bibitem{maccari2016numerical}
Ilaria Maccari, Andrea Maiorano, Enzo Marinari, and Juan~Jesus Ruiz-Lorenzo.
\newblock A numerical study of planar arrays of correlated spin islands.
\newblock {\em The European Physical Journal B}, 89(5):127, 2016.

\bibitem{mattingly2012diffusion}
Jonathan~C Mattingly, Natesh~S Pillai, Andrew~M Stuart, et~al.
\newblock Diffusion limits of the random walk {Metropolis} algorithm in high
  dimensions.
\newblock {\em The Annals of Applied Probability}, 22(3):881--930, 2012.

\bibitem{newman1999monte}
MEJ Newman and GT~Barkema.
\newblock {\em Monte Carlo Methods in Statistical Physics chapter 1-4}.
\newblock Oxford University Press: New York, USA, 1999.

\bibitem{oksendal2003stochastic}
Bernt {\O}ksendal.
\newblock Stochastic differential equations.
\newblock In {\em Stochastic differential equations}, pages 70--74. Springer,
  2003.

\bibitem{peczak1991high}
P~Peczak, Alan~M Ferrenberg, and DP~Landau.
\newblock High-accuracy {Monte Carlo} study of the three-dimensional classical
  {Heisenberg} ferromagnet.
\newblock {\em Physical Review B}, 43(7):6087, 1991.

\bibitem{roberts1997weak}
Gareth~O Roberts, Andrew Gelman, Walter~R Gilks, et~al.
\newblock Weak convergence and optimal scaling of random walk {Metropolis}
  algorithms.
\newblock {\em The annals of applied probability}, 7(1):110--120, 1997.

\bibitem{stanley1968dependence}
H~Eugene Stanley.
\newblock Dependence of critical properties on dimensionality of spins.
\newblock {\em Physical Review Letters}, 20(12):589, 1968.

\bibitem{struwe1985evolution}
Michael Struwe.
\newblock On the evolution of harmonic mappings of riemannian surfaces.
\newblock {\em Commentarii Mathematici Helvetici}, 60(1):558--581, 1985.

\bibitem{topping2002reverse}
Peter Topping.
\newblock Reverse bubbling and nonuniqueness in the harmonic map flow.
\newblock {\em International Mathematics Research Notices}, 2002(10):505--520,
  2002.

\bibitem{topping2004repulsion}
Peter Topping.
\newblock Repulsion and quantization in almost-harmonic maps, and asymptotics
  of the harmonic map flow.
\newblock {\em Annals of mathematics}, pages 465--534, 2004.

\bibitem{topping1997rigidity}
Peter~Miles Topping et~al.
\newblock Rigidity in the harmonic map heat flow.
\newblock {\em J. Differential Geom}, 45(3):593--610, 1997.

\bibitem{weinan2001numerical}
E~Weinan and Xiao-Ping Wang.
\newblock Numerical methods for the {Landau-Lifshitz} equation.
\newblock {\em SIAM Journal on Numerical Analysis}, pages 1647--1665, 2001.

\bibitem{yau1991relative}
Horng-Tzer Yau.
\newblock Relative entropy and hydrodynamics of {Ginzburg-Landau} models.
\newblock {\em Letters in Mathematical Physics}, 22(1):63--80, 1991.

\end{thebibliography}
\end{document}